\documentclass[10pt,reqno]{amsart}
\usepackage{amssymb}
\usepackage{amsfonts}
\usepackage{amsthm}
\usepackage{amsmath}
\usepackage{hyperref} 

\newcommand{\R}{{\mathbb{R}}}

\newcommand{\Z}{{\mathbb{Z}}}

\newcommand{\scriptS}{{\mathcal{S}}}

\DeclareMathOperator*{\dist}{dist}

\DeclareMathOperator*{\opt}{opt}
\DeclareMathOperator*{\wklim}{wk\,lim}

\newcommand{\eps}{\varepsilon}

\DeclareMathOperator*{\supp}{supp}
\newcommand{\loc}{\textrm{loc}}

\newtheorem{theorem}{Theorem}[section]
\newtheorem{proposition}[theorem]{Proposition}
\newtheorem{prop}[theorem]{Proposition}
\newtheorem{lemma}[theorem]{Lemma}
\newtheorem{corollary}[theorem]{Corollary}

\theoremstyle{definition}
\newtheorem{definition}[theorem]{Definition}

\theoremstyle{remark}
\newtheorem*{remark}{Remark}

\def\tempsepline{{}-\!\!\!-\!\!\!-\!\!\!-\!\!\!-\!\!\!-\!\!\!-}
\def\tempsep{\smallskip\centerline{$\tempsepline\!\!\!\!\hbox to 0mm{\hss//\hss}\!\tempsepline$}\smallskip}

\newcommand{\jp}[1]{\langle{#1}\rangle}

\newcommand{\jpn}{\langle\nabla\rangle}

\newcommand{\qtq}[1]{\quad\text{#1}\quad}

\newcommand{\mce}{\mathfrak{e}}
\newcommand{\mcd}{\mathfrak{d}}
\newcommand{\mcl}{\mathfrak{l}}

\newcommand{\mck}{\mathfrak{k}}
\newcommand{\mcq}{\mathfrak{q}}
\newcommand{\mcz}{\mathfrak{z}}

\newcommand{\nabslash}{{\raisebox{0.3ex}{\slash}\mkern-11mu\nabla}}

\numberwithin{equation}{section}
\allowdisplaybreaks

\begin{document}

\title{Blowup behaviour for the nonlinear Klein--Gordon equation}
\author{Rowan Killip}
\address{University of California, Los Angeles}
\author{Betsy Stovall}
\address{University of California, Los Angeles}
\author{Monica Visan}
\address{University of California, Los Angeles}

\begin{abstract}
We analyze the blowup behaviour of solutions to the focusing nonlinear Klein--Gordon equation in spatial dimensions $d\geq 2$.
We obtain upper bounds on the blowup rate, both globally in space and in light cones.  The results are sharp in the conformal and sub-conformal cases.  The argument relies on Lyapunov functionals derived from the dilation identity.  We also prove that the critical Sobolev norm diverges near the blowup time.
\end{abstract}

\maketitle

\tableofcontents
%
%
%
%

\section{Introduction.}

We consider the initial-value problem for the nonlinear Klein--Gordon equation
\begin{equation} \label{E:eqn}
\begin{cases}
u_{tt} - \Delta u + m^2 u = |u|^p u\\
u(0) = u_0,\ u_t(0) = u_1
\end{cases}
\end{equation}
in spatial dimensions $d\geq 2$ with $0<p<\frac4{d-2}$ and $m\in[0,1]$.  Note that when $m=0$, this reduces to the nonlinear wave equation.
We will only consider real-valued solutions to \eqref{E:eqn}; the methods adapt easily to the complex-valued case.

This equation is the natural Hamiltonian flow associated with the energy 
\begin{equation} \label{E:energy}
E_m(u) = \int_{\R^d} \tfrac12|\nabla_{t,x}u(t,x)|^2 + \tfrac{m^2}2 |u(t,x)|^2 - \tfrac1{p+2}|u(t,x)|^{p+2}\, dx.
\end{equation}

Both the linear and nonlinear Klein--Gordon equations enjoy finite speed of propagation (indeed, they are fully Poincar\'e invariant).
For this reason, many statements (including the definition of a solution) are most naturally formulated in light cones.

\begin{definition}[Light cones] \label{D:light cones}
Given $(t_0,x_0) \in \R_+ \times \R^d$, we write
$$
\Gamma(t_0,x_0) := \{(t,x) \in [0,t_0) \times \R^d : |x - x_0| \leq t_0-t\}
$$
to denote the backwards light cone emanating from this point.  Given $u : \Gamma(t_0,x_0) \to \R$, we write
$ u \in C_{t,\loc} L^2_x(\Gamma(t_0,x_0))$ if, when we extend $u$ to be zero outside of $\Gamma(t_0,x_0)$,
the function $t \mapsto u(t)$ is continuous in $L^2_x(\R^d)$ on compact subintervals of $[0,t_0)$.  We define
$$
C_{t,\loc}H^1_x(\Gamma(t_0,x_0)):=\{ u\in C_{t,\rm{loc}} L^2_x(\Gamma(t_0,x_0)) \text{ and } \nabla u \in C_{t,\rm{loc}} L^2_x(\Gamma(t_0,x_0))\}.
$$
Finally, we write $u \in L^{\frac{p(d+1)}2}_{\rm{loc}}(\Gamma(t_0,x_0))$ if $u \in L^{\frac{p(d+1)}2}_{t,x}[(K \times \R^d)\cap \Gamma(t_0,x_0)]$ for every compact time interval $K \subset [0,t_0)$.
\end{definition}

The dispersion relation for the linear Klein--Gordon equation is $\omega^2 = m^2+|\xi|^2$.  In view of this, we adopt the notation
\begin{equation}\label{E:jp def}
   \jp{\xi}_m := \sqrt{m^2+|\xi|^2},
\end{equation}
by analogy with the widely-used $\jp{\xi} := \sqrt{1+|\xi|^2}$.  With this notation, the solution of the linear Klein--Gordon equation
with $u(0)=u_0$ and $u_t(0)=u_1$ is given by
$$
\scriptS_m(t)(u_0,u_1) = \cos(t\jpn_m)u_0 + \jpn_m^{-1}\sin(t\jpn_m)u_1.
$$

\begin{definition}[Solution]  \label{D:solution}
Let $(t_0,x_0) \in \R_+ \times \R^d$. A function $u:\Gamma(t_0,x_0) \to \R$ is a \emph{(strong) solution} to \eqref{E:eqn} if
$(u,u_t) \in C_{t,\loc}[H^1_x \times L^2_x](\Gamma(t_0,x_0))$, $u \in L^{\frac{p(d+1)}2}_\loc(\Gamma(t_0,x_0))$, and $u$ satisfies the Duhamel formula
\begin{equation} \label{E:Duhamel}
u(t) = \scriptS_m(t)(u_0,u_1) + \int_0^t \jpn_m^{-1}\sin(\jpn_m(t-s))|u(s)|^pu(s)\, ds
\end{equation}
on $\Gamma(t_0,x_0)$.
\end{definition}

Strong solutions are known to be unique (cf. Proposition~\ref{P:uniqueness}) and so any initial data in $(u_0,u_1)\in H^1_x\times L^2_x$ leads to a unique \emph{maximally extended} solution defined on the union of all light cones upon which a strong solution exists.  This region of spacetime is called the \emph{domain of maximal (forward) extension}. When this is not $[0, \infty)\times\R^d$, it must take the form
$$
\{ (t,x) : x\in \R^d \text{ and } 0 \leq t < \sigma(x) \}
$$
where $\sigma:\R^d\to(0,\infty)$ is a 1-Lipschitz function.  The surface
$$
\Sigma = \{(\sigma(x),x): x \in \R^d\}
$$
is called the \emph{(forward) blowup surface}.  A point $t_0=\sigma(x_0)$ on the blowup surface is called \emph{non-characteristic} if
$\sigma(x)\geq \sigma(x_0) - (1-\eps)|x-x_0|$ for all $x$ and some $\eps>0$.  Otherwise the point is called \emph{characteristic}.

With these preliminaries out of the way, let us now describe both the principal results and the structure of the paper.

In Section~\ref{S:local theory}, we review the local well-posedness theory for our equation.  Almost nothing in this section is new.
However we include full proofs, both for the sake of completeness and because in several places the exact formulation we require is
not that which appears in the literature. Additionally, local well-posedness results are intrinsically lower bounds on the rate of blowup.
In this way, the results of Section~\ref{S:local theory} provide a counterpart to the upper bounds proved elsewhere in the paper.

Sections~\ref{S:non+blow},~\ref{S:CC}, and~\ref{S:critical blowup} culminate in a proof that the critical Sobolev norm diverges as the blowup time is approached, at least along a subsequence of times.  Here criticality is defined with respect to scaling.  The nonlinear Klein--Gordon equation does not have a scaling symmetry, except when $m=0$.  In the massless case the scaling symmetry takes the form $u(t,x)\mapsto u^\lambda(t,x) := \lambda^{2/p} u(\lambda t,\lambda x)$ and the corresponding scale-invariant spaces are $u\in \dot H^{s_c}(\R^d)$ and $u_t \in \dot H^{s_c-1}(\R^d)$ with $s_c := \frac{d}2 - \frac2p$.  As blowup is naturally associated with short length scales (i.e., $\lambda\to\infty$) and the coefficient of the mass term shrinks to zero under this scaling, it is natural to regard $s_c$ as the critical regularity for \eqref{E:eqn} even when $m > 0$.

\begin{theorem}\label{T:I:sc} Consider initial data $u_0\in H^1(\R^d)$ and $u_1\in L^2(\R^d)$ with $d \geq 2$ and suppose $p = \frac4{d-2s_c}$ with $\frac1{2d} < s_c < 1$. If the maximal-lifespan solution $u$ to \eqref{E:eqn} blows up forward in time at $0<T_* < \infty$, then
$$
\smash{\limsup_{t \uparrow T_*}} \bigl\{\|u(t)\|_{\dot H^{s_c}_x} + \|u_t(t)\|_{H^{s_c-1}_x}\bigr\} = \infty.
$$
When $s_c < \frac12$ we additionally assume that $u_0$ and $u_1$ are spherically symmetric.
\end{theorem}

By virtue of scale invariance, the blowup time can be adjusted arbitrarily without altering the size of the critical norm.  As this indicates, the link between blowup and the critical norm is subtle.  We note also the example of the mass-critical nonlinear Schr\"odinger equation for which blowup does occur despite the fact that the $L^2_x$-norm is a constant of motion!

We were prompted to investigate the behaviour of the critical norm by a recent paper of Merle and Raphael, \cite{MerleRaphaelAJM}, who considered the nonlinear Schr\"odinger equation with radial data and $0<s_c<1$.  They showed that the critical norm must blow up as a power of $|\log(T^*-t)|$.

In \cite{MerleRaphaelAJM} a rescaling argument is used to show that if a blowup solution were to exist for which the critical norm did not diverge, then one could
produce a second solution that is global in at least one time direction and has energy $E(u)\leq 0$.  The impossibility of this second type of solution is then deduced
via the virial argument.  Because the second solution has poor spatial decay, the virial argument needs to be space localized and the resulting error terms controlled;
this relies heavily on the radial hypothesis.  Note that the roles of the symmetry assumption in \cite{MerleRaphaelAJM} and here are of a completely different character.  As discussed in Section~\ref{S:local theory}, our equation is \emph{ill-posed} in $H_x^{s_c}\times H_x^{s_c-1}$ when $s_c<\frac12$, unless one imposes the restriction to radial data.  The additional restriction to $s_c>\frac1{2d}$ in Theorem~\ref{T:I:sc} stems from the fact that we do not know if the equation is well-posed in $ H_x^{s_c}\times H_x^{s_c-1}$ for spherically symmetric data; see Section~\ref{S:local theory} for more details.

To prove Theorem~\ref{T:I:sc} we argue in a broadly similar manner, showing that failure of the theorem would result in a semi-global solution to the limiting
(massless) equation with energy $E(u)\leq 0$ and then arguing that such a solution cannot exist.  In our setting we are able to handle arbitrary (nonradial) solutions when $1/2\leq s_c<1$ by employing a concentration-compactness principle for an inequality of Gagliardo--Nirenberg type.  The requisite concentration-compactness result is obtained in Section~\ref{S:CC}.  The impossibility of semi-global solutions with energy $E(u)\leq 0$ to the massless equation is proved in Section~\ref{S:non+blow}; this relies on a space-truncated virial argument.

In Section~\ref{S:other} we examine how the $L^2(\R^d)$ norms of $u$ and $\nabla_{t,x} u$ behave near the blowup time.  The arguments are comparatively straightforward applications of the virial argument; no spatial truncation is required.

In the remaining three sections of the paper, we study the behaviour of $u$ and $\nabla_{t,x} u$ near individual points on the blowup surface, rather than integrated
over all space as in Section~\ref{S:other}.  This is a much more delicate matter.  For the case of the nonlinear wave equation, this has been treated in a series
of papers by Antonini, Merle, and Zaag; see \cite{AntoniniMerle,MerleZaagAJM,MerleZaagMA,MerleZaagIMRN}.  All of these papers restrict attention only to cases where $s_c\leq \frac12$.  In this paper we will extend their results to the Klein--Gordon setting, considering also the regime $\frac12 < s_c < 1$.

The analysis of the nonlinear wave equation relies centrally on certain monotonicity formulae.  In the papers mentioned above, these appear via rather \emph{ad hoc}
manipulations mimicking earlier work of Giga and Kohn on the nonlinear heat equation \cite{GigaKohn:Indiana,GigaKohn:CPAM}.  In Section~\ref{S:Lyapunov} we uncover
the physical origins of these identities, finding that they are in fact close cousins of the dilation identity.  This in turn indicates the proper analogues in the
Klein--Gordon setting.  The identities are then used in Sections~\ref{S:superconf blow} and~\ref{S:conf blow} to control the behaviour of solutions inside light cones.
Section~\ref{S:superconf blow} treats the case $s_c>\frac12$ while Section~\ref{S:conf blow} covers $s_c\leq \frac12$.  The following theorem captures the flavour of
our results in the two cases:

\begin{theorem}\label{T:I:cone}
Let $d \geq 2$, $m \in [0,1]$, and $p = \frac4{d-2s_c}$ with $0< s_c < 1$.  If $u$ is a strong solution to \eqref{E:eqn} in the light cone $\{(t,x):0 < t \leq T, |x|<t\}$, then $u$ satisfies
\begin{equation}\label{E:I:u}
\int_{|x|< t/2} |u(t,x)|^2\, dx \lesssim  \begin{cases} t^{\frac{pd}{p+4}}  & :\text{ if } s_c > \tfrac 12 \\
t^{2s_c} &:\text{ if }  s_c \leq \tfrac 12 \end{cases}
\end{equation}
and
\begin{equation}\label{E:I:grad u}
\int_{t_0}^{2t_0} \!\!\! \int_{|x|<t/2} |\nabla_{\!t,x}u(t,x)|^2 \, dx\, dt \lesssim \begin{cases} 1  & :\text{ if } s_c > \tfrac 12 \\
t_0^{2s_c-1} &:\text{ if }  s_c \leq \tfrac 12. \end{cases}
\end{equation}
\end{theorem}

Note that the powers appearing in the two cases in RHS\eqref{E:I:u} and RHS\eqref{E:I:grad u} agree when $s_c = \frac12$.  Note also that this theorem
is best understood by considering the time-reversed evolution, that is, for initial data given at time $T$ and with $(0,0)$ being a point on the backwards blowup surface.

The local well-posedness results in Section~\ref{S:local theory} (cf. Corollary~\ref{C:lwp loc}) show that these upper bounds on the blowup rate are sharp when $0<s_c\leq\frac12$.

One peculiarity of the case $s_c=\frac12$ is that the massless equation is invariant under the full conformal group of Minkowski spacetime.  For this reason we term this the \emph{conformal} case.  Correspondingly, $s_c<\frac12$ and $s_c>\frac12$ will be referred to as the
\emph{sub-} and \emph{super-conformal} cases, respectively.  In the conformal case, the Lagrangian action is invariant under scaling and so the dilation identity
takes the form of a true conservation law (cf. \eqref{E:m dilation}), while at other regularities it does not.

The key dichotomy between $s_c \leq \frac12$ and $s_c>\frac12$ in the context of Theorem~\ref{T:I:cone} is not dictated directly by conformality, but rather
by the scaling of the basic monotonicity formulae we use.  The dilation identity scales as $s_c=\frac12$.  As a consequence, we are able to obtain stronger
results in the conformal and sub-conformal cases than in the super-conformal regime.  Indeed, in these cases \eqref{E:I:grad u} can be upgraded to a \emph{pointwise
in time} statement; see Theorem~\ref{T:cone bound subc}.  Systematic consideration of all conformal conservation laws (cf. Section~\ref{S:Lyapunov}) does not lead to any monotonicity formulae scaling at a higher regularity, thereby suggesting that the dilation is still the best tool for the job when $s_c>\frac12$.

The simplified version of our estimates given in Theorem~\ref{T:I:cone} only controls the size of the solution in the middle portion of the light cone, $\{|x| < t/2\}$.  In truth, the estimates we prove give weighted bounds in the whole light cone; however, the weight decays rather quickly near the boundary of the light cone.  If the point $(0,0)$ is not a characteristic point of the blowup surface, then simple covering arguments using nearby light cones show that the same estimates hold for the whole region $\{|x| <t\}$.  In fact, when $s_c\leq \frac12$ our results precisely coincide with those proved by Merle and Zaag for the corresponding nonlinear wave equation
in \cite{MerleZaagAJM,MerleZaagMA,MerleZaagIMRN}.  (As mentioned previously, their works do not consider the case $s_c >\frac12$.)

It turns out that it is possible to repeat the Merle--Zaag arguments virtually verbatim in the Klein--Gordon setting (with $s_c\leq \frac12$); however, this is not what
we have done. While we do follow their strategy rather closely, the implementation is quite different.  We use usual spacetime coordinates, as opposed to the similarity coordinates used by Giga and Kohn and again by Antonini, Merle, and Zaag.  This makes the geometry of light cones much more transparent, which we exploit to obtain stronger averaged Lyapunov functionals (cf. \eqref{E:L flux ineq 2}), as well as to simplify the key covering argument (cf. our passage from \eqref{almost} to \eqref{tada} with subsection~3.2 in \cite{MerleZaagIMRN}).

\noindent{\bf Acknowledgements}
The first author was supported by NSF grant DMS-1001531.  The second author was supported by an NSF Postdoctoral
Fellowship. The third author was supported by NSF grant DMS-0901166 and a Sloan Foundation Fellowship.

\subsection{Preliminaries}

We will be regularly referring to the spacetime norms
\begin{equation}\label{E:qr def}
\bigl\| u \bigr\|_{L^q_t L^{\vphantom{q}r}_{\vphantom{t}x}(\R\times\R^d)}
    := \biggl(\int_{\R} \bigg[ \int_{\R^d} |u(t,x)|^r\,dx \biggr]^{\frac qr} \;dt\biggr)^\frac1q,
\end{equation}
with obvious changes if $q$ or $r$ is infinity.

We write $X\lesssim Y$ to indicate that $X\leq C Y$ for some implicit constant $C$, which varies from place to place.

Let $\varphi(\xi)$ be a radial bump function supported in the ball $\{ \xi \in \R^d: |\xi| \leq \tfrac {11}{10} \}$ and equal to
$1$ on the ball $\{ \xi \in \R^d: |\xi| \leq 1 \}$.  For each number $N \in 2^{\Z}$, we define the Littlewood--Paley projections
\begin{align*}
\widehat{P_{\leq N} f}(\xi) &:= \varphi(\xi/N) \hat f(\xi)\\
\widehat{P_{> N} f}(\xi) &:= (1 - \varphi(\xi/N)) \hat f(\xi)\\
\widehat{P_N f}(\xi) &:= (\varphi(\xi/N) - \varphi(2\xi/N)) \hat f(\xi)
\end{align*}
and similarly $P_{<N}$ and $P_{\geq N}$.

We will use basic properties of these operators, including

\begin{lemma}[Bernstein estimates]\label{Bernstein}
For $1 \leq p \leq q \leq \infty$,
\begin{align*}
\bigl\| |\nabla|^{\pm s} P_N f\bigr\|_{L^p(\R^d)} &\sim N^{\pm s} \| P_N f \|_{L^p(\R^d)},\\
\|P_{\leq N} f\|_{L^q(\R^d)} &\lesssim N^{\frac{d}{p}-\frac{d}{q}} \|P_{\leq N} f\|_{L^p(\R^d)},\\
\|P_N f\|_{L^q(\R^d)} &\lesssim N^{\frac{d}{p}-\frac{d}{q}} \| P_N f\|_{L^p(\R^d)}.
\end{align*}
\end{lemma}

Next, we recall some well-known elliptic estimates; see, for example, \cite[Ch.~7]{GilTru} or \cite[Ch.~8]{LiebLoss}.

\begin{lemma}(Sobolev inequality for domains) \label{L:Sob domain}
Let $d\geq 2$ and let $\Omega\subset \R^d$ be a domain with the cone property.  Then
$$
\|f\|_{L^q(\Omega)}\lesssim \|f\|_{H^1(\Omega)}
$$
provided that $2\leq q<\infty$ if $d=2$ and $2\leq q\leq \frac{2d}{d-2}$ if $d\geq 3$.  The implicit constant depends only on $d, q,$ and $\Omega$.
\end{lemma}

We will only be applying this lemma to balls, exteriors of balls, and in the whole of $\R^d$; thus, the cone property automatically holds.

\begin{lemma}(Poincar\'e inequality on bounded domains)\label{L:Poincare}
Let $d\geq 2$ and let $\Omega\subset \R^d$ be a bounded domain.  Then for any $f\in H^1_0(\Omega)$,
$$
\|f\|_{L^2(\Omega)}\lesssim |\Omega|^{1/d} \|\nabla f\|_{L^2(\Omega)}.
$$
\end{lemma}

\begin{lemma}(Gagliardo--Nirenberg)\label{L:GN}
Let $d\geq 2$ and let $0<p<\infty$ if $d=2$ and let $0<p\leq\frac4{d-2}$ if $d\geq 3$.  Then
$$
\|f\|_{L^{p+2}}^{p+2} \lesssim \|f\|_{L^{\frac{pd}2}}^p \|\nabla f\|_{L^2}^2 \quad\text{and}\quad \|f\|_{L^{\frac{pd}2}} \lesssim \|f\|_{L^2}^{1-s_c} \|\nabla f\|_{L^2}^{s_c}.
$$
Moreover, for any $R>0$,
$$
\|f\|_{L^{p+2}(|x|\geq R)}^{p+2} \lesssim \|f\|_{L^{\frac{pd}2}(|x|\geq R)}^p \|\nabla f\|_{L^2(|x|\geq R)}^2.
$$
\end{lemma}

%
%
%
%

\section{Local theory.} \label{S:local theory}

\subsection{Strichartz inequalities}

\begin{lemma}[Strichartz inequality]  \label{L:Strichartz}
Fix a value of $m \in [0,1]$.  Let $u$ be a solution to the inhomogeneous equation
\begin{equation} \label{E:ivp}
u_{tt} - \Delta u + m^2 u = F \quad \text{with} \quad u(0) = u_0 \quad \text{and}\quad u_t(0) = u_1
\end{equation}
on the time interval $[0,T]$.  Let  $0 \leq \gamma \leq 1$, $2 < q,\tilde{q} \leq \infty$, and $2 \leq r, \tilde r < \infty$ be exponents satisfying the scaling and admissibility conditions:
$$
\frac1q + \frac{d}r = \frac{d}2 - \gamma = \frac1{\tilde q'} + \frac{d}{\tilde r'} - 2 \qquad \text{and} \qquad
    \frac1q + \frac{d-1}{2r},\ \frac1{\tilde q} + \frac{d-1}{2\tilde r} \leq \frac{d-1}4.
$$
Then
\begin{equation} \label{E:wave strichartz}
\begin{aligned}
&\|\jpn_m^{\gamma} u\|_{C_t L^2_x([0,T] \times \R^d)} + \|\jpn_m^{\gamma-1}u_t\|_{C_t L^2_x([0,T] \times \R^d)} + \|u\|_{L^q_tL^r_x([0,T] \times \R^d)} \\
& \qquad \qquad \lesssim \|\jpn_m^{\gamma}u_0\|_{L^2_x} + \|\jpn_m^{\gamma-1} u_1\|_{L^2_x} + \|F\|_{L^{\tilde q'}_tL^{\tilde r'}_x([0,T] \times \R^d)}.
\end{aligned}
\end{equation}
Here the implicit constant is independent of $m$ and $T$, but may depend on $d$, $\gamma$, $q$, $\tilde q$, $r$, $\tilde r$.
\end{lemma}

\begin{remark}
We will make particularly heavy use of the following special case:
\begin{equation} \label{E:H12 strichartz}
\begin{aligned}
&\|\jpn_m^{\frac12}u\|_{C_tL^2_x([0,T] \times \R^d)} + \|\jpn_m^{-\frac12}u_t\|_{C_tL^2_x([0,T] \times \R^d)} + \|u\|_{L^{\frac{2(d+1)}{d-1}}_{t,x}([0,T] \times \R^d)} \\
&\qquad\qquad  \lesssim \|\jpn_m^{\frac12} u_0\|_{L^2_x} + \|\jpn_m^{-\frac12}u_1\|_{L^2_x} + \|F\|_{L^{\frac{2(d+1)}{d+3}}_{t,x}([0,T] \times \R^d)}.
\end{aligned}
\end{equation}
\end{remark}

\begin{proof}  The principle of stationary phase may be used to show that the linear operator $e^{it\jpn_m}$ satisfies the dispersive estimate
\begin{equation} \label{E:dispersive}
\begin{aligned}
\|e^{it\jpn_m}P_N f\|_{L^{\infty}_x} &\lesssim N^d\bigl(1+\tfrac{|t|N^2}{\jp{N}_m}\bigr)^{-\frac{d-1}2}\|P_Nf\|_{L^1_x} \\ &
\lesssim |t|^{-\frac{d-1}2}\bigl\langle N \bigr\rangle_m^{\frac{d+1}2}\|P_Nf\|_{L^1_x} ,
\end{aligned}
\end{equation}
where the implicit constants are independent of $m$.  Combining this with the fact that $e^{it\jpn_m}$ is an isometry on $L^2_x$, standard arguments (cf.\ \cite{tao:keel} and the references therein) give the Strichartz estimates \eqref{E:wave strichartz}.
\end{proof}

Using the Strichartz estimate, one can easily derive the following standard result:

\begin{prop}[Uniqueness in light cones, \cite{Kapitanski94}] \label{P:uniqueness}
Let $u$ and $\tilde u$ be two strong solutions to \eqref{E:eqn} on the backwards light cone $\Gamma(T,x_0)$.  If $(u(0),u_t(0)) = (\tilde u(0),\tilde u_t(0))$ on $\{x:|x-x_0| \leq T\}$, then $u=\tilde u$ throughout $\Gamma(T,x_0)$.
\end{prop}

\begin{proof}
To keep formulae within margins, we introduce the following notation:  If $I \subset [0,\infty)$ is an interval, then we set $\Gamma_{I} := (I \times \R^d) \cap \Gamma(T,x_0)$.

Let $\eta>0$ be a small constant to be chosen shortly. By Definition~\ref{D:solution}, we may write $[0,T] = \bigcup_{j=1}^{\infty} I_j$ with
$$
\|u\|_{L^{\frac{p(d+1)}2}_{t,x}(\Gamma_{I_j})} + \|\tilde u\|_{L^{\frac{p(d+1)}2}_{t,x}(\Gamma_{I_j})} \leq \eta.
$$
Next, by Lemma~\ref{L:Sob domain}, for each $t\in I_j$ we have
\begin{align*}
\|u(t)\|_{L_x^{\frac{2(d+1)}{d-1}}(|x-x_0|<T-t)}& + \|\tilde u(t)\|_{L_x^{\frac{2(d+1)}{d-1}}(|x-x_0|<T-t)}\\
&\lesssim \|u(t)\|_{H^1_x(|x-x_0|<T-t)}+ \|\tilde u(t)\|_{H^1_x(|x-x_0|<T-t)}.
\end{align*}
Thus, by the definition of strong solution,
$$
\|u\|_{L^{\frac{2(d+1)}{d-1}}_{t,x}(\Gamma_{I_j})} + \|\tilde u\|_{L^{\frac{2(d+1)}{d-1}}_{t,x}(\Gamma_{I_j})}<\infty.
$$

We now consider the difference $w = u - \tilde u$.  By \eqref{E:wave strichartz}, finite speed of propagation, and H\"older's inequality,
\begin{align*}
\|w\|_{L^{\frac{2(d+1)}{d-1}}_{t,x}(\Gamma_{I_1})} &\lesssim \||u|^pu - |\tilde u|^p \tilde u\|_{L^{\frac{2(d+1)}{d+3}}_{t,x}(\Gamma_{I_1}) } \\
&\lesssim \|(|u|^p+ |\tilde u|^p)(u-\tilde u)\|_{L^{\frac{2(d+1)}{d+3}}_{t,x}(\Gamma_{I_1}) } \\
&\lesssim \eta^p \|w\|_{L^{\frac{2(d+1)}{d-1}}_{t,x}(\Gamma_{I_1})} .
\end{align*}
Choosing $\eta$ sufficiently small, we deduce that $w \equiv 0$ on $\Gamma_{I_1}$.  An inductive argument yields $w \equiv 0$ on $\Gamma(T,x_0)$.
\end{proof}

If $m=0$ and $u$ has zero initial data, then it was proved by Harmse in \cite{Harmse} that better estimates than those in Lemma~\ref{L:Strichartz} are possible.

\begin{lemma}[Strichartz estimates for inhomogeneous wave] \label{L:inhomog st}
Let $u$ be a solution to the initial-value problem
$$
u_{tt} - \Delta u = F \qquad \text{with} \qquad u(0) = u_t(0) = 0
$$
on the interval $[0,T]$.  Then
\begin{equation} \label{E:inhomog st}
\|u\|_{L^r_{t,x}([0,T]\times \R^d)} \lesssim \|F\|_{L^{\tilde r'}_{t,x}([0,T] \times \R^d)},
\end{equation}
whenever $r$ and $\tilde r$ satisfy the scaling and acceptability conditions $ \frac1r + \frac1{\tilde r} = \frac{d-1}{d+1}$ and $\frac1r, \frac1{\tilde r} < \frac{d-1}{2d}$.  In particular, \eqref{E:inhomog st} holds with $r = \frac{p(d+1)}2$ and $\tilde r' = \frac{p(d+1)}{2(p+1)}$, provided that $\frac1{2d} < s_c < \frac12$.
\end{lemma}

Finally, in the radial case, lower regularity Strichartz estimates than those given in Lemma~\ref{L:Strichartz} are possible.

\begin{lemma}[Radial Strichartz estimates for homogeneous wave] \label{L:radial wave}
Let $\frac1{2d} < s_c < \frac12$ and let $p = \frac4{d-2s_c}$.  If $u_0 \in \dot{H}^{s_c}_x(\R^d)$ and $u_1 \in \dot{H}^{s_c-1}_x(\R^d)$ are radial, then the solution to the linear wave equation satisfies
\begin{equation} \label{E:radial wave}
\|\scriptS_0(t)(u_0,u_1)\|_{L^{\frac{p(d+1)}2}_{t,x}} \lesssim \|u_0\|_{\dot{H}^{s_c}_x} + \|u_1\|_{\dot{H}^{s_c-1}_x}.
\end{equation}
\end{lemma}

This estimate is implicit in \cite{KlainermanMachedon93} as discussed in \cite{LindbladSogge95}.  We note that the bound \eqref{E:radial wave} is true for larger values of $s_c$
without the assumption of radiality (cf.~\eqref{E:wave strichartz}).


\subsection{Well-posedness results}

Global well-posedness and scattering for NLW with small initial data in critical Sobolev spaces is due to Lindblad and Sogge, \cite{LindbladSogge95}, in the super-conformal case ($\frac12 < s_c < 1$) and to Strauss, \cite{Strauss81}, in the conformal case ($s_c=\frac12$).

In the sub-conformal case, the Lorentz symmetry may be used to construct examples which show that such a small data theory is impossible (cf.\ \cite{LindbladSogge95}).  This
motivates the consideration of \emph{radial} initial data when $s_c<\frac12$.  However, even in order to construct global solutions in $L^{\frac{p(d+1)}2}_{t,x}$ from small radial
data, we need to impose the additional condition $s_c > \frac1{2d}$.  This originates in the fact that $\frac{p(d+1)}2$ with $p=\frac{4d}{d^2-1}$ (i.e. $s_c=\frac1{2d}$)
corresponds to an endpoint in the cone restriction conjecture.  Global well-posedness and scattering for NLW for $\frac1{2d}<s_c<\frac12$ with small radial data in critical
Sobolev spaces may again be found in \cite{LindbladSogge95}.

We summarize below the small data theory that we will use.

\begin{prop}[Critical small data theory for wave] \label{P:Hsc sdt}
Fix $d \geq 2$ and $\frac1{2d} < s_c < 1$ and let $p = \frac4{d-2s_c}$.  Let $(u_0, u_1)\in \dot H^{s_c}_x\times \dot H^{s_c-1}_x$ with $u_0$ and $u_1$ radial when
$\frac1{2d}<s_c<\frac12$.  There exists $\eta_0$ depending on $d$ and $p$ so that if $\eta \leq \eta_0$ and
$$
\|\scriptS_0(t)(u_0,u_1)\|_{L^{\frac{p(d+1)}2}_{t,x}([0,T] \times \R^d)} \leq \eta
$$
and additionally
$$
\||\nabla|^{s_c-\frac12} \scriptS_0(t)(u_0,u_1)\|_{L^{\frac{2(d+1)}{d-1}}_{t,x}([0,T] \times \R^d)} \leq \eta \qquad \text{if} \qquad \tfrac12\leq s_c<1,
$$
then there is a unique solution $u$ to \eqref{E:eqn} with $m=0$ on $[0,T] \times \R^d$.  Moreover, $u$ satisfies
\begin{align}
\|u\|_{L^{\frac{p(d+1)}2}_{t,x}([0,T] \times \R^d)} &\lesssim \eta \label{E:small data Lp}\\
\|\nabla_{t,x} u\|_{L^{\infty}_t \dot H^{s_c-1}_x([0,T] \times \R^d)} &\lesssim \|(u_0,u_1)\|_{\dot H^{s_c}_x \times \dot H^{s_c-1}_x}.\label{E:small data Hsc}
\end{align}
If in addition $(u_0,u_1) \in \dot H^1_x \times L^2_x$, then
\begin{equation} \label{E:small data H1}
\|\nabla_{t,x} u\|_{L^\infty_tL^2_x([0,T] \times \R^d)} \lesssim \|(u_0,u_1)\|_{\dot H^1_x \times L^2_x}.
\end{equation}
In particular, if $\|(u_0,u_1)\|_{\dot H^{s_c}_x \times \dot H^{s_c-1}_x} \leq \eta_1$ for some constant $\eta_1 = \eta_1(d,p) > 0$, then $u$ is global and obeys the estimates above on $\R \times \R^d$.
\end{prop}

\begin{proof}
We begin by reviewing the proof of \eqref{E:small data Lp} and \eqref{E:small data Hsc} from \cite{LindbladSogge95} and then give the additional arguments needed to establish \eqref{E:small data H1}.  Throughout the proof, all spacetime norms will be on $[0,T]\times\R^d$, unless we specify otherwise.

Using Lemma~\ref{L:Strichartz} or Lemma~\ref{L:radial wave} (depending on $s_c$), we have
\begin{align*}
\|\scriptS_0(t)(u_0,u_1)\|_{L^{\frac{p(d+1)}2}_{t,x}} +& \||\nabla|^{s_c-\frac12}\scriptS_0(t)(u_0,u_1)\|_{L^{\frac{2(d+1)}2}_{t,x}}
& \lesssim \|(u_0,u_1)\|_{\dot H^{s_c}_x \times \dot H^{s_c-1}_x}.
\end{align*}
Thus, without loss of generality we may assume that
\begin{equation} \label{E:eta lesssim Hsc}
\eta_0 \lesssim \min\{1, \|(u_0,u_1)\|_{\dot H^{s_c}_x \times \dot H^{s_c-1}_x}\}.
\end{equation}

Let $v \mapsto \Phi_0(v)$ be the mapping given by
$$
\Phi_0(v)(t): = \scriptS_0(t)(u_0,u_1) + \int_0^t |\nabla|^{-1}\sin(|\nabla|(t-s))(|v|^p v)(s)\, ds.
$$
We will use a contraction mapping argument to prove that $\Phi_0$ has a fixed point.

We start with the case when $\frac1{2d} < s_c < \frac12$.  We define
$$
B := \Bigl\{v \in L^{\frac{p(d+1)}2}_{t,x}([0,T] \times \R^d) :\,  \|v\|_{L^{\frac{p(d+1)}2}_{t,x}([0,T] \times \R^d)} \leq 2 \eta\Bigr\}.
$$
By Lemma~\ref{L:inhomog st} and our hypotheses, for $v\in B$ we have
\begin{align*}
\|\Phi_0(v)\|_{L^{\frac{p(d+1)}2}_{t,x}}
&\leq \|\scriptS_0(t)(u_0,u_1)\|_{L^{\frac{p(d+1)}2}_{t,x}} + C_{d,p}\||v|^pv\|_{L^{\frac{p(d+1)}{2(p+1)}}_{t,x}} \\
&\leq \eta + C_{d,p}\|v\|_{L^{\frac{p(d+1)}2}_{t,x}}^{p+1}\\
&\leq \eta + C_{d,p}\eta^{p+1}.
\end{align*}
Thus for $\eta$ sufficiently small, $\Phi_0$ maps $B$ into itself.

To see that $\Phi_0$ is a contraction on $B$ with respect to the metric given by $d(u,v) = \|u-v\|_{L^{\frac{p(d+1)}2}_{t,x}}$, we apply
Lemma~\ref{L:inhomog st} and use H\"older's inequality:
\begin{align*}
\|\Phi_0(u) - \Phi_0(v)\|_{L^{\frac{p(d+1)}2}_{t,x}}
&\leq C_{d,p} \||u|^pu - |v|^p v\|_{L^{\frac{p(d+1)}{2(p+1)}}_{t,x}} \\
&\leq C_{d,p} \||u|+|v|\|^p_{L^{\frac{p(d+1)}2}_{t,x}}\|u-v\|_{L^{\frac{p(d+1)}2}_{t,x}}\\
&\leq C_{d,p} \eta^p \|u-v\|_{L^{\frac{p(d+1)}2}_{t,x}}.
\end{align*}
Thus for $\eta$ sufficiently small, $\Phi_0$ is a contraction on $B$ and so it has a fixed point $u$ in $B$.  Moreover, by
Lemma~\ref{L:Strichartz}, the fixed point satisfies
\begin{align*}
\|\nabla_{t,x} u\|_{C_t \dot H^{s_c-1}_x}
&\lesssim \|(u_0,u_1)\|_{\dot H^{s_c}_x \times \dot H^{s_c-1}_x} + \||u|^pu\|_{L^{\frac{p(d+1)}{2(p+1)}}_{t,x}} \\
&\lesssim \|(u_0,u_1)\|_{\dot H^{s_c}_x \times \dot H^{s_c-1}_x} + \eta^{p+1}.
\end{align*}
The bound \eqref{E:small data Hsc} then follows from \eqref{E:eta lesssim Hsc}.

We now turn to the case when $\frac12 \leq s_c < 1$.  We define
\begin{align*}
B := \Bigl\{v \in L^{\frac{p(d+1)}2}_{t,x}([0,T] \times \R^d) :\, \|v\|_{L^{\frac{p(d+1)}2}_{t,x}} + \||\nabla|^{s_c-\frac12} v\|_{L^{\frac{2(d+1)}{d-1}}_{t,x}}\leq 3\eta\Bigr\}.
\end{align*}
By Lemma~\ref{L:Strichartz} and the fractional chain rule, for $v\in B$ we obtain
\begin{align*}
\|\Phi_0(v)\|_{L^{\frac{p(d+1)}2}_{t,x}} &+\||\nabla|^{s_c-\frac12}\Phi_0(v)\|_{L^{\frac{2(d+1)}{d-1}}_{t,x}} \\
&\leq \|\scriptS_0(t)(u_0,u_1)\|_{L^{\frac{p(d+1)}2}_{t,x}} +\||\nabla|^{s_c-\frac12}\scriptS_0(t)(u_0,u_1)\|_{L^{\frac{2(d+1)}{d-1}}_{t,x}} \\
&\quad+ C_{d,p}\||\nabla|^{s_c-\frac12}(|v|^pv)\|_{L^{\frac{2(d+1)}{d+3}}_{t,x}} \\
&\leq 2\eta + C_{d,p}\||\nabla|^{s_c-\frac12}v\|_{L^{\frac{2(d+1)}{d-1}}_{t,x}}\|v\|_{L^{\frac{p(d+1)}2}_{t,x}}^p  \\
&\leq 2\eta + C_{d,p}\eta^{p+1}.
\end{align*}
Thus if $\eta$ is sufficiently small, $\Phi_0$ maps $B$ into itself.

Since in this case we are considering $\frac4{d-1}\leq p<\frac4{d-2}$, by H\"older's inequality we have
$$
\sup_{x_0 \in \R^d} \|v\|_{L^{\frac{2(d+1)}{d-1}}_{t,x}(\Gamma(T,x_0))}
\lesssim T^{s_c-\frac12}\sup_{x_0 \in \R^d}\|v\|_{L^{\frac{p(d+1)}2}_{t,x}(\Gamma(T,x_0))}
\lesssim  T^{s_c-\frac12}\eta
$$
for any $v \in B$.  Thus we may consider the metric on $B$ given by
$$
d(u,v) = \sup_{x_0 \in \R^d} \|u-v\|_{L^{\frac{2(d+1)}{d-1}}_{t,x}(\Gamma(T,x_0))}.
$$
By \eqref{E:H12 strichartz}, finite speed of propagation, and H\"older's inequality,
\begin{align*}
\|\Phi_0(u)-\Phi_0(v)\|_{L^{\frac{2(d+1)}{d-1}}_{t,x}(\Gamma(T,x_0))}
&\lesssim \||u|^pu -|v|^pv\|_{L^{\frac{2(d+1)}{d+3}}_{t,x}(\Gamma(T,x_0))}\\
&\lesssim \|(|u|+|v|)^p\|_{L^{\frac{d+1}2}_{t,x}(\Gamma(T,x_0))} \|u-v\|_{L^{\frac{2(d+1)}{d-1}}_{t,x}(\Gamma(T,x_0))} \\
&\lesssim \eta^p\|u-v\|_{L^{\frac{2(d+1)}{d-1}}_{t,x}(\Gamma(T,x_0))},
\end{align*}
for each $x_0 \in \R^d$.  Therefore if $\eta$ is sufficiently small, $\Phi_0$ is a contraction on $B$ with respect to the metric $d$ and so $\Phi_0$ has a fixed point $u$ in $B$.  By Lemma~\ref{L:Strichartz} and the fractional chain rule, $u$ satisfies
\begin{align*}
\|\nabla_{t,x}u\|_{C_t\dot H^{s_c-1}_x}
&\lesssim \|(u_0,u_1)\|_{\dot H^{s_c}_x \times \dot H^{s_c-1}_x} + \||\nabla|^{s_c-\frac12}(|u|^pu)\|_{L^{\frac{2(d+1)}{d+3}}_{t,x}} \\
&\lesssim \|(u_0,u_1)\|_{\dot H^{s_c}_x \times \dot H^{s_c-1}_x} + \eta^{p+1}.
\end{align*}
The bound \eqref{E:small data Hsc} then follows from \eqref{E:eta lesssim Hsc}.

Finally, we turn to the persistence of regularity statement \eqref{E:small data H1}.  For $\frac1{2d} < s_c < 1$, by \eqref{E:H12 strichartz}, the fractional chain rule, and \eqref{E:small data Lp},
\begin{align*}
\|\nabla_{t,x} u\|_{C_tL^2_x} + \| |\nabla|^{\frac12}u \|_{L^{\frac{2(d+1)}{d-1}}_{t,x}}
&\lesssim \|(u_0,u_1)\|_{\dot H^1_x \times L^2_x} + \| |\nabla|^{\frac12} (|u|^pu) \|_{L^{\frac{2(d+1)}{d+3}}_{t,x}} \\
&\lesssim \|(u_0,u_1)\|_{\dot H^1_x \times L^2_x} + \| |\nabla|^{\frac12}u \|_{L^{\frac{2(d+1)}{d-1}}_{t,x}}\|u\|_{L^{\frac{p(d+1)}2}_{t,x}}^p \\
&\lesssim \|(u_0,u_1)\|_{\dot H^1_x \times L^2_x} + \eta^p \| |\nabla|^{\frac12}u \|_{L^{\frac{2(d+1)}{d-1}}_{t,x}}.
\end{align*}
Therefore \eqref{E:small data H1} holds if $\eta$ is chosen sufficiently small.  This completes the proof.
\end{proof}

The local existence theory of \eqref{E:eqn} in $H^1_x\times L_x^2$ is well-known (cf.\ \cite{GV:IHP89}, \cite{GV:MathZ85}); however, we need a result which is uniform in $m$.  This is the topic of the following proposition:

\begin{prop}[$H^1_x\times L_x^2$ local well-posedness for \eqref{E:eqn}] \label{P:lwp}
Let $d\geq 2$, $m \in [0,1]$, $0 < s_c < 1$, and take $p = \frac4{d-2s_c}$.  Let $u_0,u_1$ be initial data satisfying
$$
\|u_0\|_{H^1_x} + \|u_1\|_{L^2_x} \leq M < \infty.
$$
Then there exist $T\gtrsim_{p,d} \min\{ M^{-1/(1-s_c)}, M^{-p(d+1)/2}\}$, independent of $m$, and a unique solution $u$ to \eqref{E:eqn} on $[0,T]$.  Furthermore, this solution satisfies
\begin{align}
\|\nabla_{t,x}u\|_{C_tL_x^2([0,T] \times \R^d)} &\lesssim M \label{E:H1lwp u in H1}\\
\|u\|_{C_tL^2_x([0,T] \times \R^d)} &\lesssim (1+T)M \label{E:H1lwp mass}\\
\|u\|_{L^{\frac{p(d+1)}2}_{t,x}([0,T] \times \R^d)} &\lesssim \max\{T^{1-s_c},T^{\frac2{p(d+1)}}\}M, \label{E:H1lwp u in Lq}
\end{align}
with the implicit constants depending only on $d,p$.
\end{prop}

\begin{remark}
Well-posedness in $H^1_x\times L_x^2$ ensures that conservation of energy, which follows from an elementary computation for smooth, decaying solutions, continues to hold for solutions in $C_t(H^1_x\times L_x^2)$.
\end{remark}

\begin{proof}[Proof of Proposition~\ref{P:lwp}]
Throughout the proof, all spacetime norms will be over the set $[0,T]\times\R^d$.  We use the contraction mapping argument for the map
$v \mapsto \Phi_m(v)$ given by
$$
\Phi_m(v)(t): = \scriptS_m(t)(u_0,u_1) + \int_0^t \jpn_m^{-1}\sin\bigl((t-s)\jpn_m\bigr)(|v|^pv)(s)\, ds.
$$
Our analysis breaks into two cases.

If $\frac{d+2}{2d} < s_c < 1$ (that is, $\frac{4d}{(d-2)(d+1)}<p<\frac4{d-2}$), we define
\begin{align*}
B := \Bigl\{v \in C_tH^1_x([0,T] \times \R^d) : \, &\|\nabla_{t,x} v\|_{C_tL^2_x} + \|\jpn_m^{\frac12}v\|_{L^{\frac{2(d+1)}{d-1}}_{t,x}} \\ &+ \|v\|_{L^{\frac{2p(d+1)}{p(d+1)(d-2) - 4d}}_tL^{\frac{p(d+1)}2}_x} \leq C_{d,p}M \Bigr\}.
\end{align*}
By H\"older's inequality, for $v \in B$ we have
\begin{equation} \label{E:B subset Lq}
\|v\|_{L^{\frac{p(d+1)}2}_{t,x}} \leq T^{1-s_c}\|v\|_{L^{\frac{2p(d+1)}{p(d+1)(d-2) - 4d}}_tL^{\frac{p(d+1)}2}_x}\lesssim_{d,p} T^{1-s_c} M.
\end{equation}
Using this together with Lemma~\ref{L:Strichartz} and the fractional chain rule, we obtain
\begin{align*}
&\|\nabla_{t,x}\Phi_m(v)\|_{C_tL^2_x} + \|\jpn_m^{\frac12} \Phi_m(v)\|_{L^{\frac{2(d+1)}{d-1}}_{t,x}}
    + \|\Phi_m(v)\|_{L^{\frac{2p(d+1)}{p(d+1)(d-2) - 4d}}_tL^{\frac{p(d+1)}2}_x}\\
&\qquad \lesssim_{d,p} \|(u_0,u_1)\|_{H^1_x \times L^2_x} + \|\jpn_m^{\frac12}(|v|^pv)\|_{L^{\frac{2(d+1)}{d+3}}_{t,x}} \\
&\qquad \lesssim_{d,p} M +\|\jpn_m^{\frac12}v\|_{L^{\frac{2(d+1)}{d-1}}_{t,x}}\|v\|_{L^{\frac{p(d+1)}2}_{t,x}}^p\\
&\qquad \lesssim_{d,p} M + T^{(1-s_c)p}M^{p+1}.
\end{align*}
Thus for $T$ sufficiently small depending on $d$, $p$, and $M$, $\Phi_m$ maps $B$ into itself.

Next, we will show that $\Phi_m$ is a contraction with respect to the metric given by
\begin{equation}\label{ssdist}
d(u,v) =\|u-v\|_{L^{\frac{2(d+1)}{d-1}}_{t,x}}.
\end{equation}
We start by noting that, by the fundamental theorem of calculus,
\begin{equation} \label{E:B subset L2}
\|v\|_{C_tL^2_x} \leq \|u_0\|_{L^2_x} + T\|v_t\|_{C_tL^2_x} \lesssim_{d,p} (1+T)M
\end{equation}
for any $v\in B$.  Thus, by H\"older and Sobolev embedding,
\begin{align*}
\|v\|_{ L^{\frac{2(d+1)}{d-1}}_{t,x}} \lesssim T^{\frac{d-1}{2(d+1)}} \|v\|_{C_tH^1_x} \lesssim_{d,p} T^{\frac{d-1}{2(d+1)}} (1+T)  M.
\end{align*}
To continue, we use \eqref{E:H12 strichartz}, \eqref{E:B subset Lq}, and H\"older's inequality to estimate
\begin{align*}
\|\Phi_m(u) - \Phi_m(v)\|_{L^{\frac{2(d+1)}{d-1}}_{t,x}} &\lesssim \||u|^p u - |v|^p v\|_{L^{\frac{2(d+1)}{d+3}}_{t,x}}\\
&\lesssim \|(|u|+|v|)^p\|_{L^{\frac{d+1}2}_{t,x}} \|u-v\|_{L^{\frac{2(d+1)}{d-1}}_{t,x}}\\
&\lesssim_{d,p} T^{(1-s_c)p} M^p\|u-v\|_{L^{\frac{2(d+1)}{d-1}}_{t,x}}
\end{align*}
for any $u,v \in B$.   Therefore for $T$ sufficiently small, $\Phi_m$ is a contraction and consequently has a fixed point $u \in B$.  Claims
\eqref{E:H1lwp mass} and \eqref{E:H1lwp u in Lq} follow from \eqref{E:B subset Lq} and \eqref{E:B subset L2}.

It remains to treat the case $0 <s_c \leq\frac{d+2}{2d}$.  This time, we define
$$
B := \Bigl\{v \in C_t H^1_x([0,T] \times \R^d) : \|\nabla_{t,x} v\|_{C_tL^2_x} +\|\jpn_m^{\frac12}v\|_{L^{\frac{2(d+1)}{d-1}}_{t,x}} \leq C_{d,p}M\Bigr\}.
$$
In this case $2 < \frac{p(d+1)}2 \leq \frac{2d}{d-2}$ and so, using H\"older, Sobolev embedding, and \eqref{E:B subset L2}, we obtain
\begin{align} \label{E:B subset Lq small p}
\|v\|_{L^{\frac{p(d+1)}2}_{t,x}}
&\lesssim T^{\frac2{p(d+1)}}\|v\|_{C_t H^1_x} \lesssim_{d,p}  T^{\frac2{p(d+1)}} (1+T) M
\end{align}
for any $v\in B$.  Arguing as in the previous case, and substituting \eqref{E:B subset Lq small p} for \eqref{E:B subset Lq}, we derive
\begin{align*}
\|\nabla_{t,x}\Phi_m(v)\|_{C_tL^2_x} &+ \|\jpn_m^{\frac12}\Phi_m(v)\|_{L^{\frac{2(d+1)}{d-1}}_{t,x}} \\
&\lesssim \|(u_0,u_1)\|_{H^1_x \times L^2_x} + \|\jpn_m^{\frac12}v\|_{L^{\frac{2(d+1)}{d-1}}_{t,x}}\|v\|^p_{L^{\frac{p(d+1)}2}_{t,x}} \\
& \lesssim _{d,p} M + T^{\frac2{d+1}}(1+T)^{p}M^{p+1}.
\end{align*}
Thus if $T$ is sufficiently small, we again obtain that $\Phi_m$ maps $B$ into itself.  The proof that $\Phi_m$ is a contraction on $B$ with respect to the metric \eqref{ssdist} is exactly the same as in the previous case.  This completes the proof.
\end{proof}

Because \eqref{E:eqn} obeys finite speed of propagation, we may localize in space in Proposition~\ref{P:lwp}.  In this way we obtain the
following scale-invariant lower bound on the blowup rate.

\begin{corollary}[$H^1_\loc\times L^2_\loc$ local well-posedness and blowup criterion] \label{C:lwp loc}
Let $d\geq 2$, $m \in [0,1]$, $0 < s_c < 1$, and take $p=\frac{4}{d-2s_c}$.  Let $(u_0,u_1)$ be initial data satisfying $\|u_0\|_{H^1_\loc} + \|u_1\|_{L^2_\loc} \leq M < \infty$, where we define
$$
\|f\|_{L^2_\loc}^2 := \sup_{x_0 \in \R^d} \int_{|x-x_0|\leq 1} |f(x)|^2\, dx \qtq{and} \|f\|_{H^1_\loc}^2 := \|f\|_{L^2_\loc}^2 + \|\nabla f\|_{L^2_\loc}^2.
$$
Then there exist $T_0 > 0$, depending only on $d,p,M$ and a unique strong solution $u:[0,T_0] \times \R^d \to \R$ to \eqref{E:eqn} satisfying
$$
\|u\|_{C_tH^1_\loc([0,T_0] \times \R^d)} + \|u_t\|_{C_tL^2_\loc([0,T_0] \times \R^d)} \lesssim M.
$$
Furthermore, if $u$ blows up at time $0<T_* < \infty$ and $(T_*,x_0)$ lies on the forward-in-time blowup surface of $u$, then
\begin{align} \label{E:H1loc lb}
1 \lesssim (T_*-t)^{-2s_c} \int_{|x-x_0| \leq T_*-t} |u(t,x)|^2 + (T_*-t)^2|\nabla_{t,x} u(t,x)|^2 \, dx
\end{align}
for all $t>0$ such that $T_*-1 \leq t < T_*$.  The implicit constant depends only on $d,p$.
\end{corollary}

We note that in Theorem~$1'$ in \cite{MerleZaagAJM} and Theorem~1(ii) of \cite{MerleZaagMA}, Merle and Zaag claim that an alternative blowup criterion holds, namely,
\begin{equation} \label{E:MerleZaag H1 loc blowup}
1 \lesssim \sup_{x_0 \in \R^d} (T_*-t)^{d-2s_c}\int_{|x-x_0|\leq 1}|u(t,x)|^2 + (T_*-t)^2|\nabla_{t,x}u(t,x)|^2\, dx.
\end{equation}
This lower bound is also repeated as equation (1.8) in \cite{MerleZaagIMRN}.  It seems that in all three instances this is essentially
a typo, since \eqref{E:H1loc lb} is equivalent to the lower bound in self-similar variables given in Theorem~1 of \cite{MerleZaagAJM} and
Theorem~1(i) of \cite{MerleZaagMA}, while \eqref{E:MerleZaag H1 loc blowup} is not.  Moreover, the scaling argument that Merle and Zaag
suggest to prove \eqref{E:MerleZaag H1 loc blowup} seems only to establish \eqref{E:H1loc lb}.

It is not difficult to construct a counterexample to \eqref{E:MerleZaag H1 loc blowup}.  For a general subluminal blowup surface $t=\sigma(x)$,
Kichenassamy \cite{Kich} (see also \cite{KVblowup}) has constructed solutions with $u(t,x)\sim (\sigma(x)-t)^{-2/p}$.  Whenever the blowup surface
is smooth with non-zero curvature at the first blowup point, this is inconsistent with \eqref{E:MerleZaag H1 loc blowup}.

We turn now to the proof of Corollary~\ref{C:lwp loc}.

\begin{proof}
Both conclusions may be proved by applying Proposition~\ref{P:lwp} to spatially truncated initial data and then invoking finite speed of propagation.  Since the proof of the first conclusion is a little simpler, we give the details only for the second.  We argue by contradiction.

To this end, let $T_*-1 < t_0 < T_*$ and let $x_0\in \R^d$ be such that $(T_*,x_0)$ lies on the forward-in-time blowup surface of $u$.  By space-translation invariance, we may assume $x_0=0$.  Suppose that $u$ satisfies
\begin{align}\label{uT*}
(T_*-t_0)^{-2s_c} \int_{|x| \leq T_*-t_0} |u(t_0,x)|^2 + (T_*-t_0)^2|\nabla_{t,x} u(t_0,x)|^2 \, dx < \eta,
\end{align}
for some small constant $\eta$ to be determined in a moment.  Now set
$$
\tilde u(t,x):=(T_*-t_0)^{\frac2p}u(t_0+ (T_*-t_0) t, (T_*-t_0) x).
$$
A simple computation shows that $\tilde u$ satisfies \eqref{E:eqn} with $m$ replaced by $\tilde m:=(T_*-t_0)m$.  Moreover, as $(T_*,0)$ belongs to the forward-in-time blowup surface of $u$, we see that $(1,0)$ lies on the blowup surface of $\tilde u$.  Changing variables, \eqref{uT*} becomes
$$
\int_{|x| \leq 1} |\tilde u(0)|^2 + |\nabla_{t,x} \tilde u(0)|^2 \, dx < \eta.
$$
Thus, by the dominated convergence theorem, there exists $0<\delta<\frac12$ such that
\begin{equation} \label{E:u small 1}
\int_{|x| \leq 1+\delta} |\tilde u(0,x)|^2 + |\nabla_{t,x} \tilde u(0,x)|^2 \, dx < 2\eta.
\end{equation}

To continue, we define $v_0$ and $v_1$ such that $v_0 = \tilde u(0)$ and $v_1 = \tilde u_t(0)$ on $|x| \leq 1+\delta$, $v_0 = v_1 = 0$ on $|x| \geq 2$, and
\begin{equation}\label{E:tilde u small}
\|v_0\|_{H^1_x}^2 + \|v_1\|_{L^2_x}^2 \lesssim \eta.
\end{equation}
(For example, one can take $v_0$ to be the harmonic function on the annulus $1+\delta<|x|<2$ that matches these boundary values.)
For $\eta$ sufficiently small depending on $d,p, \delta$ (but not on $\tilde m\in[0,1]$), Proposition~\ref{P:lwp} yields a solution to the initial-value problem
$$
v_{tt} - \Delta v + \tilde m^2 v = |v|^p v \qtq{with} v(0) = v_0 \qtq{and} v_t(0) = v_1
$$
on $[0,1+\delta]\times \R^d$.  Thus by finite speed of propagation, $\tilde u$ may be extended to a strong solution on the backward light cone $\Gamma(1+\delta,0)$, which contradicts the fact that $(1,0)$ lies on the blowup surface of $\tilde u$.
\end{proof}

\begin{corollary}[$\dot H^1_x\times L_x^2$ blowup criterion] \label{C:lwp dot}
Let $d\geq 3$, $m \in [0,1]$, $0 < s_c < 1$, and take $p=\frac4{d-2s_c}$.  Given initial data $(u_0,u_1)\in \dot H^1_x\times L_x^2$, if the solution $u$ to \eqref{E:eqn} blows up at time $0<T_* < \infty$, then
\begin{align} \label{E:H1dot lb}
1 \lesssim (T_*-t)^{2-2s_c}\int_{\R^d}|\nabla_{t,x}u(t,x)|^2\, dx
\end{align}
for all $t>0$ such that $T_*-1 \leq t < T_*$. The implicit constant depends only on $d,p$.
\end{corollary}

\begin{proof}
Note that by H\"older's inequality and Sobolev embedding, $\dot H^1_x \subset H^1_\loc$.  The claim now follows from Corollary~\ref{C:lwp loc}.
\end{proof}

\begin{remark}
Note that in two dimensions, $\dot H^1_x$ cannot be realized as a space of distributions.  Moreover, it is not difficult to construct concrete initial data
that show that \eqref{E:H1dot lb} does not hold:  Given $R>1$, let
$$
u_1:= 0 \qtq{and} u_0(x) := \begin{cases} 0 & : |x| > R \\[1mm] -\frac{\log(|x|/R)}{\sqrt{\log(R)}} & : 1 \leq |x| \leq R \\[2mm]
\sqrt{\log(R)} & : |x| < 1.\end{cases}
$$
Note that $\int |u_1|^2 + |\nabla u_0|^2\,dx \sim 1$.  However, as $R\to\infty$ the corresponding solution blows up more and more quickly; indeed,
by solving the ODE and using finite speed of propagation, we see that the lifespan cannot exceed
$$
\int_{A}^\infty \bigl[ \tfrac{2}{p+2} \bigl(u^{p+2} - A^{p+2}\bigr)\bigr]^{-1/2} du \sim A^{-\frac{p+4}{2}}  \quad\text{where $A:=\sqrt{\log(R)}$.}
$$
\end{remark}

The following result shows that blowup must be accompanied by the blowup of the $\dot H^1_x$ norm of $u$.  In this sense, while non-quantitative,
it is a strengthening of Corollary~\ref{C:lwp dot}, which provides a lower bound on the full spacetime gradient of $u$.

\begin{corollary}\label{C:ee blowup}
Let $d \geq 2$, $m \in [0,1]$, and $0 < s_c < 1$.  Set $p = \frac4{d-2s_c}$.  Let $(u_0,u_1) \in H^1_x \times L^2_x$ and assume
that the maximal-lifespan solution $u$ to \eqref{E:eqn} cannot be extended past time $0<T_* < \infty$.  Then
$$
\lim_{t \uparrow T_*} \|\nabla u(t)\|_{L^2_x} = \infty.
$$
The same conclusion holds if the initial displacement $u_0$ merely belongs to $\dot H^1_x\cap \dot H^{s_c}_x$.
\end{corollary}

\begin{proof}
By Proposition~\ref{P:lwp}, the solution $u$ can be extended as long as $(u,u_t)$ remains bounded in $H^1_x\times L^2_x$.  Thus, as $u$ cannot be extended past time $0<T_*<\infty$, we must have
\begin{align}\label{E:all to infty}
\lim_{t \uparrow T_*} \bigl\{\|u(t)\|_{H^1_x}+ \|u_t(t)\|_{L_x^2} \bigr\}= \infty.
\end{align}

Let $\chi_R:=\phi(x/R)$ be a smooth cutoff to the ball of radius $R$.  Combining dominated convergence with Proposition~\ref{P:lwp}, we can find $R>10 T_*$ large enough
so that initial data $\tilde u_0:=(1-\chi_R)u_0$ and $\tilde u_1:=(1-\chi_R)u_1$ lead to a solution $\tilde u$ up to time $2T_*$.  Moreover, $\tilde u$ remains uniformly
bounded in $H^1_x\times L_x^2$ on $[0,T_*]$ and so, by conservation of energy, the potential energy of $\tilde u$ is also bounded on $[0,T_*]$.  By finite speed of propagation, the original solution $u$ agrees with $\tilde u$ on $[0, T_*]\times\{|x|\geq 3R\}$, and so inherits these bounds; in particular,
\begin{align}\label{inherited bounds}
\|u\|_{L_t^\infty L_x^2([0, T_*]\times\{|x|\geq 3R\})} + \|u\|_{L_t^\infty L_x^{p+2}([0, T_*]\times\{|x|\geq 3R\})}<\infty.
\end{align}

When $m>0$, conservation of energy and \eqref{E:all to infty} dictate
\begin{align}\label{pe blowup}
\lim_{t \uparrow T_*} \|u(t)\|_{L^{p+2}_x}^{p+2} =\infty.
\end{align}
Combining this with \eqref{E:all to infty}, we conclude
\begin{align}\label{pe blowup'}
\lim_{t \uparrow T_*} \|\chi_{6R}u(t)\|_{L^{p+2}_x}^{p+2} =\infty.
\end{align}

This conclusion also holds when $m=0$.  Indeed, the argument above is applicable to all sequences $t_n \uparrow T_*$ for which $\|\nabla_{t,x} u(t_n)\|_{L^2_x}\to \infty$.  On sequences where $\|\nabla_{t,x} u(t_n)\|_{L^2_x}$ is bounded, \eqref{E:all to infty} guarantees $\| u(t_n)\|_{L^2_x} \to \infty$.
However, in this case \eqref{pe blowup'} follows by using the $L_t^\infty L_x^2$ estimate in \eqref{inherited bounds} and H\"older's inequality.

Using the Gagliardo--Nirenberg inequality followed by Lemma~\ref{L:Poincare} (on the ball $\{|x|\leq 24R\}$), we obtain
\begin{align*}
\|\chi_{6R}u(t)\|_{L^{p+2}_x}^{p+2}&\lesssim \|\chi_{6R}u(t)\|_{L^2_x}^{p(1-s_c)} \|\nabla [\chi_{6R}u(t)]\|_{L^2_x}^{\frac{pd}2}\\
&\lesssim R^{p(1-s_c)}\|\nabla [\chi_{6R}u(t)]\|_{L_x^2}^{p+2}\\
&\lesssim R^{p(1-s_c)} \bigl[\|\nabla u(t)\|_{L_x^2}+ R^{-1}\|u(t)\|_{L^2(|x|\geq 3R)} \bigr]^{p+2}.
\end{align*}
Combining this with \eqref{inherited bounds} and \eqref{pe blowup'}, we derive the claim.

This completes the proof of the corollary for data $(u_0, u_1)\in H^1_x\times L^2_x$.  For initial data $u_0\in \dot H^1_x\cap \dot H^{s_c}_x$ we observe
that for $R>10T_*$ sufficiently large, the restriction of $u_0$ to the region $|x|\geq R$ is small in $H^1_{\loc}$.  Thus by Proposition~\ref{P:lwp},
the solution extends to the region $[0, 2T_*]\times\{|x|\geq 3R\}$ in the class $H^1_{\loc}\times L^2_{\loc}$.  Now consider the solution $v$ with initial
data $\chi_{10R} u_0$ and $\chi_{10R} u_1$.  By applying the first version of this corollary, we see that $\|\nabla v(t)\|_{L^2_x}$ diverges as $t\to T_*$.
Moreover, by the bounds on $v$ where $|x|\geq 3R$, this divergence must occur in the region $|x|\leq 6R$ where finite speed of propagation guarantees $v\equiv u$.
\end{proof}

We will also need a stability result for the nonlinear wave equation in the weak topology.

\begin{lemma} \label{L:weak stability}
Let $d\geq2$, $0 < s_c < 1$, and set $p = \frac4{d-2s_c}$.  Let $\{m_n\}_{n\geq 1},\{\lambda_n\}_{n\geq 1} \subset [0,1]$ be sequences with $\lim m_n = \lim \lambda_n = 0$ and let $\{(u_0^{(n)}, u_1^{(n)})\}_{n\geq 1}$ be a sequence of initial data such that
\begin{equation} \label{E:un to u at 0}
\nabla u^{(n)}_0 \rightharpoonup \nabla u_0  \qtq{and} u^{(n)}_1 \rightharpoonup u_1 \qtq{weakly in $L^2_x$.}
\end{equation}
Assume also that the sequence $\{u^{(n)}\}_{n\geq 1}$ of solutions to
$$
\partial_{tt} u^{(n)} - \Delta u^{(n)} + m_n^2 u^{(n)} = |u^{(n)}|^p u^{(n)}  \qtq{on} [0,T) \times \R^d
$$
with initial data $(u_0^{(n)}, u_1^{(n)})$ at time $t=0$ satisfy
\begin{equation} \label{E:bounded sequences}
\begin{aligned}
\|\nabla_{t,x}u^{(n)}\|_{C_tL^2_x([0,T) \times \R^d)} &+ \||\nabla|^{s_c} u^{(n)}\|_{C_tL^2_x([0,T) \times \R^d)} \\
&+ \|\jpn_{\lambda_n}^{s_c-1}u^{(n)}_t\|_{C_tL^2_x([0,T) \times \R^d)} \leq M<\infty.
\end{aligned}
\end{equation}
Then the initial-value problem
\begin{equation} \label{E:nlw}
u_{tt} - \Delta u = |u|^p u \qtq{with} u(0) = u_0 \qtq{and} \partial_t u(0) = u_1
\end{equation}
has a strong solution on $[0,T) \times \R^d$ with $(u,u_t) \in C_t[\dot H^1_x \times L^2_x] \cap C_t[\dot H^{s_c}_x \times \dot H^{s_c-1}_x]$.  Furthermore, for each $t \in [0,T)$, we have
\begin{equation} \label{E:un to u at t}
(u^{(n)}(t),\partial_t u^{(n)}(t)) \rightharpoonup (u(t),\partial_t u(t)) \qtq{weakly in $\dot H^1_x \times L^2_x$.}
\end{equation}
Consequently, the limiting solution $u$ obeys the bounds
\begin{equation} \label{E:stability bounds}
\|\nabla_{t,x} u\|_{C_tL^2_x([0,T)\times \R^d)} + \|\nabla_{t,x} u\|_{C_t \dot H^{s_c-1}_x([0,T) \times \R^d)} \leq M.
\end{equation}
\end{lemma}

\begin{proof}
We will prove that there exists a time $0<t_0 < \min\{1,T\}$, depending only on $M$, such that $u$ exists up to time $t_0$ and satisfies \eqref{E:un to u at t} for each $t \in [0,t_0]$.  The lemma follows from this and a simple iterative argument.

We will construct the solution $u$ on $[0, t_0]\times\R^d$ by gluing together solutions defined in light cones.  To this end, let $x_0 \in \R^d$ and let $\phi$ be a smooth cutoff such that $\phi(x) = 1$ for $|x| \leq 1$ and $\phi(x) = 0$ for $|x| \geq 2$.  For $j=0,1$ we define the initial data
\begin{gather*}
u_{x_0,j}(x) := \phi(x-x_0)u_j(x) \qtq{and} u^{(n)}_{x_0,j}(x) := \phi(x-x_0)u^{(n)}_j(x).
\end{gather*}
By \eqref{E:un to u at 0},
\begin{equation} \label{E:wk conv after loc}
(u^{(n)}_{x_0,0},u^{(n)}_{x_0,1}) \rightharpoonup (u_{x_0,0},u_{x_0,1}) \qtq{weakly in $\dot H^1_x \times L^2_x$}
\end{equation}
and so, by Rellich--Kondrashov,
\begin{equation} \label{E:H12 conv after loc}
(u^{(n)}_{x_0,0},u^{(n)}_{x_0,1}) \to (u_{x_0,0},u_{x_0,1}) \qtq{in $\dot H^{\frac12}_x \times \dot H^{-\frac12}_x$.}
\end{equation}
Furthermore, by \eqref{E:bounded sequences},
\begin{equation} \label{E:H1 after loc}
\|u_{x_0,0}\|_{H^1_x} + \|u_{x_0,1}\|_{L^2_x} + \|u^{(n)}_{x_0,0}\|_{H^1_x} + \|u^{(n)}_{x_0,1}\|_{L^2_x} \lesssim M.
\end{equation}
Thus, by Proposition~\ref{P:lwp} there exists a time $0<t_0< 1$, depending only on $M$, such that the solutions $u_{x_0}$ and $u_{x_0}^{(n)}$ to
$$
\begin{cases} \partial_{tt}u_{x_0} - \Delta u_{x_0} = |u_{x_0}|^pu_{x_0} \\ u_{x_0}(0) = u_{x_0,0}, \quad \partial_t u_{x_0}(0) = u_{x_0,1} \end{cases}
\quad
\begin{cases} \partial_{tt}u^{(n)}_{x_0} - \Delta u^{(n)}_{x_0} + m_n^2 u^{(n)} = |u^{(n)}_{x_0}|^pu_{x_0}^{(n)} \\
        u^{(n)}_{x_0}(0) = u^{(n)}_{x_0,0}, \quad \partial_t u^{(n)}_{x_0}(0) = u^{(n)}_{x_0,1} \end{cases}
$$
exist on $[0,t_0]\times \R^d$ and satisfy the bounds
\begin{align}
\|\nabla_{t,x} u_{x_0}\|_{C_tL^2_x([0,t_0] \times \R^d)} + \|\nabla_{t,x} u_{x_0}^{(n)}\|_{C_tL^2_x([0,t_0] \times \R^d)} &\lesssim M \label{E:ux0 H1}\\
\|u_{x_0}\|_{L^{\frac{p(d+1)}2}_{t,x}([0,t_0] \times \R^d)} + \|u^{(n)}_{x_0}\|_{L^{\frac{p(d+1)}2}_{t,x}([0,t_0] \times \R^d)} &< \eta, \label{E:uunx0 Lq}
\end{align}
for a small constant $\eta>0$ to be determined in a moment.  Throughout the remainder of the proof, all spacetime norms will be on $[0,t_0]\times \R^d$.

By H\"older and Sobolev embedding,
\begin{align*}
\|u_{x_0}\|_{L^{\frac{2(d+1)}{d-1}}_{t,x}} + \|u^{(n)}_{x_0}\|_{L^{\frac{2(d+1)}{d-1}}_{t,x}}
&\lesssim t_0^{\frac{d-1}{2(d+1)}}\Bigl(\|u_{x_0}\|_{C_tH^1_x} + \|u^{(n)}_{x_0}\|_{C_tH^1_x}\Bigr)\lesssim M.
\end{align*}
By Lemma~\ref{L:Strichartz} (applied with $m = 0$), \eqref{E:H12 conv after loc}, \eqref{E:ux0 H1}, \eqref{E:uunx0 Lq}, and H\"older's inequality,
\begin{align*}
\|\nabla_{t,x}(u^{(n)}_{x_0} - & u_{x_0})\|_{C_t\dot H^{-\frac12}_x} + \|u^{(n)}_{x_0} - u_{x_0}\|_{L^{\frac{2(d+1)}{d-1}}_{t,x}}\\
&\lesssim \|(u^{(n)}_{x_0,0}-u_{x_0,0}, u^{(n)}_{x_0,1} - u_{x_0,1})\|_{\dot H^{\frac12}_x \times \dot H^{-\frac12}_x}
 + \|m_n^2u^{(n)}_{x_0}\|_{L^1_t\dot H^{-\frac12}_x} \\ &\qquad \qquad + \||u^{(n)}_{x_0}|^pu^{(n)}_{x_0} - |u_{x_0}|^pu_{x_0}\|_{L^{\frac{2(d+1)}{d+3}}_{t,x}}\\
& \lesssim \eps_n + m_n^2 t_0 \|u^{(n)}_{x_0}\|_{C_t\dot H^1_x} + \eta^p\|u^{(n)}_{x_0} - u_{x_0}\|_{L^{\frac{2(d+1)}{d-1}}_{t,x}}\\
&\lesssim \eps_n + m_n^2M + \eta^p\|u^{(n)}_{x_0} - u_{x_0}\|_{L^{\frac{2(d+1)}{d-1}}_{t,x}},
\end{align*}
for some sequence $\eps_n \to 0$.  Thus, for $\eta$ sufficiently small,
\begin{equation} \label{E:un to u locally}
\|\nabla_{t,x}(u^{(n)}_{x_0} - u_{x_0})\|_{C_t\dot H^{-\frac12}_x} + \|u^{(n)}_{x_0} - u_{x_0}\|_{L^{\frac{2(d+1)}{d-1}}_{t,x}} \to 0 \qtq{as} n\to \infty.
\end{equation}

To conclude, by finite speed of propagation, the solution $u$ to \eqref{E:eqn} with $m=0$ exists and equals $u_{x_0}$ on
$\Gamma(1,x_0)\cap([0,t_0] \times \R^d)$ for each $x_0 \in \R^d$.  In particular, $u$ is a strong solution on $[0,t_0] \times \R^d$.
Additionally, $u^{(n)} = u^{(n)}_{x_0}$ on $\Gamma(T,x_0) \cap ([0,t_0] \times \R^d)$ for each $x_0 \in \R^d$.  Thus by \eqref{E:bounded sequences} and \eqref{E:un to u locally}, we obtain \eqref{E:un to u at t} for all $0\leq t\leq t_0$.  This completes the proof.
\end{proof}

%
%
%
%

\section{Blowup of non-positive energy solutions of NLW}\label{S:non+blow}

In this section we prove that non-positive energy solutions to the nonlinear wave equation blow up in finite time.  More precisely, we have

\begin{theorem}[Non-positive energy implies blowup] \label{T:nonpos blowup}
Let $\frac1{2d} < s_c < 1$ and set $p = \frac4{d-2s_c}$.  Let $(u_0,u_1) \in (\dot{H}^1_x \times L^2_x) \cap (\dot{H}^{s_c}_x \times \dot{H}^{s_c-1}_x)$ be initial data,
with $u_0$ and $u_1$ radial if $s_c < \frac12$.  Assume that $(u_0,u_1)$ is not identically zero and satisfies
$$
E(u_0,u_1) = \int_{\R^d} \tfrac12|\nabla u_0(x)|^2 + \tfrac12|u_1(x)|^2 - \tfrac1{p+2}|u_0(x)|^{p+2}\, dx \leq 0.
$$
Then the maximal-lifespan solution to the initial-value problem
$$
u_{tt} - \Delta u = |u|^p u \qtq{with} u(0) = u_0 \qtq{and} u_t(0) = u_1
$$
blows up both forward and backward in finite time.
\end{theorem}

We note that for solutions to \eqref{E:eqn} with $m>0$, finiteness of the energy dictates that $\|u_0\|_{L^2_x}$ also be finite.  Indeed, because of the estimate (cf. Lemma~\ref{L:GN})
\begin{align}\label{E:GN}
\|f\|_{p+2}^{p+2}\leq C_{\opt} \|f\|_{\frac{pd}2}^p \|\nabla f\|_2^2\lesssim \|f\|_{\dot H^{s_c}_x}^p \|f\|_{\dot H^1_x}^2,
\end{align}
the natural energy space for initial data is $(u_0,u_1) \in H^1_x \times L^2_x$.  The constant $C_{\opt}$ depends only on $d,p$ and denotes the optimal constant in the first inequality in \eqref{E:GN}.  In this case (that is, $u_0\in H^1_x$), the theorem is well-known and may be obtained by
taking two time derivatives of $\|u(t)\|_{L_x^2}$ (cf. \cite{Glassey73}, \cite{PayneSattinger}, and the proof of Proposition~\ref{P:finite mass}).  To handle data for which $\|u_0\|_{L_x^2}$ need not be finite, we adapt this argument by introducing a spatial truncation and then dealing with the resulting error terms.  The larger class of initial data considered here is dictated by the needs of Section~\ref{S:critical blowup}.

\begin{proof}  We define
$$
\phi(x) := \begin{cases}
1, &\qtq{if} |x| \leq 1;\\
1-2(|x|-1)^2, &\qtq{if} 1 \leq |x| \leq \tfrac32; \\
2(2-|x|)^2, &\qtq{if} \tfrac32 \leq |x| \leq 2; \\
0,&\qtq{if} |x| \geq 2,
\end{cases}
$$
and $\phi^c := 1-\phi$.

As $(u_0,u_1) \in (\dot{H}^1_x \times L^2_x) \cap (\dot H^{s_c}_x \times \dot H^{s_c-1}_x)$, there exists a radius $R > 0$ such that
$$
\|\phi^c\bigl(\tfrac{\cdot}{R/4}\bigr) u_0\|_{\dot H^{s_c}_x} + \|\phi^c\bigl(\tfrac{\cdot}{R/4}\bigr)u_1\|_{\dot H^{s_c-1}_x} \leq \eta_1,
$$
where $\eta_1$ is the small data threshold from Proposition~\ref{P:Hsc sdt}.  Let $v$ denote the global solution to
\begin{equation*}
v_{tt} - \Delta v = |v|^p v \qtq{with} v(0,x) = \phi^c\bigl(\tfrac{x}{R/4}\bigr)u_0(x) \qtq{and} v_t(0,x) = \phi^c\bigl(\tfrac{x}{R/4}\bigr) u_1(x).
\end{equation*}
By Proposition~\ref{P:Hsc sdt}, we may take $R$ sufficiently large that
\begin{equation} \label{E:v small}
\|\nabla_{t,x}v\|_{C_tL^2_x(\R \times \R^d)} + \|v\|_{C_t L^{\frac{pd}2}_x(\R \times \R^d)} < \eta
\end{equation}
for a small constant $\eta>0$ to be determined later.  By finite speed of propagation, $u = v$ where $|x| \geq R/2+|t|$, and so
\begin{equation} \label{E:u small}
\|\nabla_{t,x}u\|_{C_tL^2_x(\{|x| \geq R/2+|t|\})} + \|u\|_{C_t L^{\frac{pd}2}_x(\{|x| \geq R/2+|t|\})} < \eta.
\end{equation}

Next, since $E(u) \leq 0$ and $u$ is not identically zero, by \eqref{E:GN} we must have
\begin{align*}
0\geq E(u) &= \int_{\R^d}  \tfrac12|u_t(t,x)|^2 + \tfrac12|\nabla u(t,x)|^2 - \tfrac1{p+2}|u(t,x)|^{p+2}\, dx \\
&\geq \int_{\R^d} \tfrac12|u_t(t,x)|^2 + \tfrac12\bigl(1-\tfrac2{p+2}C_{\opt}\|u(t)\|_{L^{\frac{pd}2}_x}^p\bigr)|\nabla u(t,x)|^2\, dx
\end{align*}
and so,
$$
\|u(t)\|_{L^{\frac{pd}2}_x}^p \geq \tfrac{p+2}2 C_{\opt}^{-1}.
$$
Thus, using \eqref{E:u small} and taking $\eta$ small enough so that $\eta^p<\frac{p+2}4 C_{\opt}^{-1}$, we obtain
\begin{equation} \label{E:u pd/2 big}
\|u(t)\|_{L^{\frac{pd}2}_x(|x| \leq R/2+|t|)}^p \geq \tfrac{p+2}4 C_{\opt}^{-1}
\end{equation}
for each $t$ in the lifespan of $u$.

Finally, letting $\chi$ denote a smooth cutoff that is equal to one on the ball $\{|x|\leq R/2+|t|\}$ and vanishes when $|x|\geq R+2|t|$, and using Gagliardo--Nirenberg followed by Lemma~\ref{L:Poincare}, H\"older, and \eqref{E:u small}, we obtain
\begin{align*}
\|u(t)\|_{L^{\frac{pd}2}_x(|x|\leq \frac R2 + |t|)}& \lesssim \|\chi u(t)\|_{L_x^2}^{1-s_c} \|\nabla [\chi u(t)]\|_{L_x^2}^{s_c} \\
&\lesssim (R+|t|)^{1-s_c}\|\nabla [\chi u(t)]\|_{L_x^2}\\
&\lesssim (R+|t|)^{1-s_c} \Bigl[ \|\nabla u(t)\|_{L^2_x} + \|u(t)\|_{L^{\frac{pd}2}_x(|x|\geq \frac R2 + |t|)}\|\nabla \chi\|_{L_x^{\frac{2pd}{pd-4}}} \Bigr]\\
&\lesssim (R+|t|)^{1-s_c}\|\nabla u(t)\|_{L^2_x} +\eta.
\end{align*}
Thus, invoking \eqref{E:u pd/2 big} and taking $\eta$ sufficiently small, we obtain
\begin{equation} \label{E:u H1 big}
\|\nabla u(t)\|_{L^2_x}\gtrsim  (R+|t|)^{-1+s_c},
\end{equation}
throughout the lifetime of $u$.

Now we are ready to define our truncated `mass'.  We set
$$
M(t) := \int_{\R^d} \phi\bigl(\tfrac{x}{R+|t|}\bigr) |u(t,x)|^2\, dx.
$$
It is easy to see that this quantity is finite throughout the lifetime of $u$.  By \eqref{E:u pd/2 big}, it never vanishes, that is, $M(t)>0$ for all $t$ in the lifespan of $u$.

We differentiate.  Routine computations reveal that for $t \geq 0$, we have
\begin{align} \notag
&M'(t) = \int_{\R^d} -\tfrac{x}{(R+|t|)^2}\cdot \nabla \phi\bigl(\tfrac{x}{R+|t|} \bigr) |u(t)|^2\, dx + \int_{\R^d} 2\phi\bigl(\tfrac{x}{R+|t|}\bigr)u(t)u_t(t)\, dx, \\
 \label{E:M''}
&\begin{aligned}
M''(t) &= -2(p+2)E(u) + \int_{\R^d} 4 \phi\bigl(\tfrac{x}{R+|t|} \bigr) |u_t(t)|^2 + p|\nabla_{t,x}u(t)|^2\, dx \\
&\qquad + \int_{\R^d} 2\phi^c\bigl(\tfrac{x}{R+|t|} \bigr)\Bigl[|\nabla_{t,x}u(t)|^2- |u(t)|^{p+2}\Bigr]\, dx \\
&\qquad + \int_{\R^d} \left[\tfrac{2x}{(R+|t|)^3} \cdot  \nabla \phi\bigl(\tfrac{x}{R+|t|} \bigr) + \tfrac{x_ix_j}{(R+|t|)^4}\partial_i \partial_j \phi\bigl(\tfrac{x}{R+|t|} \bigr)\right] |u(t)|^2\, dx \\
&\qquad - \int_{\R^d} \tfrac2{R+|t|} \nabla\phi\bigl(\tfrac{x}{R+|t|} \bigr) \cdot \bigl[\tfrac{2x}{R+|t|} u_t(t) + \nabla u(t)\bigr]u(t) \, dx.
\end{aligned}
\end{align}
We will seek an upper bound for $|M'(t)|$ and a lower bound for $M''(t)$.  We will make repeated use of the following bound, which is a simple consequence of H\"older's inequality followed by \eqref{E:u small} and \eqref{E:u H1 big}:
\begin{align}\label{E:mass annulus}
\int_{R+|t| \leq |x| \leq 2(R+|t|)}\frac{|u(t,x)|^2}{(R+|t|)^2}\, dx
&\lesssim (R+|t|)^{-2(1-s_c)}\|u\|_{L^{\frac{pd}2}_x(|x|\geq R+|t|)}^2\notag\\
&\lesssim\eta^2 \|\nabla u(t)\|_{L^2_x}^2.
\end{align}

Using Cauchy--Schwarz and the inequality $|ab|^{\frac12} + |cd|^{\frac12} \leq (|a|+|c|)^{\frac12} (|b|+|d|)^{\frac12}$, we estimate
\begin{align*}
|M'(t)|
&\leq \biggl(\int_{\R^d} \tfrac{\eps}8 |\nabla \phi\bigl(\tfrac{x}{R+|t|}\bigr)|^2|u(t)|^2\, dx \biggr)^{\!\frac12}
    \biggl(\int_{R+|t| \leq |x| \leq 2(R+|t|)} \tfrac{8|x|^2}{\eps (R+|t|)^4} |u(t)|^2\, dx \biggr)^{\!\frac12} \\
&\quad + \biggl(\int_{\R^d}(1-\eps)\phi\bigl(\tfrac{x}{R+|t|}\bigr) |u(t)|^2\, dx\biggr)^{\!\frac12}
  \biggl( \int_{\R^d} \tfrac4{1-\eps}\phi\bigl(\tfrac{x}{R+|t|}\bigr)|u_t(t)|^2\, dx \biggr)^{\!\frac12}\\
&\leq \biggl(\int_{\R^d}\left[\tfrac\eps8 |\nabla \phi\bigl(\tfrac{x}{R+|t|}\bigr)|^2 + (1-\eps) \phi\bigl(\tfrac{x}{R+|t|}\bigr) \right] |u(t)|^2\, dx \biggr)^{\!\frac12} \\
&\quad \times \biggl( \int_{R+|t| \leq |x| \leq 2(R+|t|)} \tfrac{32}{\eps(R+|t|)^2} |u(t)|^2\, dx + \int_{\R^d} \tfrac4{1-\eps}\phi\bigl(\tfrac{x}{R+|t|}\bigr)|u_t(t)|^2\, dx \biggr)^{\!\frac12}
\end{align*}
for any $0 < \eps < 1$.  Since $\tfrac18 |\nabla \phi|^2 \leq \phi$, the first factor in the product above is bounded by $M(t)$.  Using
\eqref{E:mass annulus} to estimate the first term in the second factor gives
\begin{equation} \label{E:bound M'}
|M'(t)|^2 \leq M(t) \left( C_{\eps} \eta^2 \|\nabla u(t)\|_{L^2_x}^2 + \int_{\R^d} \tfrac4{1-\eps}\phi\bigl(\tfrac{x}{R+|t|}\bigr)|u_t(t,x)|^2\, dx \right).
\end{equation}

We now turn to $M''(t)$.  By \eqref{E:mass annulus},
\begin{align} \label{E:mass annulus'}
\Biggl| \int_{\R^d}\Bigl[  &  \tfrac{2x}{(R+|t|)^3}\cdot\nabla \phi\bigl(\tfrac{x}{R+|t|}\bigr)+\tfrac{x_ix_j}{(R+|t|)^4}\partial_i\partial_j\phi\bigl(\tfrac{x}{R+|t|} \bigr)\Bigr] |u(t,x)|^2\, dx \Biggr|\notag\\
&\qquad\lesssim \int_{R+|t| \leq |x| \leq 2(R+|t|)}\tfrac1{(R+|t|)^2}|u(t,x)|^2\, dx \lesssim \eta^2 \|\nabla u(t)\|_{L^2_x}^2.
\end{align}
Next, by Lemma~\ref{L:GN} and \eqref{E:u small},
\begin{align} \label{E:p+2 annulus}
\int_{\R^d} \phi^c\bigl(\tfrac{x}{R+|t|}\bigr)|u(t,x)|^{p+2}\, dx
&\leq \|u(t)\|_{L_x^{p+2}(|x|\geq R+|t|)}^{p+2}\notag\\
&\lesssim\|u(t)\|_{L^{\frac{pd}2}_x(|x|\geq R+|t|)}^p\|\nabla u(t)\|_{L^2_x(|x|\geq R+|t|)}^2\notag\\
&\lesssim \eta^p \|\nabla u(t)\|_{L^2_x}^2.
\end{align}
Finally, by Young's inequality and \eqref{E:mass annulus},
\begin{align} \label{E:u nabla u annulus}
\Biggl|\int_{\R^d} \tfrac2{R+|t|} & \nabla\phi\bigl(\tfrac{x}{R+|t|}\bigr) \cdot \Bigl[\tfrac{2x}{R+|t|} u_t(t,x) + \nabla u(t,x)\Bigr]u(t,x) \, dx\Biggr|\notag\\
&\leq \int_{R+|t| \leq |x| \leq 2(R+|t|)} \tfrac{C_{\eps}}{(R+|t|)^2} |u(t,x)|^2\, dx + \int_{\R^d} \eps |\nabla_{t,x} u(t,x)|^2\, dx \notag\\
& \leq C_{\eps} \eta^2 \|\nabla u(t)\|_{L^2_x}^2 + \eps \|\nabla_{t,x} u(t)\|_{L^2_x}^2,
\end{align}
for any $\eps > 0$.

Now let $\delta > 0$.  Combining  $E(u)\leq 0$, \eqref{E:mass annulus'}, \eqref{E:p+2 annulus}, and \eqref{E:u nabla u annulus} with the identity \eqref{E:M''}, and choosing $\eps=\eps(\delta)$ and then $\eta=\eta(\eps)$ sufficiently small, we obtain
\begin{equation}\label{E:M'' lb 1}
\begin{aligned}
M''(t) &\geq \int_{\R^d} 4 \phi\bigl(\tfrac{x}{R+|t|}\bigr) |u_t(t,x)|^2 \, dx + (p-\delta) \int_{\R^d} |\nabla_{t,x} u(t,x)|^2\, dx \\
& \geq \int_{\R^d} \tfrac4{1-2\eps} \phi\bigl(\tfrac{x}{R+|t|}\bigr) |u_t(t,x)|^2\, dx + (p-\delta) \int_{\R^d} |\nabla u(t,x)|^2\, dx.
\end{aligned}
\end{equation}

Combining \eqref{E:bound M'} and \eqref{E:M'' lb 1}, we get
\begin{equation} \label{E:M'MM''}
|M'(t)|^2 \leq cM(t) M''(t),
\end{equation}
for some constant $0 < c < 1$.  Using this we will prove that $u$ blows up in finite time, forward in time; finite-time blowup backward in time follows from time-reversal symmetry.  We argue by contradiction.  Suppose that the solution $u$ may be continued forward in time indefinitely.

First, we consider the case when $M'(0) > 0$.  By \eqref{E:u H1 big} and \eqref{E:M'' lb 1}, $M''(t) > 0$ for all $t$ in the lifespan of $u$, and so $M'(t)>0$ for all $t>0$.  Thus by \eqref{E:bound M'},
$$
\frac{M'(t)}{M(t)} \leq c\frac{M''(t)}{M'(t)}.
$$
Integrating both sides, we see that for $t \geq 0$ we have
$$
\log \left( \frac{M(t)}{M(0)} \right) \leq c \log \left(\frac{M'(t)}{M'(0)} \right),
$$
that is,
$$
M'(0)M(0)^{-1/c} \leq M'(t)M(t)^{-1/c}.
$$
Integrating a second time and recalling that $M(t)> 0$, we obtain
$$
t M'(0) M(0)^{-1/c} \leq \frac{c}{1-c}(M(0)^{1-1/c} - M(t)^{1-1/c}) \leq \frac{c}{1-c}M(0)^{1-\frac1c}.
$$
But this is impossible since the left-hand side grows linearly as $t \to \infty$, while the right-hand side is bounded.  Thus we must have $M'(0) \leq 0$.

More generally, if we suppose that $M'(t_0) \geq 0$ for any $t_0$ in the lifespan of $u$, then $M'(t) > 0$ for all $t > t_0$, since $M''(t)>0$ for all $t$ in the lifespan of $u$ (by \eqref{E:u H1 big} and \eqref{E:M'' lb 1}). Arguing as above, we again obtain a contradiction to indefinite forward in time existence of $u$.

Thus we may assume that $M'(t) < 0$ for as long as $u$ exists, and therefore
\begin{equation} \label{E:mass bounded}
0 < M(t) < M(0) \qtq{for all} t > 0.
\end{equation}
Furthermore, as $M'(t)$ stays negative, we must have by \eqref{E:M'' lb 1} that
$$
|M'(0)| \geq \int_0^{\infty} M''(t) \,dt \gtrsim \int_0^{\infty} \|\nabla_{t,x} u(t)\|_{L^2_x}^2\, dt.
$$
From this, we see that along some sequence $t_n \to \infty$, we have
\begin{equation} \label{E:H1 to 0}
 \|\nabla u(t_n)\|_{L^2_x}^2 \to 0.
\end{equation}
Next, using Gagliardo--Nirenberg followed by H\"older, \eqref{E:u small}, \eqref{E:u H1 big}, and \eqref{E:mass bounded}, we obtain
\begin{align*}
\|&u(t_n)\|_{L^{\frac{pd}2}_x(|x| \leq R+t_n)}\\
&\lesssim \bigl\|\phi(\tfrac{\cdot}{R+t_n}) u(t_n)\bigr\|_{L^2_x}^{1-s_c}\bigl\|\nabla\bigl[\phi(\tfrac{\cdot}{R+t_n}) u(t_n)\bigr]\bigr\|_{L^2_x}^{s_c}\\
&\lesssim M(t_n)^{\frac{1-s_c}2} \Bigl[\|\nabla u(t_n)\|_{L^2_x} + \|u(t_n)\|_{L^{\frac{pd}2}_x(|x| \geq R+t_n)} \|\nabla \phi(\tfrac{\cdot}{R+t_n})\|_{L_x^{\frac{2pd}{pd-4}}}\Bigr]^{s_c}\\
&\lesssim M(0)^{\frac{1-s_c}2} \Bigl[\|\nabla u(t_n)\|_{L^2_x} + \eta (R+t_n)^{-1+s_c}\Bigr]^{s_c}\\
&\lesssim M(0)^{\frac{1-s_c}2} (1+\eta)^{s_c}\|\nabla u(t_n)\|_{L^2_x}^{s_c} \to 0 \quad \text{as}\quad n\to \infty,
\end{align*}
which contradicts \eqref{E:u pd/2 big}.

This completes the proof of the theorem.
\end{proof}

%
%
%
%

\section{Concentration compactness for a Gagliardo--Nirenberg inequality}\label{S:CC}

In this section we develop a concentration compactness principle associated to the following Gagliardo--Nirenberg inequality (cf. Lemma~\ref{L:GN})
\begin{align}\label{E:GN1}
\|f\|_{L_x^{p+2}}^{p+2} \lesssim \|f\|_{\dot H^{s_c}_x}^p \|f\|_{\dot H^1_x}^2.
\end{align}
More precisely, we prove

\begin{theorem}[Bubble decomposition for \eqref{E:GN1}] \label{T:bubble gn}  Fix a dimension $d \geq 2$ and an exponent $\frac4d \leq p < \frac4{d-2}$.  Let $s_c = \frac{d}2 - \frac2p$.  Let $\{f_n\}_{n\geq 1}$ be a bounded sequence in $\dot H^1_x \cap \dot
H^{s_c}_x$.  Then there exist $J^* \in \{0,1,\ldots\} \cup \{\infty\}$, nonzero functions $\{\phi^j\}_{j=1}^{J^*} \subset \dot H^1_x \cap \dot H^{s_c}_x$, $\{x_n^j\}_{j=1}^{J^*} \subset \R^d$, and a subsequence of $\{f_n\}_{n\geq 1}$ such that along this subsequence
\begin{equation} \label{E:bubble gn}
f_n(x) = \sum_{j=1}^J \phi^j(x-x_n^j) + r_n^J(x) \qtq{for each}  0 \leq J < J^*+1 ,
\end{equation}
with
\begin{equation}\label{E:rnJ wkly to 0}
r_n^J(\cdot + x_n^J) \rightharpoonup 0 \qtq{weakly in} \dot H^1_x \cap \dot H^{s_c}_x \qtq{for each} 1 \leq J < J^*+1.
\end{equation}
  Furthermore, along this subsequence, the $r_n^J$ satisfy
\begin{align}
\label{E:rnJ to 0}
\lim_{J \to J^*} \limsup_{n \to \infty} \|r_n^J\|_{L^{p+2}_x} = 0,
\end{align}
and for each $0 \leq J < J^*+1$, we have the following:
\begin{align}
\label{E:H1 decoup}
&\lim_{n \to \infty} \Bigl\{\|f_n\|_{\dot{H}_x^1}^2 - \Bigl[\sum_{j=1}^J \|\phi^j\|_{\dot{H}_x^1}^2 + \|r_n^J\|_{\dot{H}_x^1}^2 \Bigr]\Bigr\} = 0 \\
\label{E:Hsc decoup}
&\lim_{n \to \infty} \Bigl\{\|f_n\|_{\dot{H}_x^{s_c}}^2 - \Bigl[\sum_{j=1}^J \|\phi^j\|_{\dot{H}_x^{s_c}}^2 + \|r_n^J\|_{\dot{H}_x^{s_c}}^2 \Bigr]\Bigr\}  = 0 \\
\label{E:p+2 decoup}
&\lim_{n \to \infty} \Bigl\{\|f_n\|_{L^{p+2}_x}^{p+2} - \Bigl[\sum_{j=1}^J \|\phi^j\|_{L^{p+2}_x}^{p+2} + \|r_n^J\|_{L^{p+2}_x}^{p+2}\Bigr]\Bigr\} = 0.
\end{align}
Finally, for each $j' \neq j$, we have
\begin{equation} \label{E:GN orthog}
\lim_{n \to \infty} |x_n^j - x_n^{j'}| = \infty.
\end{equation}
\end{theorem}

There are many results of this type, beginning with the work \cite{Solimini} of Solimini on Sobolev embedding.
The argument below is modeled on the treatment in \cite{ClayNotes}, with the main step being the following inverse inequality.

\begin{prop}[Inverse Gagliardo--Nirenberg inequality]  \label{P:inverse gn}
Fix a dimension $d \geq 2$ and an exponent $\frac4d \leq p < \frac4{d-2}$.  Let $s_c = \frac{d}2 - \frac2p$.  Let $\{f_n\}_{n\geq 1} \subset \dot{H}^1_x(\R^d) \cap \dot{H}_x^{s_c}(\R^d)$ and assume that
$$
\limsup_{n \to \infty} \|f_n\|_{\dot H_x^{s_c}}^2 + \|f_n\|_{\dot H_x^1}^2 = M^2 \qtq{and} \liminf_{n \to \infty} \|f_n\|_{L^{p+2}_x} = \eps > 0.
$$
Then there exist $\phi \in \dot H_x^1(\R^d) \cap \dot H_x^{s_c}(\R^d)$ and $\{x_n\}_{n\geq 1} \subset \R^d$ such that after passing to a subsequence, we have the following:
\begin{align}
\label{E:fn wkly to phi}
f_n(\cdot + x_n) \rightharpoonup \phi \qtq{weakly in} &\dot H^1_x \cap \dot H^{s_c}_x\\
\label{E:inv st H1}
\lim_{n \to \infty} \Bigl\{ \|f_n\|_{\dot H^1_x}^2 - \|f_n - \phi(\cdot - x_n)\|_{\dot H^1_x}^2 \Bigr\} &= \|\phi\|_{\dot H^1_x}^2 \gtrsim \eps^2\bigl(\frac{\eps}{M}\bigr)^{\alpha_1} \\
\label{E:inv st Hsc}
\lim_{n \to \infty}  \Bigl\{ \|f_n\|_{\dot H^{s_c}_x}^2 - \|f_n - \phi(\cdot - x_n)\|_{\dot H^{s_c}_x}^2 \Bigr\} &= \|\phi\|_{\dot H^{s_c}_x}^2 \gtrsim \eps^2\bigl(\frac{\eps}{M}\bigr)^{\alpha_2} \\
\label{E:inv st p+2}
\lim_{n \to \infty} \Bigl\{\|f_n\|_{L^{p+2}_x}^{p+2} - \|f_n - \phi(\cdot-x_n)\|_{L^{p+2}_x}^{p+2} \Bigr\} &= \|\phi\|_{L^{p+2}_x}^{p+2} \gtrsim \eps^{p+2}\bigl(\frac{\eps}{M}\bigr)^{\alpha_3},
\end{align}
for certain positive constants $\alpha_1,\alpha_2,\alpha_3$ depending on $d,p$.
\end{prop}

\begin{proof}[Proof of Proposition \ref{P:inverse gn}]
By passing to a subsequence, we may assume that
\begin{equation} \label{E:control norms fn}
\|f_n\|_{\dot H^1_x}^2 + \|f_n\|_{\dot H^{s_c}_x}^2 \leq 2M^2 \qtq{and} \|f_n\|_{L^{p+2}_x} \geq \tfrac\eps2 \qtq{for all $n$.}
\end{equation}
Note that by \eqref{E:GN1} we must have $\eps \lesssim M$.

Now by \eqref{E:GN1}, \eqref{E:control norms fn}, and Bernstein's inequality, for all dyadic frequencies $N$ we have
\begin{align*}
\|P_N f_n\|_{L^{p+2}_x}^{p+2} &\lesssim \|P_Nf_n\|_{\dot H^{s_c}_x}^p \|P_N f_n\|_{\dot H^1_x}^2 \lesssim \min\{N^{-p(1-s_c)}, N^{2(1-s_c)}\}M^{p+2}.
\end{align*}
Thus, if we define
$$
K = C\bigl(M\eps^{-1}\bigr)^{\frac{p+2}{2p(1-s_c)}}
$$
for a suitably large constant $C$, we obtain
$$
\|P_{\leq K^{-p}} f_n\|_{L^{p+2}_x}^{p+2} + \|P_{\geq K^2} f_n\|_{L^{p+2}_x}^{p+2} \ll \eps^{p+2}.
$$
Hence, by the pigeonhole principle there exist dyadic frequencies $N_n$ satisfying $K^{-p} \leq N_n \leq K^2$ such that
\begin{equation} \label{E:lb fN p+2}
 (\log K)^{-1} \eps \lesssim \|P_{N_n} f_n\|_{L^{p+2}_x}.
\end{equation}
By passing to a subsequence, we may assume that $N_n = N$ for all $n$.

By H\"older's inequality, the Sobolev embedding $\dot H^{s_c}_x \hookrightarrow L^{\frac{pd}2}_x$, and \eqref{E:control norms fn},
$$
\|P_N f_n\|_{L^{p+2}_x} \leq \|P_N f_n\|_{L^{\infty}_x}^{1-\frac{pd}{2(p+2)}}\|P_Nf_n\|_{L^{\frac{pd}2}_x}^{\frac{pd}{2(p+2)}} \lesssim M^{\frac{pd}{2(p+2)}} \|P_N f_n\|_{L^{\infty}_x}^{1-\frac{pd}{2(p+2)}},
$$
and so by \eqref{E:lb fN p+2}, there exists a sequence $\{x_n\} \subset \R^d$ such that
$$
\bigl(\eps^2 M^{-1-\frac{pd}{2(p+2)}}\bigr)^{\frac{p+2}{p(1-s_c)}}     \lesssim \left(\frac{\eps M^{-\frac{pd}{2(p+2)}}}{\log K}\right)^{\frac{2(p+2)}{4-p(d-2)}} \lesssim |P_N f_n(x_n)|.
$$

We consider the sequence $f_n(\cdot+x_n)$.  This sequence is bounded in $\dot H^1_x(\R^d) \cap \dot H^{s_c}_x(\R^d)$ by \eqref{E:control norms fn}, and so, after passing to a subsequence, there exists a weak limit $\phi \in \dot H^1_x(\R^d) \cap \dot H^{s_c}_x(\R^d)$ as in \eqref{E:fn wkly to phi}.  The equalities in \eqref{E:inv st H1} and \eqref{E:inv st Hsc} are immediate.  We now turn to the $L^{p+2}_x$ decoupling, \eqref{E:inv st p+2}.  By \eqref{E:fn wkly to phi} and the Rellich--Kondrashov theorem, $f_n(\cdot + x_n) \to \phi$ in $L^2_\loc$ and hence, after passing to a subsequence, almost everywhere.  The equality in \eqref{E:inv st p+2} is then an immediate consequence of the Fatou lemma of Br\'ezis and Lieb; see \cite{BrezisLieb} or \cite{LiebLoss}.

Finally, to obtain the lower bounds in \eqref{E:inv st H1}, \eqref{E:inv st Hsc}, and \eqref{E:inv st p+2}, we test $\phi$ against the function $k = P_N \delta_0$ ($\delta_0$ being the Dirac delta).  We have
$$
|\langle \phi , k \rangle| = \lim_{n \to \infty} \Bigl|\int_{\R^d} f_n(x +x_n) \overline{k(x)}\, dx\Bigr| = \lim_{n \to \infty} | P_N f_n(x_n)| \gtrsim \bigl(\eps^2 M^{-1-\frac{pd}{2(p+2)}}\bigr)^{\frac{p+2}{p(1-s_c)}}.
$$
Routine computations reveal that
$$
\|k\|_{\dot H^{-s_c}_x} \sim N^{\frac d2-s_c}, \qquad \|k\|_{\dot H^{-1}_x} \sim N^{\frac d2-1}, \qquad \|k\|_{L^{\frac{p+2}{p+1}}_x} \sim N^{\frac d{p+2}},
$$
and since $K^{-p} \leq N \leq K^2$, the lower bounds follow.

This completes the proof.
\end{proof}

We are now ready to prove Theorem~\ref{T:bubble gn}.

\begin{proof}[Proof of Theorem~\ref{T:bubble gn}]  To begin, we set $r_n^0 := f_n$.  The identities \eqref{E:H1 decoup}, \eqref{E:Hsc decoup}, and \eqref{E:p+2 decoup} with $J=0$ are thus trivial.  After passing to a subsequence, we may assume that
$$
\lim_{n \to \infty} \|r_n^0\|_{\dot H^1_x}^2 + \|r_n^0\|_{\dot H^{s_c}_x}^2 = M_0^2 \qtq{and} \lim_{n \to \infty} \|r_n^0\|_{L^{p+2}_x} = \eps_0.
$$
By hypothesis $M_0 < \infty$, while by \eqref{E:GN} we have $\eps_0 \lesssim M_0$.

We now proceed inductively, assuming that a decomposition satisfying \eqref{E:bubble gn}, \eqref{E:rnJ wkly to 0}, \eqref{E:H1 decoup}, \eqref{E:Hsc decoup}, and \eqref{E:p+2 decoup} has been carried out up to some integer $J \geq 0$ and that the remainder satisfies
$$
\lim_{n \to \infty} \|r_n^J\|_{\dot H^1_x}^2 + \|r_n^J\|_{\dot H^{s_c}_x}^2 = M_J^2 \qtq{and} \lim_{n \to \infty}\|r_n^J\|_{L^{p+2}_x} = \eps_J,
$$
with $\eps_J \lesssim M_J$ (which follows from \eqref{E:GN}) and $M_J < M_0$ (which will be established below).

If $\eps_J  = 0$, we stop, setting $J^* = J$.  The relations \eqref{E:bubble gn} through \eqref{E:p+2 decoup} have thus been established; we will come to \eqref{E:GN orthog} in a moment.

If $\eps_J > 0$, we apply Proposition~\ref{P:inverse gn} to $\{r_n^J\}$, producing a sequence $\{x_n^{J+1}\}$ of points and a (subsequential) weak limit
$$
r_n^J(\cdot + x_n^{J+1}) \rightharpoonup \phi^{J+1} \qtq{in} \dot H^1_x \cap \dot H^{s_c}_x.
$$
Setting $r_n^{J+1} := r_n^J - \phi^{J+1}(\cdot - x_n^{J+1})$, we obtain \eqref{E:bubble gn} and \eqref{E:rnJ wkly to 0} with $J$ replaced by $J+1$.  The identity in \eqref{E:inv st H1} is just
$$
\lim_{n \to \infty} \Bigl\{\|r_n^J\|_{\dot H^1_x}^2 - \Bigl[\|\phi^{J+1}\|_{\dot H^1_x}^2 + \|r_n^{J+1}\|_{\dot H^1_x}^2\Bigr]\Bigr\} = 0.
$$
Adding this to \eqref{E:H1 decoup} shows that this continues to hold at $J+1$:
$$
\lim_{n \to \infty} \Bigl\{\|f_n\|_{\dot H^1_x}^2 - \Bigl[\sum_{j=1}^{J+1} \|\phi^j\|_{\dot H^1_x}^2 + \|r_n^{J+1}\|_{\dot H^1_x}^2\Bigr]\Bigr\} = 0.
$$
To derive \eqref{E:Hsc decoup} and \eqref{E:p+2 decoup} with $J$ replaced by $J+1$, one argues similarly.

Passing to a subsequence and applying \eqref{E:inv st H1}, \eqref{E:inv st Hsc}, and \eqref{E:inv st p+2}, we obtain
\begin{align}\label{E:Sobolev decrement}
M_{J+1}^2 &= \lim_{n \to \infty} \Bigl[\|r_n^{J+1}\|_{\dot H^1_x}^2 + \|r_n^{J+1}\|_{\dot H^{s_c}_x}^2\Bigr] \notag\\
&= \lim_{n \to \infty} \Bigl[\|r_n^J\|_{\dot H^1_x}^2 - \|\phi^{J+1}\|_{\dot H^1_x}^2 +  \|r_n^J\|_{\dot H^{s_c}_x}^2 - \|\phi^{J+1}\|_{\dot H^{s_c}_x}^2\Bigr] \notag\\
&\leq M_J^2 - C\eps_J^2\Bigl[\bigl(\tfrac{\eps_J}{M_J}\bigr)^{\alpha_1} + \bigl(\tfrac{\eps_J}{M_J}\bigr)^{\alpha_2}\Bigr]
\end{align}
and
\begin{align}\label{E:p+2 decrement}
\eps_{J+1}^{p+2} =  \lim_{n \to \infty} \|r_n^{J+1}\|_{L^{p+2}_x}^{p+2}
= \lim_{n \to \infty} \Bigl[ \|r_n^J\|_{L^{p+2}_x}^{p+2} - \|\phi^J\|_{L^{p+2}_x}^{p+2} \Bigr]
\leq \eps_J^{p+2} - C\eps_J^{p+2}\bigl(\tfrac{\eps_J}{M_J}\bigr)^{\alpha_3}.
\end{align}

Either this process eventually stops and we obtain some finite $J^*$ or we set $J^* = \infty$.  If we do have $J^* = \infty$, then \eqref{E:rnJ to 0} follows from
\eqref{E:Sobolev decrement} and \eqref{E:p+2 decrement}.

Finally, we prove the asymptotic orthogonality \eqref{E:GN orthog}.  Let us suppose that this fails for some $j \neq j'$.  We may assume that $j'>j$ and that
$\lim_{n \to\infty}|x_n^j - x_n^{k}| = \infty$ for $j < k < j'$.  Passing to a subsequence, we may assume that $\lim_{n \to \infty} (x_n^j - x_n^{j'}) = y$.  We recall that
$$
\phi^{j'} = \wklim_{n \to \infty} \, r_n^{j'-1}(\cdot + x_n^{j'}),
$$
while
$$
r_n^{j'-1} = r_n^j - \sum_{k=j+1}^{j'-1}\phi^{k}(\cdot - x_n^{k}).
$$
Therefore,
\begin{align*}
\phi^{j'} = \wklim_{n \to \infty} \Bigl\{  r_n^j(\cdot + x_n^{j'}) - \sum_{k=j+1}^{j'-1} \phi^{k}(\cdot + x_n^{j'} - x_n^{k}) \Bigr\}
&= \wklim_{n \to \infty} \, r_n^j(\cdot + x_n^j + y) = 0,
\end{align*}
where we used $\lim_{n \to \infty} |x_n^{j'} - x_n^{k}| = \infty$ for all $j < k < j'$ in order to derive the second equality and \eqref{E:rnJ wkly to 0} to derive the third equality.  But $\phi^{j'}$ cannot be $0$ in view of \eqref{E:inv st H1} and the fact that our inductive procedure stops once we obtain $\eps_J = 0$.

This completes the proof of Theorem~\ref{T:bubble gn}.
\end{proof}

%
%
%
%

\section{Blowup of the critical norm}\label{S:critical blowup}

The goal of this section is to show that finite-time blowup of solutions to \eqref{E:eqn} is accompanied by subsequential blowup of their critical Sobolev norm.
More precisely, we prove the following result which is slightly more general than Theorem~\ref{T:I:sc} given in the Introduction.

\begin{theorem} \label{T:liminf v2}  Let $d \geq 2$, $m \in [0,1]$, and $\frac1{2d} < s_c < 1$.  Set $p = \frac4{d-2s_c}$.  Let $(u_0,u_1)$ be initial data for \eqref{E:eqn} satisfying
$$
\|\jpn_m u_0\|_{L^2_x} + \|u_1\|_{L^2_x} + \||\nabla|^{s_c}u_0\|_{L^2_x} \leq M<\infty,
$$
with $u_0$ and $u_1$ radial if $s_c < \frac12$.  Assume that the maximal-lifespan solution $u$ to \eqref{E:eqn} blows up forward in time at $0<T_* < \infty$.  Then
$$
\limsup_{t \uparrow T_*} \bigl\{\|u(t)\|_{\dot H^{s_c}_x} + \|u_t(t)\|_{H^{s_c-1}_x}\bigr\} = \infty.
$$
\end{theorem}

The remainder of the section is dedicated to the proof of the theorem.  We assume by way of contradiction that
\begin{equation} \label{E:lie}
\|u\|_{L^{\infty}_t \dot H^{s_c}_x([0,T_*)\times \R^d)} + \|u_t\|_{L^{\infty}_t H^{s_c-1}_x([0,T_*) \times \R^d)} = K < \infty.
\end{equation}
By Corollary~\ref{C:ee blowup},
$$
\|\nabla_{t,x}u(t)\|_{L^2_x} \to \infty \qtq{as} t\to T_*.
$$
Thus we may choose a sequence of times $\{t_n\}_{n\geq 1}$ increasing to $T_*$ and satisfying
$$
\|\nabla_{t,x}u(t_n)\|_{L^2_x} = \|\nabla_{t,x} u\|_{C_tL^2_x([0,t_n] \times \R^d)} \to \infty \qtq{as} n \to \infty.
$$
Let
$$
\lambda_n := (\|\nabla_{t,x} u(t_n)\|_{L^2_x})^{-\frac1{1-s_c}} \to 0
$$
and define
$$
u^{(n)}(t,x) := \lambda_n^{\frac2p}u(t_n - \lambda_n t, \lambda_n x) \qtq{for all} (t,x) \in [0,T_n] \times \R^d,
$$
where $T_n := \frac{t_n}{\lambda_n} \to \infty$.  Then $u^{(n)}$ solves
\begin{equation} \label{E:un soln}
u^{(n)}_{tt} - \Delta u^{(n)} + m_n^2 u^{(n)} = |u^{(n)}|^p u^{(n)}
\end{equation}
on $[0,T_n] \times \R^d$ with $m_n := \lambda_n m \to 0$.  Furthermore, by our choice of $t_n$ and $\lambda_n$, $u^{(n)}$ satisfies
\begin{align} \label{E:un H1}
\|\nabla_{t,x} u^{(n)}\|_{C_tL^2_x([0,T_n] \times \R^d)} = \|\nabla_{t,x} u^{(n)}(0)\|_{L^2_x} = 1
\end{align}
and
\begin{equation} \label{E:un Hsc}
\begin{aligned}
&\|u^{(n)}\|_{C_t \dot H^{s_c}_x([0,T_n] \times \R^d)} + \|\langle \nabla\rangle_{\lambda_n}^{s_c-1}\partial_t u^{(n)}\|_{C_t L_x^2([0,T_n] \times \R^d)}\\
&\qquad \qquad = \|u\|_{C_t\dot H^{s_c}_x([0,t_n] \times \R^d)} + \|u_t\|_{C_t H^{s_c-1}_x([0,t_n] \times \R^d)} \leq K.
\end{aligned}
\end{equation}
(We note that the subscript $\lambda_n$ is essential in \eqref{E:un Hsc} because we need the weak limit of $\partial_t u^{(n)}(0)$ to belong to $\dot H^{s_c-1}_x$; cf. Lemma~\ref{L:weak stability}.)  Finally, by conservation of energy,
\begin{align} \label{E:un nrg}
\int_{\R^d} \tfrac12|&\nabla_{t,x} u^{(n)}(0,x)|^2 - \tfrac1{p+2} |u^{(n)}(0,x)|^{p+2}\, dx \notag\\
&= \lambda_n^{2-2s_c} \int_{\R^d} \tfrac12|\nabla_{t,x}u(t_n,x)|^2 - \tfrac1{p+2}|u(t_n,x)|^{p+2}\, dx \notag\\
&\leq \lambda_n^{2-2s_c}  \int_{\R^d} \tfrac12|\nabla_{t,x} u(0,x)|^2 + \tfrac{m^2}2 |u(0,x)|^2 - \tfrac1{p+2} |u(0,x)|^{p+2}\, dx \to 0.
\end{align}
Using this and \eqref{E:un H1}, for sufficiently large $n$ we obtain
\begin{equation} \label{E:p+2 lb}
\int_{\R^d} |u^{(n)}(0,x)|^{p+2}\, dx \geq 1.
\end{equation}

To continue, we will use Lemma~\ref{L:weak stability} and Theorem~\ref{T:bubble gn} to prove that under the assumption \eqref{E:lie}, we have the following:

\begin{lemma} \label{L:bad limit}
The sequence $\{u^{(n)}\}_{n\geq 1}$ gives rise to a nonzero solution $w$ to the nonlinear wave equation \eqref{E:eqn} with $m=0$ which is global forward in time, satisfies $(w,w_t) \in C_t ([0, \infty);\dot H^1_x \times L^2_x\cap \dot H^{s_c}_x \times \dot H^{s_c-1}_x)$, and has $E(w) \leq 0$.
\end{lemma}

By Theorem~\ref{T:nonpos blowup}, a solution $w$ as described in Lemma~\ref{L:bad limit} cannot exist.  The restriction to radial data for
$s_c<\frac12$ arises only from the use of this theorem.  Thus, in order to conclude the proof of Theorem~\ref{T:liminf v2}, it remains to prove
Lemma~\ref{L:bad limit}.

To prove the lemma, we treat the sub- and super-conformal cases separately.  We start with the sub-conformal case, where the radial assumption on the initial data allows for a simpler treatment.

\begin{proof}[Proof of Lemma~\ref{L:bad limit} when $s_c < \frac12$]
By hypothesis, in this case we have that $u_0$ and $u_1$ are radial, and so the $u^{(n)}$ are radial also.  Using \eqref{E:un H1} and \eqref{E:un Hsc} and passing to a subsequence, we obtain a weak limit
\begin{equation} \label{E:un wkly to w}
(u^{(n)}(0),u^{(n)}_t(0)) \rightharpoonup (w_0,w_1) \qtq{in $[\dot H^1_x \cap \dot H^{s_c}_x] \times L^2_x$.}
\end{equation}
Additionally, by \eqref{E:un Hsc}, $w_1 \in \dot H^{s_c-1}_x$.  As the embedding $\dot H^1_{\rm{rad}} \cap \dot H^{s_c}_{\rm{rad}} \hookrightarrow L^{p+2}_{\rm{rad}}$ is compact, \eqref{E:un wkly to w} dictates
\begin{equation} \label{E:un to w p+2}
u^{(n)}(0) \to w_0 \qtq{strongly in} L^{p+2}_x.
\end{equation}
By \eqref{E:p+2 lb}, $w_0$ is not identically 0.  Finally, by \eqref{E:un nrg}, \eqref{E:un wkly to w}, and \eqref{E:un to w p+2}, we have
\begin{align*}
E(w_0,w_1) &= \int_{\R^d} \tfrac12 |\nabla w_0(x)|^2 + \tfrac12|w_1(x)|^2 - \tfrac1{p+2}|w_0(x)|^{p+2}\, dx \\
&\leq \lim_{n \to \infty} \int_{\R^d} \tfrac12|\nabla_{t,x}u^{(n)}(0,x)|^2 - \tfrac1{p+2}|u^{(n)}(0,x)|^{p+2}\, dx \leq 0.
\end{align*}

Now let $w$ be the solution to \eqref{E:eqn} with $m=0$ and initial data $(w_0,w_1)$ at time $t=0$.  By Lemma~\ref{L:weak stability} and the fact that $T_n \to \infty$, we obtain
that $w$ is global forward in time.  Moreover it satisfies $(w,w_t) \in C_t ([0, \infty);\dot H^1_x \times L^2_x\cap \dot H^{s_c}_x \times \dot H^{s_c-1}_x)$ and $E(w) \leq 0$.
This completes the proof of the lemma in the sub-conformal case.
\end{proof}

It remains to prove Lemma~\ref{L:bad limit} in the conformal and super-conformal cases; in this setting, we will substitute Theorem~\ref{T:bubble gn} for the compact radial embedding used in the sub-conformal case.

\begin{proof}[Proof of Lemma~\ref{L:bad limit} when $s_c \geq \frac12$]  Applying Theorem~\ref{T:bubble gn} to $\{u^{(n)}(0)\}_{n\geq 1}$ and passing to a subsequence, we obtain the decomposition
$$
u^{(n)}(0) = \sum_{j=1}^J \phi^j_0(\cdot - x_n^j) + r_n^J \qtq{for all} 0 \leq J < J^*+1,
$$
satisfying the conclusions of that theorem.  By \eqref{E:p+2 lb} and \eqref{E:rnJ to 0} we must have $J^* \geq 1$.  Using \eqref{E:rnJ wkly to 0} followed by
\eqref{E:GN orthog}, for each $j$ we have
\begin{equation} \label{E:phiJ0}
\phi^j_0 = \wklim_{n \to \infty} \Bigl\{u^{(n)}(0,\cdot + x_n^j) - \sum_{k=1}^{j-1} \phi_0^k(\cdot - x_n^k + x_n^j)\Bigr\} = \wklim_{n \to \infty} u^{(n)}(0,\cdot + x_n^j) ,
\end{equation}
where the weak limit is taken in $\dot H^1_x \cap \dot H^{s_c}_x$.  Using \eqref{E:un H1} and passing to a subsequence, we may define
\begin{equation} \label{E:phiJ1}
\begin{aligned}
\phi^j_1 &= \wklim_{n \to \infty} u^{(n)}_t(0,\cdot + x_n^j)
= \wklim_{n \to \infty} \Bigl\{u^{(n)}_t(0,\cdot + x_n^j) - \sum_{k=1}^{j-1} \phi^k_1(\cdot - x_n^k + x_n^j)\Bigr\},
\end{aligned}
\end{equation}
where now the weak limits are taken in $L^2_x$.  By \eqref{E:un Hsc}, we have $\phi^j_1 \in L^2_x \cap \dot H^{s_c-1}_x$ for all $1 \leq j < J^*+1$.  By Lemma~\ref{L:weak stability} and the fact that $T_n \to \infty$, the solutions $w^j$ to
$$
w_{tt}^j - \Delta w^j = |w^j|^pw^j \qtq{with} w^j(0) = \phi^j_0 \qtq{and} w_t^j(0) = \phi^j_1
$$
are global forward in time and satisfy $(w^j,w^j_t) \in C_t ([0, \infty);\dot H^1_x \times L^2_x\cap \dot H^{s_c}_x \times \dot H^{s_c-1}_x)$.

Since the $\phi^j_0$ are all nonzero, the lemma will follow if we can prove that there exists $j_0$ such that
$$
E(w^{j_0})=E(\phi^{j_0}_0,\phi^{j_0}_1) = \int_{\R^d} \tfrac12|\nabla\phi^{j_0}_0(x) |^2 +\tfrac12 |\phi^{j_0}_1(x)|^2 - \tfrac{1}{p+2}|\phi^{j_0}_0(x)|^{p+2}\, dx \leq 0.
$$
Indeed, $w^{j_0}$ would then be the solution described in Lemma~\ref{L:bad limit}.  Now, by \eqref{E:H1 decoup},
$$
\sum_{j=1}^{J^*} \|\nabla \phi^j_0\|_{L^2_x}^2 \leq \lim_{n \to \infty} \|\nabla u^{(n)}(0)\|_{L^2_x}^2,
$$
and by the definition of $\phi^j_1$ (cf.\ the proof of \eqref{E:H1 decoup}), we have
$$
\sum_{j=1}^{J^*} \|\phi^j_1\|_{L^2_x}^2 \leq \lim_{n \to \infty} \|u^{(n)}_t(0)\|_{L^2_x}^2.
$$
Moreover, by \eqref{E:p+2 decoup} and \eqref{E:rnJ to 0},
$$
\sum_{j=1}^{J^*} \|\phi^j_0\|_{L^{p+2}_x}^{p+2} = \lim_{n \to \infty} \|u^{(n)}(0)\|_{L^{p+2}_x}^{p+2}.
$$
Therefore, using \eqref{E:un nrg},
$$
\sum_{j=1}^{J^*} E(w^j) \leq \lim_{n \to \infty} \int_{\R^d} \tfrac12|\nabla_{t,x}u^{(n)}(0,x)|^2 - \tfrac1{p+2}|u^{(n)}(0,x)|^{p+2}\, dx \leq 0,
$$
and so at least one $w^j$ must have non-positive energy.  This completes the proof of the lemma.
\end{proof}

%
%
%
%

\section{Growth of other global norms}\label{S:other}

\begin{proposition} \label{P:finite mass}
Let $d \geq 2$, $m \in [0,1]$, and $0 < s_c < 1$.  Set $p = \frac4{d-2s_c}$.  Let $(u_0,u_1) \in H^1_x \times L^2_x$ and assume
that the maximal-lifespan solution $u$ to \eqref{E:eqn} blows up forward in time at $0<T_* < \infty$.  Then we have the pointwise in time bound
\begin{equation} \label{E:ptwise L2 bound}
\int_{\R^d} |u(t,x)|^2\, dx \lesssim (T_*-t)^{-\frac4p} \qtq{for all} 0 \leq t < T_*.
\end{equation}
Furthermore, if $I \subset [0,T_*)$ is an interval with $|I| \sim \dist(I,T_*)$, then we have the time-averaged bound
\begin{equation} \label{E:avg H1 bound}
\frac1{|I|} \int_I \int_{\R^d} |\nabla_{t,x} u(t,x)|^2\, dx\, dt \lesssim \dist(I,T_*)^{-\frac4p - 2}.
\end{equation}
The implicit constants in \eqref{E:ptwise L2 bound} and \eqref{E:avg H1 bound} may depend on $u$ but are independent of~$t$.
\end{proposition}

\begin{remark}
The blowup rate of solutions to the ODE $v''+m^2v -|v|^p=0$, namely, $v(t)\sim (T_*-t)^{-2/p}$, shows that the blowup rate in the proposition is sharp.
On the other hand, solutions such as those constructed by Kichenassamy, whose blowup surface $t=\sigma(x)$ has a non-degenerate minimum at $T_*$, show that one cannot expect lower bounds of comparable size to the upper bounds given above.
\end{remark}

\begin{proof}
We let
$$
M(t) := \int_{\R^d} |u(t,x)|^2\, dx
$$
denote the `mass'.  We differentiate twice with respect to time and use the equation and integration by parts to see that
\begin{align*}
M'(t) &= \int_{\R^d} 2 u(t) u_t(t)\, dx\\
M''(t) &= \int_{\R^d} 2 |u_t(t)|^2 - 2|\nabla u(t)|^2 - 2m^2 |u(t)|^2 + 2|u(t)|^{p+2}\, dx\\
&= -2(p+2)E(u) + \int_{\R^d} (p+4) |u_t(t)|^2 + p|\nabla u(t)|^2 + p m^2 |u(t)|^2\, dx.
\end{align*}
By Proposition~\ref{P:lwp}, $M(t),M'(t),M''(t)$ are all finite for $0 \leq t < T_*$.

As the solution $u$ blows up at $T_*$, by Corollary~\ref{C:ee blowup} we must have
$$
\lim_{t \uparrow T_*} \int_{\R^d} |\nabla u(t,x)|^2 \, dx = \infty.
$$
Using this and the conservation of energy, we deduce that there exists $0<t_0<T_*$ such that
$$
2(p+2)E(u) \leq \tfrac{p}2\|\nabla u(t)\|_{L^2_x}^2 \qtq{for all} t_0<t<T_*.
$$
Thus,
\begin{equation} \label{E:M'' lb}
M''(t) \geq \int_{\R^d} (p+4) |u_t(t,x)|^2 + \tfrac{p}2|\nabla u(t,x)|^2\, dx \qtq{for all} t_0 < t < T_*
\end{equation}
and so, by Cauchy--Schwarz,
\begin{equation} \label{E:M'<MM''}
|M'(t)|^2 \leq \frac4{p+4} M(t)M''(t) \qtq{for all} t_0 < t < T_*.
\end{equation}
From \eqref{E:M'' lb} we see that $M(t)\geq 0$ is strictly convex on $(t_0, T_*)$ and so vanishes at most once on this interval.  Altering $t_0$ if necessary, we may thus assume $M(t)>0$ on $(t_0, T_*)$.  This and \eqref{E:M'<MM''} show that $M(t)^{-\frac{p}4}$ is concave on $(t_0,T_*)$; indeed,
$$
\partial_t^2 M(t)^{-\frac{p}4} = -\tfrac p4\bigl[M''(t)M(t) - \tfrac{p+4}4(M'(t))^2 \bigr]M(t)^{-\frac{p+8}4} \leq 0.
$$
Therefore for all $t,T$ satisfying $t_0 < t \leq T < T_*$, we have
$$
M(t)^{-\frac{p}4} \geq \frac{t-t_0}{T-t_0}M(T)^{-\frac{p}4} + \frac{T-t}{T-t_0}M(t_0)^{-\frac{p}4}\geq  \frac{T-t}{T-t_0}M(t_0)^{-\frac{p}4}.
$$
Letting $T\uparrow T_*$ and rearranging yields
$$
M(t) \leq M(t_0)(T_*-t_0)^{\frac4p}(T_*-t)^{-\frac4p} \qtq{for all} t_0 < t < T_*.
$$
This proves \eqref{E:ptwise L2 bound}, at least for $t_0 < t < T_*$.  For $0\leq t\leq t_0$ this is trivial by the local-in-time continuity of $M$.

We now turn to \eqref{E:avg H1 bound}.  It suffices to consider intervals $I \subset (t_0,T_*)$ with $|I| \sim \dist(I,T_*)$.  Let $\phi_I$ be a smooth cutoff
with $\phi \equiv 1$ on $I$, $\supp \phi_I \subset [0,T_*)$, $|\supp \phi_I| \sim\dist(\supp \phi_I, T_*) \sim |I|$, and $|\phi''| \lesssim |I|^{-2}$.  Then by \eqref{E:M'' lb}, integration by parts, and \eqref{E:ptwise L2 bound}, we have
\begin{align*}
\int_I \int_{\R^d} |\nabla_{t,x} u(t,x)|^2\, dx\, dt &\lesssim \int_0^{T_*} \phi_I(t) M''(t)\, dt = \int_0^{T_*} \phi_I''(t)M(t)\, dt \\
&\lesssim |I||I|^{-2}\dist(I,T_*)^{-\frac4{p}} \sim |I|^{-\frac4{p} -1}.
\end{align*}

This completes the proof of the proposition.
\end{proof}

When $0 < s_c \leq \frac12$, the estimate \eqref{E:avg H1 bound} can be improved to a pointwise in time estimate, namely,
\begin{equation} \label{E:temp cone bound}
\int_{|x-x_0|<T_*-t}  (T_*-t)^{2(1-s_c)}|\nabla_{t,x}u(t,x)|^2\, dx \lesssim 1.
\end{equation}
Inequality \eqref{E:temp cone bound} is proved for $m=0$ in \cite{MerleZaagAJM, MerleZaagMA}.  For $0<m\leq1$, this is the content of Theorem~\ref{T:cone bound subc}.

By the local well-posedness in Proposition~\ref{P:lwp} and finite speed of propagation, there exists $R>0$ such that
\begin{equation} \label{E:global H1 outer space}
\int_{|x|\geq R+T_*} |\nabla_{t,x} u(t,x)|^2\, dx \lesssim 1.
\end{equation}
Since $\{|x|\leq R+T_*\}$ is contained in the union of $C_{d,T_*,R}(T_*-t)^{-d}$ balls of radius $(T_*-t)$, \eqref{E:temp cone bound} implies that
$$
\int_{|x|\leq R+T_*}|\nabla_{t,x}u(t,x)|^2\, dx \lesssim (T_*-t)^{-2(1-s_c)-d} = (T_*-t)^{-\frac4p-2}.
$$
The estimate \eqref{E:avg H1 bound} follows by combining the inequality above with \eqref{E:global H1 outer space}.

In Section~\ref{S:superconf blow} we prove averaged-in-time estimates inside light cones in the super-conformal case.  However, when these are used to derive global in space bounds (in the manner just shown), the result is weaker than that given in \eqref{E:avg H1 bound}.

%
%
%
%

\section{Lyapunov functionals}\label{S:Lyapunov}

The most flexible way to describe conservation laws is in their microscopic form, that is, as the fact that a certain
vector field is divergence-free in spacetime.  Myriad consequences can then be derived by applying the divergence
theorem, or, more generally, by pairing the vector field with the gradient of a function and integrating by parts.  One
of our goals in this section is to identify the underlying microscopic identities that yield the key monotonicity formulae
in the analyses of Merle and Zaag.  This points the way to the appropriate analogues for the results in the remaining sections.

To simplify various expressions, for the remainder of the article, we will work in light cones
$$
\{(t,x):0 < t \leq T, |x-x_0| < t\} \qtq{with} (T,x_0) \in \R_+\times\R^d,
$$
rather than the backwards light cones discussed earlier.  It is clear how to adapt Definition~\ref{D:solution} to this case. All the local theory results from Section~\ref{S:local theory} carry over by applying the time translation/reversal symmetry $u(t,x) \mapsto u(T-t,x)$.

We begin with energy conservation: If $u$ is a solution to \eqref{E:eqn} and
\begin{equation}\label{E:m E defn}
\mce^0 := \tfrac12 u_t^2 + \tfrac12 |\nabla u|^2 + \tfrac{m^2}2 u^2 - \tfrac1{p+2} |u|^{p+2}
    \qtq{and} \vec\mce := - u_t \nabla u,
\end{equation}
then
\begin{equation}\label{E:m E cons}
\partial_t \mce^0 + \nabla \cdot \vec\mce = 0.
\end{equation}

The closest thing to a general procedure for discovering conservation laws is via Noether's theorem which makes the connection to (continuous) symmetries.
The general nonlinear Klein--Gordon equation \eqref{E:eqn} has only the $\binom{d+2}{2}$-dimensional Poincar\'e group as symmetries; however, in the special
case of $m=0$ and $p=4/(d-1)$ the symmetry group becomes the full $\binom{d+3}{2}$-dimensional conformal group (of $(d+1)$-dimensional \emph{spacetime}).
Note that $p=4/(d-1)$ corresponds to $s_c=\frac12$, which explains the sub-/super-conformal nomenclature used in this paper.

While some elements of the conformal group fail to be true symmetries of the equation, the vestigial `conservation laws' that arise have proven to be very useful.
The requisite computations are rather lengthy; nevertheless, the results are very neatly catalogued in the paper \cite{Strauss77} by Strauss.  This paper also
contains a proof that the only continuous symmetries are those described above, which is to say, they generate the full Lie algebra of Killing fields.
Let us quickly review the list.

\emph{Translations:} Time translation symmetry is responsible for the energy conservation \eqref{E:m E cons}, above.  Spatial translation symmetry implies
the conservation of momentum.  Note that momentum conservation is of limited use, since it is not coercive.  While energy is not coercive in the strictest
sense in the focusing case, it is at least a scalar.

\emph{Rotations:}  Here we include the full group of spacetime rotations $SO(1,d)$, which includes both spatial rotations and Lorentz boosts.  This produces
a tensor of conserved quantities, of which the usual angular momentum is a part.  Again their utility is limited because they are not coercive.

\emph{Dilation:} By dilation, we mean rescaling both space and time.  This gives rise to a very important conservation law: if
\begin{equation}\label{E:frak d}
\begin{aligned}
  \mcd^0 &{}:= \hbox to 0.5em{\hss$t$\hss} \bigl[ \tfrac12 |\nabla u|^2 - \tfrac12 u_t^2 + \tfrac{m^2}2 u^2 - \tfrac1{p+2} |u|^{p+2} \bigr]
    + \bigl[ x\cdot\!\nabla u + t u_t + \tfrac{d-1}{2} u \bigr]  u_t \\
\vec\mcd &{}:= \hbox to 0.5em{\hss$x$\hss} \bigl[ \tfrac12 |\nabla u|^2 - \tfrac12 u_t^2 + \tfrac{m^2}2 u^2 - \tfrac1{p+2} |u|^{p+2} \bigr]
    - \bigl[ x\cdot\!\nabla u + t u_t + \tfrac{d-1}{2} u \bigr]\nabla u
\end{aligned}
\end{equation}
then
\begin{align}\label{E:m dilation}
  \partial_t \mcd^0 + \nabla\! \cdot\!\;\! \vec\mcd =  \tfrac{p(d-1)-4}{2(p+2)} |u|^{p+2} + m^2 |u|^2.
\end{align}
This is an honest conservation law only in the conformally invariant case ($m=0$ and $p=\frac4{d-1}$).  However, in the super-conformal case (i.e., $p > \frac4{d-1}$),
both terms have the same sign; thus we obtain a monotonicity formula --- a Lyapunov functional!

\emph{Conformal translations:}  Recall that inversion in a cone, that is,
$$
(t,x) \mapsto \bigl(\tfrac{t}{t^2 - |x|^2},\tfrac{x}{t^2 - |x|^2}),
$$
is a conformal map of spacetime.  This involution does not commute with translations; by forming commutators, we obtain a $(d+1)$-dimensional family of continuous symmetries (at
least in the conformally invariant case).  The resulting conservation laws are called conformal energy (relating to time translation) and conformal momentum
(resulting from spatial translations).  The conformal momentum lacks coercivity.  The conservation of conformal energy reads as follows: if
\begin{align}\label{E:frak q0}
\mck^0 &:= (t^2+|x|^2) \mce^0 + 2tu_t(x\cdot\!\nabla u) + (d-1)tuu_t - \tfrac{d-1}2 u^2 \\
\vec\mck &:= -\bigl[ (t^2+|x|^2) u_t + 2t(x\cdot\!\nabla u)+(d-1)tu\bigr] \nabla u  \notag \\
& \qquad\qquad\qquad - 2xt\bigl[ \tfrac12 u_t^2 - \tfrac12 |\nabla u|^2 - \tfrac{m^2}2 u^2 + \tfrac1{p+2} |u|^{p+2} \bigr] \label{E:frak qj}
\end{align}
then
\begin{align}\label{E:m conf E}
  \partial_t \mck^0 + \nabla\! \cdot\!\;\! \vec\mck = t \tfrac{p(d-1)-4}{(p+2)} |u|^{p+2} + 2tm^2 |u|^2 .
\end{align}

This completes the list.  We found $d+1$ translations, $\tbinom{d+1}{2}$ rotations, $1$ dilation, and $d+1$ conformal translations.  These generate the promised
$\binom{d+3}{2}$-dimensional group of conformal symmetries.

We now turn to converting these microscopic conservation laws into integrated form.  The key identities we need originate from energy conservation
\eqref{E:m E cons} and the dilation identity \eqref{E:m dilation}.  The conformal energy identity \eqref{E:m conf E} has a very similar structure to the dilation identity \eqref{E:m dilation};
however, the extra factor of $2t$ on the right-hand side of \eqref{E:m conf E} makes it inferior for our purposes.

Integrating the energy identity \eqref{E:m E cons} yields the family of well-known energy flux identities.  The particular cases we need are the following:

\begin{lemma}[Energy flux identity] \label{L:E flux}
Let $u$ be a strong solution to \eqref{E:eqn} in the lightcone
\begin{equation}\label{E:basic cone}
\Gamma := \bigl\{ (t,x): 0<t\leq T \text{ and } |x|<t\bigr\}.
\end{equation}
Then for all $0<t_0<t_1<T$,
\begin{align}
& \int_{|x| < t_1} \tfrac{t_1^2-|x|^2}{t_1} \mce^0(t_1,x) \,dx -  \int_{|x| < t_0} \tfrac{t_0^2-|x|^2}{t_0} \mce^0(t_0,x) \,dx  \notag \\
& \qquad = \int_{t_0}^{t_1} \!\!\! \int_{|x|<t} \tfrac14(1+\tfrac{|x|}t)^2\bigl[u_t(t,x)+u_r(t,x)\bigr]^2 + \tfrac14(1-\tfrac{|x|}t)^2\bigl[u_t(t,x)-u_r(t,x)\bigr]^2\notag \\
&\qquad \qquad  + (1+\tfrac{|x|^2}{t^2})\bigl[ \tfrac12|\nabslash u(t,x)|^2 + \tfrac{m^2}2 |u(t,x)|^2 - \tfrac1{p+2}|u(t,x)|^{p+2} \bigr]\, dx\, dt. \label{E:E flux}
\end{align}
Here $u_r := \frac{x}{|x|}\cdot\nabla u$ and $\nabslash u := \nabla u - \frac{x}{|x|} u_r$ denote the radial and angular derivatives, respectively, and
$$
\mce^0(t,x) = \tfrac12 |u_t(t,x)|^2 + \tfrac12 |\nabla u(t,x)|^2 + \tfrac{m^2}2 |u(t,x)|^2 - \tfrac1{p+2} |u(t,x)|^{p+2},
$$
as in \eqref{E:m E defn}.
\end{lemma}

\begin{proof}
The identity follows easily from \eqref{E:m E cons} and integration by parts.
First we define a function $\phi$ and a frustrum $F$ as follows:
$$
F := \bigl\{ (t,x): t_0<t<t_1 \text{ and } |x|<t\bigr\}
    \quad\text{by}\quad
\phi(t,x) = \begin{cases} \tfrac{t^2-|x|^2}{t} &:|x|<t \\ 0 &:|x|\geq t. \end{cases}
$$
Then integration by parts and \eqref{E:m E cons} show that
\begin{align*}
\iint_F \mce^0(t,x)\partial_t\phi(t,x) & +  \vec\mce(t,x)\cdot \nabla\phi(t,x) \,dx\,dt \\
&=\int_{\R^d}\mce^0(t_1,x)\phi(t_1,x) -\mce^0(t_0,x)\phi(t_0,x)\, dx,
\end{align*}
which is equivalent to \eqref{E:E flux}.

An alternate proof of \eqref{E:E flux} can be based on applying the divergence theorem to a family of concentric frustra inside $F$ with varying opening angle and then
averaging over this family.  This proof is more intuitive: on each frustrum we obtain the usual energy flux identity, namely, the energy at the top is equal to the
energy at the bottom plus the energy flux out through the side of the frustrum.  However, at the low regularity we are considering, this intermediate step is
ill-defined: $\mce^0$ and $\vec\mce$ are merely $L^1_\loc$.
\end{proof}

It is tempting (and not difficult) to run the same argument using the dilation identity; however, the result takes a more satisfactorily coercive form
if we make a trivial modification.  Here we mean trivial in a cohomological sense: observe that for any vector-valued function $\vec f:\R\times\R^d\to\R^d$ on spacetime,
$( \nabla\cdot \vec f ,\ - \partial_t \vec f)$ is divergence free, by equality of mixed partial derivatives.  This is quite different from \eqref{E:m E cons}
or \eqref{E:m dilation}, which rely on the fact that $u(t,x)$ solves a PDE, namely, \eqref{E:eqn}.

Specifically, defining
\begin{align}\label{E:frak l}
\mcl^0 := \mcd^0 + \tfrac{d-1}{4} \nabla\!\cdot\bigl(\tfrac{x}t u^2\bigr) \qtq{and} \vec\mcl := \vec\mcd - \tfrac{d-1}{4} \tfrac{\partial\ }{\partial t} \bigl(\tfrac{x}t u^2\bigr)
\end{align}
we deduce that
\begin{align}\label{E:m l}
  \partial_t \mcl^0 + \nabla\! \cdot\!\;\! \vec\mcl = \partial_t \mcd^0 + \nabla\! \cdot\!\;\! \vec\mcd = \tfrac{p(d-1)-4}{2(p+2)} |u|^{p+2} + m^2 |u|^2.
\end{align}
To see the improvement of coercivity over the original dilation identity \eqref{E:m dilation}, we need to expand out the definition of $\mcl^0$ and collect terms.
This yields
\begin{equation}\label{E:frak l0}
\begin{aligned}
\mcl^0 &= \tfrac1{2t}\bigl|x\!\;\!\cdot\!\nabla u + t u_t + \tfrac{d-1}{2} u\bigr|^2 + \tfrac{t}2\bigl(|\nabla u|^2 - |\tfrac{x}t \cdot \nabla u|^2\bigr)
    - \tfrac{t}{p+2}|u|^{p+2} \\
& \qquad {} + \tfrac{d^2-1}{8t} u^2  + t \tfrac{m^2}{2} u^2;
\end{aligned}
\end{equation}
indeed, the modification \eqref{E:frak l} was chosen precisely to complete the squares here and in \eqref{E:L bndry}.

\begin{lemma}[Two dilation inequalities]\label{L:dilat id}
Let $u$ be a strong solution to \eqref{E:eqn} in the light cone \eqref{E:basic cone}.  Then
\begin{equation}\label{E:L flux ineq}
\begin{aligned}
\int_{t_0}^{t_1} \!\!\! \int_{|x|<t} \! \tfrac{p(d-1)-4}{2(p+2)} & |u(t,x)|^{p+2} + m^2 |u(t,x)|^2 \, dx\,dt + \int_{|x| < t_0} \mcl^0(t_0,x) \,dx \\
&\leq \int_{|x| < t_1} \mcl^0(t_1,x) \,dx
\end{aligned}
\end{equation}
for all $0<t_0<t_1\leq T$.  Moreover, in the conformal case $p(d-1)=4$ we have
\begin{equation}\label{E:L flux ineq 2}
\begin{aligned}
\!\!\! \int_{t_0}^{(1+\alpha)t_0} \!\!\! \int_{|x|<\alpha t} \! (t-|x|)^{d+1}|\nabla_{\!t,x} u(t,x)|^2 + (t-|x|)^{d-1} |u(t,x)|^2 dx\!\:dt
    \lesssim \alpha t_0^{d+1}\!,
\end{aligned}
\end{equation}
uniformly for $0<\alpha\leq 1$ and $[t_0, (1+\alpha)t_0]\subset (0, T]$.  The implicit constant depends on $d$, $T$, and the $H^1_x\times L^2_x$ norm of $(u(T),u_t(T))$ on the ball $\{|x|<T\}$.
\end{lemma}

\begin{proof}
We begin with \eqref{E:L flux ineq}.  If $u$ where $C^2$, then we could apply the divergence theorem on the frustrum $F=\{ (t,x): t_0<t<t_1 \text{ and } |x|<t\}$ to obtain
\begin{align}
&\int_{t_0}^{t_1} \!\!\! \int_{|x|<t} \! \tfrac{p(d-1)-4}{2(p+2)} |u(t,x)|^{p+2} + m^2 |u(t,x)|^2 \, dx\,dt  \label{E:L div thm} \\
{}={}& \int_{|x| < t_1} \mcl^0(t_1,x) \,dx - \int_{|x| < t_0} \mcl^0(t_0,x) \,dx + \int_{t_0}^{t_1} \!\!\! \int_{|x|=t} \vec\mcl\cdot\tfrac{x}{|x|} - \mcl^0 \ dS(x)\,dt. \notag
\end{align}
Here, $dS$ denotes surface measure on the sphere $\{|x|=t\}$, or, equivalently, $(d-1)$-dimensional Hausdorff measure.  Note that although $(-1,x/|x|)$ is not
a unit vector, this is compensated for by the fact that $dS(x)\,dt$ is $2^{-1/2}$ times $d$-dimensional surface measure on the cone.

Thus, for $u\in C^2$ the inequality \eqref{E:L flux ineq} follows directly from \eqref{E:L div thm} by neglecting the manifestly sign-definite term
\begin{equation}\label{E:L bndry}
\int_{t_0}^{t_1} \!\!\! \int_{|x|=t} \vec\mcl\;\!\cdot\tfrac{x}{|x|} - \mcl^0 \ dS(x)\,dt =
    - \int_{t_0}^{t_1} \!\!\! \int_{|x|=t} \tfrac1t\bigl[x\cdot\!\nabla u + t u_t + \tfrac{d-1}{2} u\bigr]^2  \, dS(x)\,dt.
\end{equation}

To make this argument rigorous when $(u,u_t)$ is merely in $H^1_x\times L^2_x$, one can use the integration by parts technique of the previous lemma.  This time,
one chooses $\phi$ to be a mollified version of the characteristic function of the frustrum $F$.

We turn now to \eqref{E:L flux ineq 2}.  Again we simplify the presentation by assuming that the solution is $C^2$.  In the conformal case, the coefficient
of the potential energy term in \eqref{E:L div thm} is zero, so we rely instead on the positivity of \eqref{E:L bndry}.  Applying
the dilation identity \eqref{E:L div thm} to $u(t+s,x+y)$ with fixed $0<s<T$ and $|y|<s$, and using the fact that
\begin{align*}
\lim_{t\searrow s} \int_{|x-y|<t-s} \tfrac{t-s}{p+2} |u(t,x)|^{p+2}\,dx = 0 \quad \text{for all $0<s<T$},
\end{align*}
by the definition of a strong solution, we obtain
\begin{align*}
\int_s^T \!\!\! \int_{|x-y|=t-s} \bigl[(x-y)\cdot\!\nabla u(t,x) + (t-s) u_t(t,x) + \tfrac{d-1}{2} u(t,x)\bigr]^2  \frac{dS(x)\,dt}{t-s}
   \lesssim 1,
\end{align*}
where the implicit constant depends on $T$ and the $H^1_x\times L^2_x$ norm of $(u(T),u_t(T))$ on the ball $\{|x|<T\}$.
We will deduce \eqref{E:L flux ineq 2} by integrating this over all choices of $(s,y)$ in the region
\begin{equation}\label{sy restr}
 R(\alpha):= \bigl\{ (s,y) : (1-\alpha)t_0 < s+|y| < (1+\alpha)^2 t_0 \ \text{and} \ |y| < s \bigr\}.
\end{equation}
As $R(\alpha)$ has volume $O(\alpha t_0^{d+1})$, we deduce
\begin{align*}
\iint_{R(\alpha)}\! \int_s^T \!\!\! \int_{|x-y|=t-s} \bigl[(x-y)\cdot\!\nabla u + (t-s) u_t + \tfrac{d-1}{2} u\bigr]^2  \frac{dS(x)\,dt}{t-s} \,dy\,ds
  \lesssim \alpha t_0^{d+1}.
\end{align*}
Next we replace the variable $x$ by $\omega\in S^{d-1}$ via $x=y+(t-s)\omega$ and then change variables a second time from $y$ to $x$ via $y=x-(t-s)\omega$.
This yields
\begin{align*}
\int\!\!\!\int\!\!\!\int\!\!\!\int_{\Omega(\alpha)} \bigl[(t-s) \bigl(\omega\cdot\!\nabla u(t,x) &+ u_t(t,x)\bigr) + \tfrac{d-1}{2} u(t,x)\bigr]^2 (t-s)^{d-2} \,ds\,dS(\omega)\,dx\,dt\\
&\lesssim \alpha t_0^{d+1},
\end{align*}
where the region of integration is
$$
\Omega(\alpha) := \bigl\{ (s,\omega,x,t) : \omega\in S^{d-1}, \ 0< s < t < T, \ \text{and} \ \bigl(s,x-(t-s)\omega\bigr)\in R(\alpha) \bigr\}.
$$
To find a lower bound for this integral, we replace $\Omega(\alpha)$ by a smaller region, namely,
$$
\tilde\Omega(\alpha):= \{ (s,\omega,x,t) : \omega\in S^{d-1}, \ |x|<\alpha t, \ t_0<t<(1+\alpha)t_0, \ 0 < t-s < \tfrac{t-|x|}2  \}.
$$

Verifying $\tilde\Omega(\alpha)\subseteq \Omega(\alpha)$ rests on two simple observations.  First,
$$
|x-(t-s)\omega| < s \iff 2\bigl(t-x\cdot\omega\bigr)(t-s) < t^2 -|x|^2
$$
and so $0<t-s< \tfrac{t-|x|}{2}$ implies $|x-(t-s)\omega| < s$ whenever $|x|< t$. The second observation is that for $|x|< \alpha t$ and $t_0<t<(1+\alpha)t_0$,
$$
s + |x - (t-s)\omega| \in \bigl[t-|x|, t+|x|\bigr] \subseteq \bigl((1-\alpha)t_0, (1+\alpha)^2t_0\bigr).
$$

The decoupling of variables that is built into the structure of $\tilde\Omega(\alpha)$ makes the integral very easy to evaluate.  Indeed, freezing $t$ and $x$
for a moment we have
\begin{align*}
\frac1{|S^{d-1}|} \int_{S^{d-1}} &\int_{(t+|x|)/2}^t \bigl[(t-s) \bigl(\omega\cdot\!\nabla u + u_t\bigr) + \tfrac{d-1}{2} u\bigr]^2 (t-s)^{d-2} \,ds\,dS(\omega) \\
&= \tfrac{1}{d+1} \bigl(\tfrac{t-|x|}{2}\bigr)^{d+1} \bigl[ u_t^2 + \tfrac1d |\nabla u|^2\bigr]
    + \tfrac{d-1}{d} \bigl(\tfrac{t-|x|}{2}\bigr)^{d} u u_t
    + \tfrac{d-1}{4} \bigl(\tfrac{t-|x|}{2}\bigr)^{d-1} u^2.
\end{align*}
Thus the desired estimate \eqref{E:L flux ineq 2} follows easily by restoring the integrals over $t$ and $x$ and by noting that the
quadratic form
$$
\tfrac{1}{d+1} X^2 + \tfrac{d-1}{d} XY + \tfrac{d-1}{4} Y^2
$$
has negative discriminant (and so is positive definite).
\end{proof}

Notice that the LHS\eqref{E:L div thm} is sign-definite for any $p\geq \frac{4}{d-1}$, thus providing a Lyapunov functional in that case.  More precisely,
Lemma~\ref{L:dilat id} shows that
\begin{equation} \label{E:L(t)}
L(t) := \int_{|x| < t} \mcl^0(t,x) \,dx
\end{equation}
is an increasing function of time.  Moreover, the expansion \eqref{E:frak l0} shows that this Lyapunov functional has good coercivity properties.  Specializing to
the conformally invariant case $m=0$ and $p=\frac{4}{d-1}$ we find precisely the Lyapunov functional used by Merle and Zaag in \cite{MerleZaagMA} in their treatment
of this case (cf. Lemma~21 in \cite{MerleZaagMA}).  This is not immediately apparent because Merle and Zaag use similarity variables
\begin{equation}\label{E:ss vars}
w(s,y) := t^{-2/p} u(t,x) \qtq{with} y := x/t \qtq{and} s := \log(t),
\end{equation}
as advocated in earlier work \cite{GigaKohn:Indiana} of Giga and Kohn on the semilinear heat equation.  In particular, the connection to the dilation identity
does not seem to have been noted before.

In Section~\ref{S:conf blow}, we will revisit the work of Merle and Zaag on the conformally invariant wave equation in the course of discussing analogous
results for Klein--Gordon.  The principal deviation from their argument is the use of the estimate \eqref{E:L flux ineq 2} appearing in Lemma~\ref{L:dilat id},
which should be compared with Proposition~2.4 in \cite{MerleZaagMA} and Proposition~4.2 in \cite{MerleZaagIMRN}.  Like their estimates, \eqref{E:L flux ineq 2}
was proved by averaging the identity \eqref{E:L div thm}.  Here we see one advantage to working in usual coordinates, namely, it makes it clear how to perform
this averaging so as to obtain control over all directional derivatives.  More specifically, the region $R(\alpha)$ appearing in \eqref{sy restr} was
chosen to contain all spacetime points whose future light cones intersect the region of $(t,x)$ integration appearing in LHS\eqref{E:L flux ineq 2}.

\begin{lemma}\label{L:L>0}
Let $d \geq 2$, $m \in [0,1]$, and $p = \frac4{d-2s_c}$ with $\tfrac12 \leq s_c < 1$.  If $u$ is a strong solution to \eqref{E:eqn} in the light cone
$\{(t,x):0 < t \leq T, |x|<t\}$, then $L(t)\geq 0$ for all $0 < t \leq T$.
\end{lemma}

\begin{proof}  The following argument appears also in \cite{AntoniniMerle}.  Suppose by way of contradiction that $L(t_0)<0$ for some $t_0 \in (0,T]$. By the dominated convergence theorem, \eqref{E:frak l0}, and our assumption that $u$ is a strong solution, there exists $0<\delta< t_0$ such that
\begin{equation} \label{E:L delta}
\begin{aligned}
L_\delta(t) := \int_{|x| < t-\delta}& \tfrac1{2(t-\delta)}\bigl|x\!\;\!\cdot\!\nabla u + (t-\delta) u_t + \tfrac{d-1}{2} u\bigr|^2 + \tfrac{t-\delta}2\bigl(|\nabla u|^2 - |\tfrac{x}{t-\delta} \cdot \nabla u|^2\bigr)\\
&\qquad {}     - \tfrac{t-\delta}{p+2}|u|^{p+2} + \tfrac{d^2-1}{8(t-\delta)} u^2  + (t-\delta) \tfrac{m^2}{2} u^2\, dx
\end{aligned}
\end{equation}
is negative at time $t=t_0$.

Letting $\mcl_u^0$ denote the quantity in \eqref{E:frak l0} and $u^\delta(t) := u(t+\delta)$, we may write
$$
L_\delta(t) = \int_{|x| < t-\delta} \mcl^0_{u^\delta}(t-\delta,x)\, dx.
$$
As $u^\delta$ is a strong solution in the light cone $\{(t,x):-\delta < t \leq T-\delta, |x|<t+\delta\}$, Lemma~\ref{L:dilat id} implies that $L_\delta$ is an increasing function of time.  As $L_\delta(t_0)<0$ and $0<\delta <t_0$, we deduce that
\begin{equation} \label{E:lim Ldelta neg}
\lim_{t \searrow \delta} L_\delta(t) < 0.
\end{equation}

On the other hand, \eqref{E:L delta} gives
\begin{equation} \label{E:Ldelta lb}
L_\delta(t) \geq -\int_{|x| < t} \tfrac{t-\delta}{p+2} |u(t,x)|^{p+2}\, dx.
\end{equation}
As $u$ is a strong solution, Lemma~\ref{L:Sob domain} implies that $\|u(t)\|_{L^{p+2}(|x|<t)}$ is uniformly bounded for $t \in [\delta,T]$; the bound may depend on $\delta$, but this is irrelevant.  Thus, by \eqref{E:Ldelta lb},
$$
\lim_{t \searrow \delta} L_\delta(t) \geq 0,
$$
contradicting \eqref{E:lim Ldelta neg}.  This completes the proof of the lemma.
\end{proof}

In the next section, we will see that our results in the super-conformal case (i.e., when $p > \frac{4}{d-1}$) are less complete than for the conformal and sub-conformal cases.
Without going into any details of the analysis, we can already explain why this happens: scaling.  Recall from the introduction that when $m=0$, the set of solutions to
\eqref{E:eqn} is invariant under the scaling $u(t,x)\mapsto u^\lambda(t,x) := \lambda^{2/p} u(\lambda t,\lambda x)$.  Recall also that the critical regularity $s_c$ is determined
by invariance under this scaling, which leads to $s_c = \frac d2 - \frac2p$.  Note that it is reasonable to neglect the mass term when computing the scaling of the equation, on
account of it being subcritical when compared with the other terms.

We can apply the same reasoning to the conservation of energy to see that it has $\dot H^1_x$ scaling, at least in the cases when $p\leq4/(d-2)$, which includes those discussed in
this paper.  When $p>4/(d-2)$ the derivative terms in the energy are subcritical relative to the potential energy, which means that it is more reasonable to assert that the energy
has the scaling of $\dot H^s_x$ with $s=pd/[2(p+2)]$.  (For such $p$, this $s$ is less than $s_c$, so the equation is supercritical relative to both terms in the energy.)

The dilation identity has $\dot H_x^{1/2}$ scaling; notice, for example, that the components of $\mcd$ resemble those of $\mce$ multiplied by length (or time, which has the same
dimensionality).  By comparison, the conformal energy scales as $L^2_x$ and this is why it is inferior for our purposes.  Indeed, experience has shown that after coercivity, the
utility of a conservation/monotonicity law is dictated by its proximity to critical scaling.  Thus we obtain optimal results when $s_c=1/2$, but only weaker results at higher
critical regularity.

When $p<4/(d-1)$, that is, when $s_c < 1/2$, the equation is subcritical relative to the dilation identity.  As we are working on a finite time-interval, this is a favourable
situation.  However, in this case the identity is no longer coercive, which is very bad news.

As the dilation identity is local in space, we can produce a whole family of identities (with lower scaling regularity) by averaging translates.  While the previously neglected
coercivity \eqref{E:L bndry} will now produce an additional positive volume integral term, it is far from obvious that coercivity can be restored.  Nevertheless, Antonini and
Merle \cite{AntoniniMerle} demonstrated that there is a Lyapunov functional when $p<4/(d-1)$.  In the limit $p\nearrow 4/(d-1)$, one recovers the functional used in the subsequent
paper \cite{MerleZaagMA}.  Given the connection of this limiting case to the dilation identity (the topic of Lemma~\ref{L:dilat id}), one would expect to find a connection for all
$p<4/(d-1)$.  Our next task is to explicate this connection.  To do so, we need to begin with a slight detour.

Let us consider \emph{complex}-valued solutions to \eqref{E:eqn} for a moment.  In this case we pick up an additional symmetry, namely, phase rotation invariance; the class of solutions is invariant under $u(t,x)\mapsto e^{i\theta}u(t,x)$.  This begets the law of charge conservation:  If  $\mcq^0 = \bar u u_t$ and $\vec\mcq = - \bar u \nabla u$, then
\begin{gather}\label{E:m charge id}
 \partial_t \mcq^0 + \nabla\! \cdot\!\;\! \vec\mcq = |u_t|^2 - |\nabla u|^2 - m^2 |u|^2 + |u|^{p+2}.
\end{gather}
Strictly speaking, charge conservation corresponds to the imaginary part of this identity, for which the right-hand side vanishes.  (For real-valued solutions, this is just $0=0$.)  The real part of \eqref{E:m charge id} is a non-trivial identity, even in the case of real-valued solutions to \eqref{E:eqn}, although it is no longer a true conservation law.  Like the dilation identity, this law has $\dot H_x^{1/2}$ scaling.

Although it is incidental to the main themes of this section, let us pause to observe that the charge identity underpins the virial theorem for this system.  Recall that the virial theorem (cf. \cite{Clausius} or \cite[\S 10]{LandauLif1}) shows the following: For a mechanical system whose potential energy is a homogenous function of the coordinates (of $u$ in our case), the time-averaged potential and kinetic energies are in the proportion dictated by the homogeneity.  The proof extends immediately to the case of potential energies that are a sum of terms with different homogeneities, as is the case for our equation:

\begin{lemma}[Virial identity]\label{L:EnEqPart}
Let $u:[0,\infty)\times\R^d\to\R$ be a global strong solution to \eqref{E:eqn} with $u \in L^\infty_t H^1_x$ and  $u_t\in L^\infty_t L^2_x$.  Then
\begin{align*}
\lim_{T\to\infty}\frac1T \int_0^T \!\!\! \int_{\R^d} \! & |\nabla u(t,x)|^2 + m^2 |u(t,x)|^2  -  |u(t,x)|^{p+2} \,dx\,dt \\
 &= \lim_{T\to\infty} \frac1T \int_0^T \!\!\! \int_{\R^d} \! |u_t(t,x)|^2 \,dx\,dt.
\end{align*}
\end{lemma}

\begin{remark}
In the lemma we assume that $\|(u(t),u_t(t))\|_{H^1 \times L^2}$ is bounded; it suffices that this norm is merely $o(t)$, as will become immediately apparent in the proof.
\end{remark}

\begin{proof}
Integrating \eqref{E:m charge id} over the space-time slab $[0,T]\times\R^d$ gives
\begin{align*}
\int_0^T \!\!\! \int_{\R^d} \!  |u_t(t,x)|^2 - |\nabla u(t,x)|^2 - m^2 |u(t,x)|^2 &+ |u(t,x)|^{p+2} \,dx\,dt \\
 &= \int_{\R^d} \! \mcq^0(T,x) - \mcq^0(0,x)  \,dx.
\end{align*}
Observing that the right-hand side is $O(1)$ by hypothesis, the result follows by dividing by $T$ and rearranging a little.
\end{proof}

We now return to the question of determining the link between the Lyapunov functional introduced in \cite{AntoniniMerle} and the dilation identity.
It is natural to try averaging the dilation identity against a Lorentz-invariant function.  For a generic tensor $\mcz$, integration by parts formally yields the following:
\begin{equation}\label{E:general z id}
\begin{aligned}
 \int_{\R^d} & \mcz^0(t_1,x)\psi(t_1^2-|x|^2)\,dx - \int_{\R^d} \mcz^0(t_0,x) \psi(t_0^2-|x|^2)\,dx \\
={}& \int_{t_0}^{t_1} \! \int_{\R^d} [\partial_t \;\! \mcz^0 + \nabla \! \cdot \!\;\! \vec \mcz \,] \psi(t^2-|x|^2) + 2 [ t \;\! \mcz^0 - x \cdot \vec \mcz \,] \psi'(t^2-|x|^2) \,dx\,dt.
\end{aligned}
\end{equation}

We consider for a moment the case when $\mcz=\mcd$.  Starting with \eqref{E:frak d}, a few elementary manipulations reveal
\begin{equation}\label{E:td0-xd}
\begin{aligned}
t \;\! \mcd^0 - x \cdot \vec \mcd &= (t^2 - |x|^2) \bigl[ \tfrac12 |\nabla u|^2 - \tfrac12 {u_t}^2 + \tfrac{m^2}2 u^2 - \tfrac1{p+2} |u|^{p+2} \bigr] \\
    &\qquad  {} + \bigl[ x\cdot\!\nabla u + t u_t + \tfrac{d-1}{2} u \bigr] \bigl[ t u_t + x\cdot\!\nabla u\bigr].
\end{aligned}
\end{equation}
Notice the similarity of the first term in square brackets to the right-hand side of the charge identity \eqref{E:m charge id}.

More generally, if $\mcz$ has $\dot H^{1/2}_x$ scaling (like $\mcd$), then to obtain a formula with critical scaling, the function $\psi(t^2-|x|^2)$ should have dimensions of length to the power $2\alpha$ with $\alpha=\frac12-s_c$.
That is, $\psi$ should be homogenous of degree $\alpha$.  By Euler's formula, this implies
\begin{equation}\label{E:psi alpha-homo}
 \psi'(t^2-|x|^2) = \alpha (t^2-|x|^2)^{-1} \psi(t^2-|x|^2).
\end{equation}

\begin{lemma}[Combined dilation + charge identity]\label{L:dilat id2}
Assume $p< 4/(d-1)$ so that $\alpha:=\frac12-s_c > 0$ and let $u$ be a strong solution to \eqref{E:eqn} in the light cone \eqref{E:basic cone}.
Let
\begin{equation*}
\mcz^0 := \tfrac1{2t}\bigl|x\!\;\!\cdot\!\nabla u + t u_t + \tfrac2p u\bigr|^2 + \tfrac{t}2\bigl(|\nabla u|^2 - |\tfrac{x}t \!\cdot\! \nabla u|^2\bigr)
    - \tfrac{t}{p+2}|u|^{p+2}  + \bigl(\tfrac{m^2t}{2} + \tfrac{p+2}{p^2 t}\bigr) u^2.
\end{equation*}
Then
\begin{align}
\int_{|x|<t_1} &  \mcz^0(t_1,x) (t_1^2-|x|^2)^\alpha \,dx - \int_{|x|<t_0} \mcz^0(t_0,x) (t_0^2-|x|^2)^\alpha\,dx \label{E:AM mono}  \\
= & \int_{t_0}^{t_1} \!\!\!\int_{|x|<t} \!\! 2 \bigl| x\cdot\!\nabla u + t u_t + \tfrac2p u \bigr|^2 \alpha (t^2-|x|^2)^{\alpha-1}
    +  m^2 u^2 (t^2-|x|^2)^\alpha dx\,dt \notag
\end{align}
for all $0<t_0<t_1\leq T$.
\end{lemma}

\begin{proof}
We will prove the identity in slightly greater generality, by employing \eqref{E:general z id} with a general $\psi(t^2-|x|^2)$ with $\psi$ homogeneous of degree $\alpha$ and with
\begin{equation} \label{E:mcz defn}
\begin{aligned}
\mcz^0 = \mcd^0 + \alpha \mcq^0 + \tfrac1p\nabla \! \cdot \! \vec f + \tfrac{2\alpha}{p} g \qtq{and} \vec\mcz = \vec\mcd + \alpha \vec\mcq - \tfrac1p\partial_t \vec f .
\end{aligned}
\end{equation}
Here $\mcd$ and $\mcq$ represent the tensors associated with the dilation and charge identities (as in \eqref{E:m dilation} and \eqref{E:m charge id}, respectively), while
$$
\vec f(t,x) = \frac{x}t |u(t,x)|^2 \qtq{and} g(t,x) = t^{-1} |u(t,x)|^2.
$$

A little patience is all that is required to show that this definition of $\mcz^0$ agrees with that stated in the lemma.

Notice that the $\vec f$ terms in \eqref{E:mcz defn} differ only in the prefactor from those appearing in \eqref{E:frak l};  equality of mixed partial derivatives shows that this term does not affect the divergence.  It was chosen to complete squares in a formula below.

The additional summand $g$ appearing in \eqref{E:mcz defn} has no analogue in our previous computations.  It is not clear how one might intuit the introduction of this
term; however, if one proceeds without it, then the left-over terms can be recognized as a complete derivative and so explain \emph{a posteriori} its inclusion.

With these preliminaries out of the way, applying \eqref{E:general z id} we obtain
\begin{equation*}
\begin{aligned}
\int_{\R^d} \mcz^0(t_1,x)  & \psi(t_1^2-|x|^2)\,dx   - \int_{\R^d} \mcz^0(t_0,x) \psi(t_0^2-|x|^2)\,dx  \\
&= \int_{t_0}^{t_1} \!\!\!\int_{\R^d}  2 \bigl| x\cdot\!\nabla u(t,x) + t u_t(t,x) + \tfrac2p u(t,x) \bigr|^2  \psi'(t^2-|x|^2) \,dx\,dt \\
& \qquad\qquad  + \int_{t_0}^{t_1} \!\!\! \int_{\R^d}  m^2 |u(t,x)|^2 \psi(t^2-|x|^2) \,dx\,dt.
\end{aligned}
\end{equation*}
Specializing to
\begin{equation*}
\psi(t^2- |x|^2) = \begin{cases} (t^2 - |x|^2)^\alpha &: |x| < t \\ 0 &:  |x|\geq t \end{cases}
\end{equation*}
we obtain the identity stated in the lemma.
\end{proof}

Since the integrand on the right-hand side of \eqref{E:AM mono} is positive, this identity shows the monotonicity in time of the function
\begin{equation} \label{E:Z(t)}
Z(t) := \int_{|x|<t} \mcz^0(t,x) (t^2-|x|^2)^\alpha \,dx.
\end{equation}
After switching to self-similar variables (cf. \eqref{E:ss vars}), this agrees with the Lyapunov functional introduced by Antonini and Merle in \cite{AntoniniMerle}
and used subsequently in \cite{MerleZaagAJM}.  As we have seen, this functional is less directly deducible from the dilation identity than that discussed in Lemma~\ref{L:dilat id}, which appeared in the later paper \cite{MerleZaagMA} of Merle and Zaag.  Note that formally taking the limit $p\nearrow 4/(d-1)$ in Lemma~\ref{L:dilat id2} yields Lemma~\ref{L:dilat id}.

While we find the physical meaning of these Lyapunov functionals difficult to properly understand when written in self-similar variables, the example discussed in
Lemma~\ref{L:dilat id2} makes a good case for their utility as a technique for finding such laws.  In the two key examples discussed in this section, the correct
multiplier appears after switching to similarity variables and converting the resulting equation to divergence form; compare \cite[Eqn (4)]{MerleZaagAJM} and \cite[Eqn (1.22)]{GigaKohn:CPAM}.

Next we give analogues of Lemma~\ref{L:L>0} and \eqref{E:L flux ineq 2} for the functional $Z$.

\begin{corollary}\label{C:Z id}
Let $d \geq 2$, $m \in [0,1]$, and $p = \frac4{d-2s_c}$ with $0< s_c < \tfrac12$.  If $u$ is a strong solution to \eqref{E:eqn} in the light cone
$\{(t,x):0 < t \leq T, |x|<t\}$, then $Z(t)\geq 0$ for all $0 < t \leq T$.  Moreover,
\begin{align}\label{E:smudge Z}
\!\!\int_{t_0}^{2t_0} \!\!\!\!\int_{|x|<t} (t-|x|)^{d+2-2s_c} |\nabla_{t,x} u(t,x)|^2 + (t-|x|)^{d-2s_c} |u(t,x)|^2  \,dx\,dt \lesssim t_0^{d+1}, \!
\end{align}
uniformly for $0<t_0\leq\frac12T$.
\end{corollary}

\begin{proof}
The proof that $Z(t)\geq 0$ is identical to that of Lemma~\ref{L:L>0}.

To prove \eqref{E:smudge Z} we argue much as we did for \eqref{E:L flux ineq 2}.  First we observe that applying Lemma~\ref{L:dilat id2}
to the light cone with apex $(s,y)$ yields
$$
\int_s^T \int_{|x-y|<t-s} \bigl| (x-y)\cdot\!\nabla u + (t-s) u_t + \tfrac2p u \bigr|^2 \bigl\{(t-s)^2-|x-y|^2\bigr\}^{\alpha-1} dx\,dt \lesssim 1,
$$
where the implicit constant depends on $T$ and the $H^1_x\times L^2_x$ norm of $(u(T),u_t(T))$ on the ball $\{|x|<T\}$.  Next we integrate this inequality over the region $|y|<s<2t_0$ and so deduce
$$
\int\!\!\int\!\!\int\!\!\int \bigl| (x-y)\cdot\!\nabla u + (t-s) u_t + \tfrac2p u \bigr|^2 \bigl\{(t-s)^2-|x-y|^2\bigr\}^{\alpha-1} dx\,dt\,dy\,ds
\lesssim t_0^{d+1},
$$
where the integral is over a region which contains
$$
\Omega := \{ (s,y,t,x) : t_0<t<2t_0,\ \tfrac{t+|x|}{2} < s < t,\ |x-y| < t-s,\text{ and } |x|<t\}.
$$
Freezing $(t,x)$ and integrating out $y$ and then $s$ produces the estimate \eqref{E:smudge Z}.
\end{proof}

%
%
%
%

\section{Bounds in light cones:  The super-conformal case}\label{S:superconf blow}

In this section, we consider the super-conformal case, $\frac12 < s_c < 1$.  Little seems to be known about the behaviour of local norms for blowup solutions in this case.  In particular, the work of Merle and Zaag \cite{MerleZaagAJM, MerleZaagMA, MerleZaagIMRN} only considers the conformal and sub-conformal cases $(0 < s_c \leq \frac12$).

The majority of this section will be devoted to a proof of the following theorem:

\begin{theorem}[Mass/Energy bounds] \label{T:cone bound superc}  Let $d \geq 2$, $m \in [0,1]$, and $p = \frac4{d-2s_c}$ with $\frac12 < s_c < 1$.  If $u$ is a strong solution to \eqref{E:eqn} in the light cone $\{(t,x):0 < t \leq T, |x|<t\}$, then $u$ satisfies
\begin{equation} \label{E:cone bound superc}
\begin{aligned}
\int_0^{T} \int_{|x|<t} \bigl(1-\tfrac{|x|}{t}\bigr)^2 & |\nabla_{t,x}u(t,x)|^2 + |u(t,x)|^{p+2}\, dx\, dt \lesssim 1,
\end{aligned}
\end{equation}
as well as the pointwise in time bounds
\begin{equation} \label{E:mass bound superc}
\int_{|x|< t} |u(t,x)|^2\, dx \lesssim t^{\frac{pd}{p+4}} \qtq{and} \int_{|x|< t} |u(t,x)|^{\frac{p+4}2}\, dx \lesssim 1
\end{equation}
for all $t \in (0,T]$.  The implicit constants in \eqref{E:cone bound superc} and \eqref{E:mass bound superc} depend on $d$, $p$, $T$, $\|u(T)\|_{H^1_x(|x|<T)}$, and $\|u_t(T)\|_{L^2_x (|x|<T)}$.
\end{theorem}

\begin{proof}
We rely on the information provided by the Lyapunov functional $L$ introduced in the previous section.  Specifically, we will make use of
Lemmas~\ref{L:dilat id}~and~\ref{L:L>0}, which assert that
\begin{equation} \label{E:L superc}
L(t) := \int_{|x| < t} \mcl^0(t,x)\, dx
\end{equation}
is a nonnegative increasing function of time.

To keep formulae within margins, we will not keep track of the specific dependence on $T$ or $(u(T),u_t(T))$ in the  estimates that follow.

We first consider \eqref{E:cone bound superc}.  Combining Lemmas~\ref{L:dilat id}~and~\ref{L:L>0}, we immediately obtain
\begin{align} \label{E:p+2 superc}
\int_0^{T} \int_{|x|<t} |u(t,x)|^{p+2}\, dx\, dt
&\lesssim L(T)\lesssim 1.
\end{align}
Thus there exists a sequence $t_n \searrow 0$ such that
$$
\lim_{n \to \infty} t_n\|u(t_n)\|_{L_x^{p+2}(|x|<t_n)}^{p+2} = 0.
$$
Therefore, if we define
$$
g(t) := \int_{|x|<t} \tfrac{t^2-|x|^2}t \mce^0(t,x)\, dx,
$$
then
$$
\liminf_{n \to \infty} g(t_n) \geq 0.
$$
Thus, combining Lemma~\ref{L:E flux} and \eqref{E:p+2 superc}, we obtain
\begin{align*}
\int_0^{T} \int_{|x| < t} &\tfrac14\bigl(1+\tfrac{|x|}t\bigr)^2\bigl[u_t(t,x)+u_r(t,x)\bigr]^2 + \tfrac14\bigl(1-\tfrac{|x|}t\bigr)^2\bigl[u_t(t,x)-u_r(t,x)\bigr]^2\\
&\qquad +\tfrac12\bigl(1+\tfrac{|x|^2}{t^2}\bigr)|\nabslash u(t,x)|^2 \, dx\, dt\\
& \leq g(T) + \int_0^{T} \int_{|x| < t} \tfrac1{p+2}\bigl(1+\tfrac{|x|^2}{t^2}\bigr)|u(t,x)|^{p+2}\, dx\, dt \lesssim 1.
\end{align*}
This bounds each of the derivatives of $u$ with respect to the frame $\nabslash$, $\partial_t + \partial_r$, and $\partial_t - \partial_r$, which spans
all possible spacetime directions.  The estimate for the last of these is the weakest, since it deteriorates near the edge of the cone, and so dictates the
form of~\eqref{E:cone bound superc}.

We turn now to \eqref{E:mass bound superc}.  Observe that the first inequality follows from the second and H\"older's inequality.  Thus, it remains to establish
the second bound in \eqref{E:mass bound superc}.  With the estimates at hand, there are several ways to proceed.  The method we present below is informed by the needs of Section~\ref{S:conf blow}; in particular, it uses only the estimates available in that case.

We first notice that, as a consequence of \eqref{E:frak l0}, \eqref{E:cone bound superc}, and the monotonicity of $L(t)$,
\begin{align}\label{t0 bound}
\int_{t_0}^{2t_0} \!\!\int_{|x|<t} \! \bigl|x\cdot \nabla u + tu_t +\tfrac{d-1}2u \bigr|^2\, dx\, dt &\leq \int_{t_0}^{2t_0} \!2tL(t)\, dt
+ \int_{t_0}^{2t_0} \!\!\int_{|x|<t} \!\tfrac{2t^2}{p+2}|u|^{p+2}\, dx\, dt\notag\\
&\lesssim t_0^2
\end{align}
for all $0<t_0\leq \frac12 T$.  Moreover, by H\"older and \eqref{E:cone bound superc}, we can control the desired quantity but only on average in time:
\begin{align}\label{t0 bound q}
\int_{t_0}^{2t_0}\!\!\! \int_{|x|<t} |u(t,x)|^{\frac{p+4}2}\, dx\, dt
&\lesssim t_0^{\frac{p(d+1)}{2(p+2)}}\biggl[\int_{t_0}^{2t_0}\!\!\! \int_{|x|<t}|u(t,x)|^{p+2}\, dx\, dt\biggr]^{\frac{p+4}{2(p+2)}}\notag\\
&\lesssim t_0^{\frac{p(d+1)}{2(p+2)}}.
\end{align}

Next we will prove that the estimate \eqref{t0 bound q} together with the $L^2_{t,x}$ control \eqref{t0 bound} over the directional derivatives of $u$ inside the light cone imply the desired pointwise in time estimates.

Observe that for a $C^1$ function $f:[1,2]\to \R$,
\begin{align*}
\sup_{1\leq \lambda\leq 2} |f(\lambda)| &\leq \int_1^2 |f(\lambda)| \,d\lambda + \int _1^2 |f'(\lambda)|\,d\lambda.
\end{align*}
The second part of \eqref{E:mass bound superc} follows by applying this to $f(\lambda):= \int_{|x|<t_0} |\lambda^{\frac{d-1}2} u(\lambda t_0, \lambda x)|^{\frac{p+4}2} \, dx$ and using Cauchy--Schwarz, \eqref{E:cone bound superc}, \eqref{t0 bound}, and \eqref{t0 bound q}.  Indeed,
\begin{align*}
&\sup_{t_0\leq t\leq 2t_0} \int_{|x|<t} |u(t,x)|^{\frac{p+4}2}\, dx\\
&\quad \lesssim \sup_{1\leq \lambda\leq 2} |f(\lambda)|\\
&\quad\lesssim \frac1{t_0}\int_{t_0}^{2t_0} \!\!\!\int_{|x|<t} |u|^{\frac{p+4}2}\, dx\, dt\\
&\qquad +\frac1{t_0} \biggl(\int_{t_0}^{2t_0} \!\!\!\int_{|x|<t}\! |u|^{p+2}\, dx\, dt\biggr)^{\!1/2}
\biggl(\int_{t_0}^{2t_0} \!\!\!\int_{|x|<t}\bigl|x\cdot \nabla u + tu_t +\tfrac{d-1}2u \bigr|^2dx\, dt\biggr)^{\!1/2}\\
&\quad \lesssim t_0^{\frac{p(d-1)-4}{2(p+2)}} + 1\\
&\quad \lesssim_T 1.
\end{align*}
The last inequality relies on the super-conformality hypothesis, namely, $p(d-1)>4$.

This completes the proof of Theorem~\ref{T:cone bound superc}.
\end{proof}

We conclude this section with a corollary of Theorem~\ref{T:cone bound superc}.

\begin{corollary} \label{C:cone bound superc}
Let $d \geq 2$, $m \in [0,1]$, and $\frac12 < s_c < 1$.  Set $p=\frac4{d-2s_c}$.  Assume that there exists $0 < \eps \leq 1$ such that $u$ is a strong solution to \eqref{E:eqn} in the cone $\{(t,x):0 < t \leq T, |x|<(1+\eps)t\}$.  Then for each $0 < t_0 \leq \frac T2$,
\begin{equation} \label{E:dyadic bound superc}
\begin{aligned}
&\int_{t_0}^{2t_0} \int_{|x|<t} |\nabla_{t,x} u(t,x)|^2\, dx\, dt \lesssim 1,
\end{aligned}
\end{equation}
with the implicit constant depending on $d$, $p$, $\eps$, $T$, and the $H^1_x\times L_x^2$ norm of $(u(T), u_t(T))$ on the ball $\{|x|<T\}$.
\end{corollary}

We note that the assumption that $u$ is defined in the cone $\{(t,x):0 < t \leq T, |x|<(1+\eps)t\}$ is equivalent to the assumption that $(0,0)$ is not a characteristic point of the (backwards in time) blowup surface of $u$.

\begin{proof}
We begin with a simple covering argument.  There exist $N$, depending on $\eps$ and $d$, and a set $\{x_j\}_{j=1}^N$ with $|x_j|<(1+\eps)\frac{t_0}2$ such that
$$
\{x :|x|<t_0\} \subset \bigcup_{j=1}^N \{x:|x-x_j|<(1-\tfrac\eps2)\tfrac{t_0}2\}.
$$
Therefore,
\begin{align*}
&\{(t,x):t_0 \leq t \leq 2t_0, \, |x|<t\}\\
&\qquad  \subset \bigcup_{j=1}^N \{(t,x) :t_0 \leq t \leq 2t_0, \, |x-x_j|<(1-\tfrac\eps2)\tfrac{t_0}2 + (t-t_0)\}\\
&\qquad \subset \bigcup_{j=1}^N \{(t,x):t_0 \leq t \leq 2t_0, \, |x-x_j| < (1-\tfrac\eps{8})(t-\tfrac{t_0}2)\}.
\end{align*}
By assumption, $u$ is defined on each light cone
$$
\{(t,x):\tfrac12t_0<t \leq T, \, |x-x_j| < t-\tfrac{t_0}2 \},
$$
so, by \eqref{E:cone bound superc},
\begin{align*}
&\int_{t_0}^{2t_0} \int_{|x-x_j|<(1-\frac{\eps}8)(t-\frac{t_0}2)}|\nabla_{t,x}u(t,x)|^2\, dx\, dt \\
&\qquad \lesssim_\eps \int_{\frac{t_0}2}^{T} \int_{|x-x_j|<t-\frac{t_0}2} \bigl(1-\tfrac{|x-x_j|}{t-\tfrac{t_0}2}\bigr)^2|\nabla_{t,x}u(t,x)|^2\, dx\, dt
\lesssim_\eps 1.
\end{align*}
Summing this inequality over $1 \leq j \leq N$, we derive the claim.

\end{proof}

%
%
%
%

\section{Bounds in light cones:  The conformal and sub-conformal cases}\label{S:conf blow}

The goal of this section is to give pointwise in time upper bounds on the blowup rate of solutions to \eqref{E:eqn} in the conformal and sub-conformal cases,
that is, when $0<s_c\leq \frac12$.

\begin{theorem} \label{T:cone bound subc}  Let $d \geq 2$, $m \in [0,1]$, $0 < s_c \leq \tfrac12$, and $p=\frac4{d-2s_c}$.  If $u$ is a strong solution to \eqref{E:eqn} in the light cone $\{ (t,x) : 0<t\leq T \text{ and } |x|<t \}$, then
\begin{equation} \label{E:cone bound subc}
\int_{|x|<t/2} t^{-2s_c}|u(t,x)|^2 + t^{2(1-s_c)}|\nabla_{t,x}u(t,x)|^2\, dx \lesssim 1.
\end{equation}
The implicit constant depends on $d, s_c, T,$ and the $H^1_x\times L_x^2$ norm of $(u(T),u_t(T))$ on the ball $\{|x|<T\}$.
\end{theorem}

For the nonlinear wave equation, that is, \eqref{E:eqn} with $m=0$, this theorem was proved by Merle and Zaag in \cite{MerleZaagIMRN}, building on earlier work
\cite{MerleZaagAJM,MerleZaagMA} that considered solutions defined in a spacetime slab.  This result describes the behaviour of solutions near a general
blowup surface $t=\sigma(x)$, as defined in the Introduction.  In particular, in the case of a non-characteristic point, a simple covering argument yields
\eqref{E:cone bound subc} with integration over the larger region $|x|<t$.

The arguments of Merle and Zaag adapt \emph{mutis mutandis} to the Klein--Gordon equation \eqref{E:eqn}, since the mass term always appears with the helpful sign.  However, our Lemma~\ref{L:dilat id} and Corollary~\ref{C:Z id} allow us to streamline the arguments of \cite{MerleZaagAJM}, \cite{MerleZaagMA}, and \cite{MerleZaagIMRN}.  We focus first on the conformal case; the discussion of the sub-conformal case can be found at the end of this section.

In the conformal case, our argument relies on \eqref{E:L flux ineq 2}, which gives control over all directional derivatives of the solution; this should be compared with Proposition~2.4 in \cite{MerleZaagMA} and Proposition~4.2 in \cite{MerleZaagIMRN}, which only provide control over a subset of directional derivatives.  An immediate consequence of \eqref{E:L flux ineq 2} is
\begin{align}\label{E:deriv on ball}
\int_{t_0}^{2t_0}\!\!\! \int_{|x|<t-\frac1{10}t_0} |\nabla_{t,x} u|^2 + t_0^{-2} |u|^2 +  t_0^{-2}\bigl|x\cdot\nabla u +tu_t +\tfrac{d-1}2 u\bigr|^2\, dx\, dt \lesssim 1,
\end{align}
uniformly in $0<t_0\leq \frac12 T$.   Applying \eqref{E:L flux ineq 2} to a spacetime translate of our solution yields
\begin{align}\label{E:deriv on alpha ball}
\int_{t_0}^{(1+\alpha) t_0}\!\!\! \int_{|x-x_0|<\alpha t} \bigl|(x-x_0)\cdot\nabla u +tu_t +\tfrac{d-1}2 u\bigr|^2\, dx\, dt \lesssim \alpha t_0^2,
\end{align}
uniformly for $0<\alpha\leq \frac1{10}$, $|x_0|<\frac45 t_0$, and $[t_0,(1+\alpha)t_0]\subseteq (0,T]$.  For comparison, see the proof of Proposition~3.1 in
\cite{MerleZaagMA} and Proposition~4.2 in \cite{MerleZaagIMRN}.

Next we transfer the estimate \eqref{E:deriv on ball} to a bound on the potential energy.  To do this, we will employ a translated version of the functional
$$
L(t) = \int_{|x|<t} \mcl^0(t,x)\, dx,
$$
introduced in Section~\ref{S:Lyapunov}; recall that $\mcl^0$ is defined as
$$
\mcl^0 = \tfrac1{2t}|x\cdot \nabla u + tu_t + \tfrac{d-1}2u|^2 + \tfrac{t}2(|\nabla u|^2 - |\tfrac{x}t \cdot \nabla u|^2) - \tfrac{(d-1)t}{2(d+1)}|u|^{\frac{2(d+1)}{d-1}} + \tfrac{d^2-1}{8t}u^2 + t\tfrac{m^2}2u^2.
$$
By Lemmas~\ref{L:dilat id}~and~\ref{L:L>0}, $L$ is a nonnegative increasing function of time.

Using this functional adapted to the cone $\{|x|<t-\frac1{10}t_0\}$, specifically the fact that  $L\geq 0$, we deduce
\begin{align}\label{E:pot on ball}
\int_{t_0}^{2t_0}\!\!\! \int_{|x|<t-\frac1{10}t_0} |u(t,x)|^{\frac{2(d+1)}{d-1}}\, dx\, dt \lesssim \text{LHS\eqref{E:deriv on ball}} \lesssim 1,
\end{align}
uniformly in $0<t_0\leq \frac12 T$.

To prove Theorem~\ref{T:cone bound subc}, we need to upgrade the averaged in time estimates obtained above to pointwise in time estimates.  The first step is the following result:

\begin{lemma}[Pointwise in time estimates on the mass and a critical norm] \label{L:ptwise mass subc}\leavevmode\\
Let $0<t_0\leq \frac12 T$.  For $t \in [t_0, 2t_0]$ we have
\begin{equation} \label{E:ptwise mass subc}
\int_{|x|<t-\frac1{10}t_0} |u(t,x)|^2\, dx \lesssim t_0
\end{equation}
and
\begin{equation} \label{E:ptwise q subc}
\int_{|x-x_0|<\alpha t} |u(t,x)|^{\frac{2d}{d-1}}\, dx \lesssim \alpha^{1/2},
\end{equation}
whenever $0<\alpha\leq \frac1{10}$ and $|x_0|<\frac45 t_0$.
\end{lemma}

\begin{proof}
The proof follows the argument used to establish \eqref{E:mass bound superc}.  To derive \eqref{E:ptwise mass subc}, one uses the function $f(\lambda):= \int_{|x|<t_0} |\lambda^{\frac{d-1}2} u(\lambda t_0,\lambda x)|^2\, dx$, while for \eqref{E:ptwise q subc} one uses the original version of $f$ with $p=\frac{4}{d-1}$.  We need two ingredients:  The first ingredient is an integral bound over the directional derivatives of $u$ on the appropriate cone; the role of the first ingredient in the current setting is played by \eqref{E:deriv on ball} and \eqref{E:deriv on alpha ball}.  The second ingredient we need is averaged in time estimates for the left-hand sides of \eqref{E:ptwise mass subc} and \eqref{E:ptwise q subc}; the role of the second ingredient is played by \eqref{E:deriv on ball} and
\begin{align*}
\int_{t_0}^{(1+\alpha)t_0} \!\!\! &\int_{|x-x_0|<\alpha t} |u(t,x)|^{\frac{2d}{d-1}}\, dx\, dt\\
&\lesssim t_0\alpha^{\frac{d}{d+1}}\biggl[\int_{t_0}^{(1+\alpha)t_0}\!\!\! \int_{|x-x_0|<\alpha t}|u(t,x)|^{\frac{2(d+1)}{d-1}}\, dx\, dt\biggr]^{\frac{d}{d+1}}
\lesssim t_0\alpha^{\frac{d}{d+1}},
\end{align*}
which follows from H\"older and \eqref{E:pot on ball}.
\end{proof}

The simple argument just used does not allow us to upgrade our integrated gradient or potential energy estimates to versions that are pointwise in time.
We will instead employ a bootstrap argument close to that in the work of Merle and Zaag.  The requisite smallness is provided by
\eqref{E:ptwise q subc} by choosing  $\alpha$ small enough.  Combining this estimate with the Gagliardo--Nirenberg inequality gives
\begin{align}
\int_{|x-x_0|<r} & |u(t,x)|^{\frac{2(d+1)}{d-1}}\, dx \notag\\
&\lesssim\biggl[\int_{|x-x_0|<r} |u(t,x)|^{\frac{2d}{d-1}}\, dx\biggr]^{\frac2d} \int_{|x-x_0|<r} |\nabla u(t,x)|^2 + \tfrac1{r^2}  |u(t,x)|^2\, dx \notag\\
&\lesssim \alpha^{1/d}\int_{|x-x_0|<r} |\nabla u(t,x)|^2 + \tfrac1{r^2}  |u(t,x)|^2\, dx, \label{E:nonconc}
\end{align}
uniformly for $0<\alpha\leq \frac1{10}$, $r<\alpha t$, and $|x_0|<\frac45 t$.

To obtain an inequality in the opposite direction, we use boundedness of the functional $L$ adapted to the cone $\{(s,y) : |y-x_0|<r+s-t\}$ together with
the observation
\begin{align*}
\bigl(1-\tfrac{|x|^2}{t^2}\bigr) \bigl|\nabla_{t,x} u|^2
&\lesssim t^{-2}|x\cdot \nabla u + tu_t + \tfrac{d-1}2u|^2 + (|\nabla u|^2 - |\tfrac{x}t \cdot \nabla u|^2) + t^{-2}u^2\\
&\lesssim t^{-1} \mcl^0 + |u|^{\frac{2(d+1)}{d-1}}.
\end{align*}
This gives
\begin{align}\label{E:grad from pot}
\int_{|x-x_0|<r}\bigl( 1-\tfrac{|x-x_0|^2}{r^2} \bigr) |\nabla_{t,x} u(t,x)|^2\, dx \lesssim   \tfrac1r + \int_{|x-x_0|<r} |u(t,x)|^{\frac{2(d+1)}{d-1}}\,dx,
\end{align}
where the coefficient of $1/r$ depends on the norm of $(u(T),u_t(T))$ via the monotonicity of $L$.

Combining \eqref{E:nonconc} and \eqref{E:grad from pot} yields
\begin{align}\label{almost}
\int_{|x-x_0|<\frac12 r} |\nabla u(t,x)|^2\, dx &\lesssim \tfrac1r +  \alpha^{1/d}\int_{|x-x_0|<r} |\nabla u(t,x)|^2 + \tfrac1{r^2}  |u(t,x)|^2\, dx,
\end{align}
which is not immediately amenable to bootstrap because the two regions of integration are different.  To remedy this, we set $R=\frac35t$ and $r= \frac\alpha3(R-|x_0|)$
and apply the following averaging operator to both sides:
$$
f(x_0) \mapsto \frac1{R^{d+2}} \int_{|x_0|<R}  (R-|x_0|)^2 f(x_0) \,dx_0.
$$
We note that 
$$
\bigl\{ (x_0,x) : |x_0|\!<\!R,\ |x-x_0| \!<\! \tfrac\alpha3(R-|x_0|) \bigr\} \subseteq \bigl\{ (x_0,x) : |x|\!<\!R,\ |x_0-x| \!<\! \tfrac\alpha2(R-|x|) \bigr\}
$$
and
$$
\bigl\{ (x_0,x) : |x_0|\!<\!R,\  |x_0-x|\! <\! \tfrac\alpha6(R-|x_0|) \bigr\} \supseteq \bigl\{ (x_0,x) : |x|\!<\!R,\ |x-x_0| \!<\! \tfrac\alpha7(R-|x|) \bigr\},
$$
and that $R-|x|\sim R-|x_0|$ on any of these sets.  Using Fubini, we deduce
\begin{align*}
\int_{|x|<R} \bigl( 1 - \tfrac{|x|}R\bigr)^{d+2}  &|\nabla_{t,x} u(t,x)|^2 \, dx\\
&\lesssim (\alpha R)^{-1} +\alpha^{1/d}\int_{|x|<R}\bigl( 1 - \tfrac{|x|}R\bigr)^{d+2} |\nabla u(t,x)|^2  \, dx \\
&\quad + (\alpha R)^{-2}\alpha^{1/d}\int_{|x|<R} \bigl( 1 - \tfrac{|x|}R\bigr)^{d} |u(t,x)|^2\, dx.
\end{align*}
Choosing $\alpha$ sufficiently small and recalling $R=\frac35t$ and \eqref{E:ptwise mass subc}, we obtain
\begin{align}\label{tada}
\int_{|x|<\frac35t} \bigl( 1 - \tfrac{5|x|}{3t}\bigr)^{d+2}  |\nabla_{t,x} u(t,x)|^2 \, dx \lesssim t^{-1},
\end{align}
which yields the requisite bound on the spacetime gradient of $u$.  To finish the proof of \eqref{E:cone bound subc}, we merely note
that the bound on the $L^2_x$ norm was obtained already in Lemma~\ref{L:ptwise mass subc}.  This completes the proof of Theorem~\ref{T:cone bound subc}
in the conformal (i.e., $s_c=1/2$) case.

Our argument for the sub-conformal case is similar, but slightly simpler, with \eqref{E:smudge Z} taking over the role played above
by \eqref{E:L flux ineq 2}.  In particular we have the following analogue of \eqref{E:deriv on ball}
\begin{align}\label{E:D ball}
\int_{t_0}^{2t_0}\!\!\! \int_{|x|<t-\frac1{10}t_0} |\nabla_{t,x} u|^2 + t_0^{-2} |u|^2 +  t_0^{-2}\bigl|x\cdot\nabla u +tu_t +\tfrac2p u\bigr|^2\, dx\, dt \lesssim t_0^{2s_c -1}.
\end{align}
From this and the fact that $Z\geq 0$, we deduce
\begin{align}\label{E:P ball}
\int_{t_0}^{2t_0}\!\!\! \int_{|x|<t-\frac1{10}t_0} |u(t,x)|^{p+2}\, dx\, dt \lesssim t_0^{2s_c -1}.
\end{align}

Using the same argument as in Lemma~\ref{L:ptwise mass subc} and~\eqref{E:mass bound superc}, modifying $f(\lambda)$ as needed, we obtain
\begin{equation}\label{E:wk pt}
\int_{|x|<\frac9{10} t} |u(t,x)|^2\, dx \lesssim t^{2 s_c}
\qtq{and}
\int_{|x|<\frac9{10} t} |u(t,x)|^{\frac{p+4}2}\, dx \lesssim t^{2s_c -1}.
\end{equation}

Let $\gamma:=\frac 12-s_c$.  Using the boundedness of the functional $Z$ associated to the cone with apex $(t-r,x_0)$, we obtain
\begin{align}\label{p4g}
\int_{|x-x_0|<r} \!\bigl(1 - \tfrac{|x-x_0|^2}{r^2}\bigr)^{\gamma + 1}|\nabla_{t,x} u|^2 \, dx
\lesssim r^{2s_c-2} + \int_{|x-x_0|<r} \!\bigl(1 - \tfrac{|x-x_0|^2}{r^2}\bigr)^{\gamma}|u|^{p+2} \, dx,
\end{align}
for all $|x_0|<\frac45 t$ and $0<r < \tfrac1{10} t$.  This plays the role of \eqref{E:grad from pot}.

In order to obtain an upper bound on the potential energy we use the Gagliardo--Nirenberg inequality:
\begin{align*}
\int_{|x-x_0|<r}  &  \bigl(1 - \tfrac{|x-x_0|^2}{r^2}\bigr)^{\gamma}|u|^{p+2} \, dx\\
&\lesssim \biggl[ \int_{|x-x_0|<r} |u|^{\frac{p+4}2} \, dx\biggr]^{\frac{8-2p(d-2)}{8-p(d-2)}} \biggl[ \int_{|x-x_0|<r}  |\nabla u|^{2} + r^{-2} |u|^2 \, dx\biggr]^{\frac{pd}{8-p(d-2)}}.
\end{align*}
Incorporating \eqref{E:wk pt} we deduce
\begin{align*}
\int_{|x-x_0|<r} \bigl(1 - \tfrac{|x-x_0|^2}{r^2}\bigr)^{\gamma} &|u|^{p+2} \, dx \\
&\lesssim t^{(2s_c-1)\frac{8-2p(d-2)}{8-p(d-2)}}\biggl[ r^{-2} t^{2s_c} + \int_{|x-x_0|<r} |\nabla u|^{2} \, dx\biggr]^{\frac{pd}{8-p(d-2)}}.
\end{align*}
Combining this estimate with \eqref{p4g} yields
\begin{align*}
\int_{|x-x_0|<r/2} |\nabla_{t,x} u|^2 \, dx
\lesssim r^{2s_c-2} + t^{2s_c-2}
    \biggl[ r^{-2} t^{2} + t^{2-2s_c} \!\!\int_{|x-x_0|<r} |\nabla u|^{2} \, dx\biggr]^{\frac{pd}{8-p(d-2)}}.
\end{align*}

The basic bootstrap relation follows by setting $r=\tfrac13(\frac45 t-|x_0|)$ and applying the averaging operator
$$
f(x_0) \mapsto \frac1{t^{d+2}} \int_{|x_0|< \frac45 t} f(x_0)  (\tfrac{4t}5 -|x_0| )^2 \,dx_0
$$
to both sides.  A little patience and Jensen's inequality then yield
\begin{align*}
\int_{|x|<\frac45 t} \bigl(1 - \tfrac{5|x|^2}{4t^2}\bigr)^{d + 2}& |\nabla_{t,x} u|^2 \, dx\\
&\lesssim t^{2s_c-2}  \biggl[ 1+ t^{2-2s_c}\int_{|x|<\frac45 t}
            \bigl(1 - \tfrac{5|x|^2}{4t^2}\bigr)^{d + 2} |\nabla u|^{2} \, dx\biggr]^{\frac{pd}{8-p(d-2)}},
\end{align*}
in much the same manner as in the conformal case.  Noting that the last power here is smaller than one, this inequality yields
$$
\int_{|x|<\frac45 t} \bigl(1 - \tfrac{5|x|^2}{4t^2}\bigr)^{d + 2} |\nabla_{t,x} u(t,x)|^2 \, dx \lesssim t^{2s_c-2}.
$$
This immediately implies the estimate on the spacetime gradient stated in Theorem~\ref{T:cone bound subc} in the sub-conformal case.  The stated estimate on the $L^2_x$-norm was given in \eqref{E:wk pt}.

This completes the proof of Theorem~\ref{T:cone bound subc}.


\end{document}